\documentclass[aop,preprint]{imsart}

\RequirePackage[OT1]{fontenc}
\RequirePackage{amsthm,amsmath}
\RequirePackage[numbers]{natbib}
\RequirePackage[colorlinks,citecolor=blue,urlcolor=blue]{hyperref}

\arxiv{arXiv:1606.03988}

\usepackage[bottom=0.4cm]{geometry}
\usepackage{amsfonts}
\usepackage{color}
\usepackage{xcolor}
\usepackage{amssymb}
\usepackage{graphicx}
\usepackage{soul}
\usepackage{enumitem}
\usepackage{footmisc}

\usepackage{times}

\startlocaldefs

\hypersetup{
	bookmarksnumbered,
	pdfstartview={FitH},
	citecolor={blue},
	linkcolor={blue},
	urlcolor=[rgb]{0,0.55,0},
	pdftitle={Limit\ theory\ for\ geometric\ statistics\ of\ point\ processes\ having\ fast\ decay\ of\ correlations},
	pdfauthor={B.\ Blaszczyszyn\ -\ D.\ Yogeshwaran\ -\ J.\ E.\ Yukich}
}

\newcommand{\tm}{\tilde{m}}
\newcommand{\tM}{\tilde{M}}
\newcommand{\tA}{\tilde{A}}
\newcommand{\tB}{\tilde{B}}

\newcommand{\EXCLUDE}[1]{}

\newcommand{\be}{\begin{equation}}
\newcommand{\ee}{\end{equation}}

\newcommand{\bea}{\begin{eqnarray}}
\newcommand{\no}{\nonumber}

\newcommand{\eea}{\end{eqnarray}}
\newcommand\bmodif{\begin{modif}}
	\newcommand\emodif{\end{modif}}

\renewcommand{\qed}{\hfill $\Box$}
\newcommand{\sP}{\mathbb{P}}  
\newcommand{\sE}{\mathbb {E} } 

\newcommand{\E}{\mathbb {E} }
\newcommand{\B}{\mathbb {B} }

\newcommand{\EXP}[1]{\mathbb {E}\!\left(#1\right) }

\newcommand{\Var}{{\rm Var}}

\newcommand{\Vol}{{\rm Vol}}
\newcommand{\1}[1]{\mathsf{1}\!\left[\,#1\,\right] }
\newcommand{\remove}[1]{}
\newcommand{\tod}{\stackrel{{\cal D}}{\longrightarrow}}
\newcommand{\psip}{\psi^!}

\newtheorem{theorem}{Theorem}[section]
\newtheorem{corollary}[theorem]{Corollary}
\newtheorem{lemma}[theorem]{Lemma}
\newtheorem{proposition}[theorem]{Proposition}
\newtheorem{remark}[theorem]{Remark}

\newtheorem{definition}[theorem]{Definition}

\newtheorem{example}[theorem]{ Example}

\numberwithin{equation}{section}

\newcommand{\limn}{\lim_{n \to \infty} }
\def\0{{\bf 0}}

\newcommand{\tps}{ \tilde{\psi}}
\newcommand{\tpsp}{ \tilde{\psi}^!}
\newcommand{\tx}{ \tilde{\xi}}

\newcommand{\lam}{\lambda}
\newcommand{\Lam}{\Lambda}

\def\T{\top}
\def\Del{\Delta}

\def\N{\mathbb{N}}

\def\mR{\mathbb{R}}
\DeclareRobustCommand{\stirling}{\genfrac\{\}{0pt}{}}

\def\mP{\mathbb{P}}

\def\mE{\mathbb {E} \,}

\def\X{{\cal X }}
\def\Y{{\cal Y }}
\def\x{{\bf x}}
\def\a{{\bf a}}
\def\y{{\bf y}}
\def\z{{\bf z}}

\newcommand{\md}{\mathrm{d}}

\newcommand{\cD}{{\mathcal D}}

\newcommand{\cX}{{\mathcal X}}
\newcommand{\cA}{{\mathcal A}}
\newcommand{\cN}{{\mathcal N}}
\newcommand{\cK}{{\mathcal K}}
\newcommand{\cY}{{\mathcal Y}}
\newcommand{\cL}{\mathcal{L}}

\def\R{\mathbb{R}}
\def\N{\mathbb{N}}

\def\P{\mathcal P}
\def\A{\mathcal A}

\def\1{\mathbf{1}}

\def\H{\widehat{H}}

\def\Cech{\v{C}ech\ }
\def\C{{\cal C}}
\def\cCB{{\cal C}_B}

\def\Comment#1{\remove{#1}}

\endlocaldefs

\begin{document}

\begin{frontmatter}
\title{Limit theory for geometric statistics of point processes having fast decay of correlations}
\runtitle{Limit theory for geometric statistics of point processes}
\begin{aug}
\author{\fnms{B.} \snm{B{\l}aszczyszyn}\ead[label=e1]{Bartek.Blaszczyszyn@ens.fr},}
\author{\fnms{D.}  \snm{Yogeshwaran} \thanksref{t2}\ead[label=e2]{d.yogesh@isibang.ac.in}}
\and
\author{\fnms{J. E.}   \snm{Yukich} \corref{} \thanksref{t3}\ead[label=e3]{joseph.yukich@lehigh.edu}}
\thankstext{t2}{Research supported by DST-INSPIRE faculty award, CPDA from the Indian Statistical Institute and TOPOSYS grant.}
\thankstext{t3}{Research supported in part by NSF grant DMS-1406410}
\runauthor{B{\l}aszczyszyn, Yogeshwaran and Yukich.}
\affiliation{Inria/ENS, Paris; Indian Statistical Institute, Bangalore and Lehigh University, Bethlehem}
\address{Inria/ENS, 2 rue Simone Iff \\
			CS 42112 75589 Paris Cedex 12. \\
			France.\\
			\printead{e1} }
\address{Statistics and  Mathematics Unit\\
		Indian Statistical Institute\\
		Bangalore - 560059\\
\printead{e2}}

\address{Department of Mathematics\\
		Lehigh University\\
		Bethlehem PA 18015\\
\printead{e3}}
\end{aug}

\renewcommand{\baselinestretch}{1.1}\normalsize			
	\begin{abstract}
		Let $\P$ be  a simple, stationary point process on $\R^d$ having fast decay of correlations, i.e., its correlation functions factorize up to an additive error decaying  faster than any power of the separation distance.
		Let $\P_n:= \P \cap W_n$ be its restriction
		to windows $W_n:= [- {1 \over 2} n^{1/d}, {1 \over 2}  n^{1/d}]^d \subset  \R^d$. We consider the statistic $H_n^\xi:= \sum_{x \in \P_n} \xi(x, \P_n)$ where  $\xi(x, \P_n)$ denotes a score function representing the interaction of $x$ with respect to
		$\P_n$.  When $\xi$ depends on local data in the sense that its radius of stabilization has an exponential tail, we
		establish expectation asymptotics, variance asymptotics, and central limit theorems
		for $H_n^\xi$ and, more generally, for statistics of the  re-scaled, possibly signed,  $\xi$-weighted point  measures
		$\mu_n^{\xi} := \sum_{x \in \P_n} \xi(x, \P_n) \delta_{n^{-1/d} x}$,
		as  $W_n \uparrow  \R^d$.
This gives the limit theory for non-linear geometric statistics (such as clique counts,  the number of Morse critical points, intrinsic volumes of the Boolean model, and total edge length of the $k$-nearest neighbors graph) of  $\alpha$-determinantal point processes (for $-1/\alpha \in \N$) having fast decreasing kernels, including the $\beta$-Ginibre ensembles, extending the Gaussian fluctuation results of Soshnikov \cite{So} to non-linear statistics. It also gives the limit theory for geometric U-statistics of $\alpha$-permanental point processes (for $1/\alpha \in \N$) as well as the zero set of Gaussian entire functions, extending the central limit theorems of Nazarov and Sodin \cite{Nazarov12} and  Shirai and Takahashi \cite{Shirai03}, which are also confined to linear statistics. The proof of the central limit theorem relies on a factorial moment expansion originating in \cite{Blaszczyszyn95, Bartek97} to show the fast decay of the correlations of $\xi$-weighted point measures.
The latter property is shown to imply a condition equivalent to Brillinger mixing and consequently yields the asymptotic normality of $\mu_n^\xi$ via  an extension of the cumulant method.
				\end{abstract}
		


\end{frontmatter}
	
	\renewcommand{\baselinestretch}{0.8}\normalsize
	{\hypersetup{linkbordercolor=black}
	\tableofcontents
     }
	\renewcommand{\baselinestretch}{1.1}\normalsize	
\setlength{\abovedisplayskip}{7pt}
\setlength{\belowdisplayskip}{7pt}

\setlength{\abovedisplayshortskip}{3.5pt}
\setlength{\belowdisplayshortskip}{3.5pt}

\section{\bf Introduction and main results}
	
Functionals of  geometric structures on finite point sets $\X \subset \R^d$ often consist of sums of spatially dependent terms admitting the representation
\be \label{basic}\sum_{x \in \X} \xi(x,\X), \ee  where the $\mR$-valued {\em score function} $\xi$, defined on pairs $(x, \X)$,  $x \in \X$, represents the {\em interaction}
of $x$ with respect to $\X$, called the {\em input}.  The sums \eqref{basic} typically describe a global geometric feature of a structure on $\X$  in terms of local contributions $\xi(x,\X)$.

It is frequently the case in stochastic geometry, statistical physics, and spatial statistics that one seeks the large $n$ limit behavior of
\begin{equation}
\label{eqn:total_mass}
H_n^\xi:= H_n^\xi(\P):= \sum_{x \in \P_n} \xi(x, \P_n)
\end{equation}
where $\xi$ is an appropriately chosen score function,  $\P$ is  a simple, stationary point process on $\R^d$, and $\P_n$ is the restriction of $\P$ to  $W_n:= [- {1 \over 2} n^{1/d}, {1 \over 2}  n^{1/d}]^d$. For example if $\P_n$ is either a Poisson or binomial point process and if $\xi$ is either a local $U$-statistic or  an exponentially stabilizing score function, then the limit theory for $H_n^\xi$ is established in \cite{Baryshnikov05, EicTha, Lachieze15, LPS, PeEJP, PY1, PY5, Reitzner13}.  If $\P_n$ is a rarified Gibbs point process on $W_n$ and $\xi$ is exponentially stabilizing, then \cite{SY, XY} treat the limit theory for  $H_n^\xi$.

It is natural to ask whether  the limit theory of these papers extends to more general input satisfying a notion of `asymptotic independence' for point processes. Recall that if $\xi \equiv 1$ and if $\P$ is an $\alpha$-determinantal point process  with $\alpha = -1/m$ or an $\alpha$-permanental point process  with $\alpha = 2/m$ for some $m$ in the set of positive integers $\N$ (respectively $\P$ is the  zero set of a Gaussian entire function), then  remarkable results of Soshnikov \cite{So}, Shirai and Takahashi \cite{Shirai03} (respectively  Nazarov and Sodin \cite{Nazarov12}), show that the counting statistic $\P_n(W_n):= \sum_{x \in \P_n} {\bf 1}[x \in W_n]$ is asymptotically normal. One may ask whether asymptotic normality of  $H_n^\xi$ still holds when $\xi$ is
either a local $U$-statistic or  an exponentially stabilizing score function.  We answer these questions affirmatively.  Loosely speaking, subject to a mild growth condition on $\Var H_n^\xi$,  our approach shows that  $H_n^\xi$ is asymptotically normal whenever  $\P$ is a point process having fast decay of correlations.

Heuristically, when the score functions depend on `local data' and when the input is `asymptotically independent', one might expect that the statistics $H_n^\xi$ obey a strong law and a central limit theorem. The notion of dependency on  `local data'  for score functions is  formalized via stabilization in \cite{Baryshnikov05, EicTha,  PeEJP, PY1, PY5}. 
Here we formalize the idea of asymptotically independent input $\P$  via the notion of `fast decay of correlation functions'.
We thereby extend the limit theory of the afore-mentioned papers to input having fast decay of correlation functions.
A point process $\P$ on $\R^d$ has fast decay of correlations if for all $p, q \in \N$  and all
	$x_1,\ldots,x_{p + q} \in \R^d$, its correlation functions
	$\rho^{(p + q)}(x_1,\ldots,x_{p + q})$ factorize into $\rho^{(p)}(x_1,\ldots,x_{p})\rho^{(q)}(x_{p + 1},\ldots,x_{p + q})$
	up to an additive error decaying faster than any power of the separation distance
	\be \label{defs}
	s := d(\{x_1,\ldots ,x_p\},\{x_{p+1},\ldots ,x_{p+q}\}):= \inf_{i \in \{1,...,p\}, j \in \{p + 1,...,p + q\} } |x_i - x_j|
	\ee
	as at \eqref{eqn:clustering_condition} below. Roughly speaking, such point processes exhibit asymptotic independence at large distances. Examples of such point processes are given in Section \ref{sec:ex_pp}.
Point processes with fast decay of correlations are called `clustering point processes' in statistical physics  \cite{Malyshev75,Martin80,Nazarov12}.  We shall avoid this terminology since, at least from the point of view of spatial statistics, it suggests that the points of $\P$ clump or aggregate together, which need not be the case.

	
  If  either $\P$ has fast decay of correlations and $\xi$ is a local $U$-statistic or if $\P$ has exponentially fast decay of correlations and $\xi$ is an exponentially stabilizing score function, then with $\delta_x$ denoting the point mass at $x$,
   our main  results  establish expectation and variance asymptotics as $n \to \infty$, as well as central limit theorems for the re-scaled, possibly signed, $\xi$-weighted point  measures
	\begin{equation}
	\label{eqn:rand_meas}
	\mu_n^\xi:=  \sum_{x \in \P_n} \xi(x, \P_n) \delta_{n^{-1/d}x},
	\end{equation}
	thereby also establishing the limit theory for  the total mass of $\mu_n^\xi$ given by the non-linear statistics  $H_n^\xi$.

As shown in  Theorems \ref{prop:variance_upper_bound}-\ref{prop:variance_lower_bound} this yields the limit theory for	general non-linear statistics of $\alpha$-determinantal and $\alpha$-permanental point processes, the point process given by the zero set of a Gaussian entire function, as well as rarified Gibbsian input.

The benefit of the general approach taken here is four-fold:
(i) we establish the  asymptotic normality of the random measures $\mu_n^\xi$, with $\P$ either an $\alpha$-permanental point process (with $1/\alpha \in \N$), an $\alpha$-determinantal point process (with $-1/\alpha \in \N$), or the zero set of a Gaussian entire function, thereby extending the work  of  Soshnikov \cite{So}, Shirai and Takahashi \cite{Shirai03}, and Nazarov and Sodin \cite{Nazarov12},  who restrict to linear statistics,  (ii)  we extend the limit theory of \cite{Baryshnikov05, LPS, PeEJP, Pe, PY1, PY5}, which is confined to Poisson and binomial input, to point processes having fast decay of correlations, (iii)  we apply our general results to deduce asymptotic normality  and variance asymptotics for geometric statistics of input having fast decay of correlations, including statistics of simplicial complexes and germ-grain models, clique counts, Morse critical points, as well of statistics of random graphs,
(iv) our general proof of the asymptotic normality of ~$\mu_n^{\xi}$ relates the  fast decay of correlations of the input process $\P$
to a similar fast correlation decay for the family of {\em $\xi$-weighted (point)  measures}
\be \label{xwm}
\sum_{x \in \P_n} \xi(x, \P_n) \delta_{x},
\ee
consequently implying Brillinger mixing of these measures, and thus directly relating the two concepts:  Fast decay of correlations implies Brillinger mixing.

Given input $\P$ having fast decay of correlations, an interesting feature of the measures $\mu_n^\xi$ is that  their variances are at most of order $\Vol_d(W_n)$, the volume of the window $W_n$ (Theorem \ref{prop:variance_upper_bound}). This holds also for the statistic $\hat{H}_n^{\xi} := \sum_{x \in \P_n}\xi(x,\P)$, which involves summands having no boundary effects. An interesting feature of this statistic is that if its variance is $o(\Vol_d(W_n))$ then it has to be $O(\Vol_{d-1}(\partial W_n))$, where $\partial W_n$ denotes the boundary of $W_n$  and $\Vol_{d-1}(.)$ stands for the $(d-1)$th intrinsic volume (Theorem \ref{prop:variance_lower_bound}). In other words, if the fluctuations of $\hat{H}_n^\xi$ are not of volume order, then they are at most of surface order.


	
Coming back to our set-up, when a functional $H_n^\xi(\P)$ is expressible as a sum of local $U$-statistics or,  more generally,  as a sum of exponentially stabilizing score functions $\xi$, then a key step towards proving the central limit theorem is to show that the correlation functions of the $\xi$-weighted measures defined via Palm expectations
		$\sE_{x_1,\ldots,x_k}$ (cf Section~\ref{sec:prelim}) and given by
	\begin{equation}
	m^{(k_1,...,k_{p + q})} (x_1,\ldots,x_{p + q} ; n) :=\sE_{x_1,\ldots,x_{p +
			q}}(\xi(x_1,\P_n)^{k_1} \ldots\xi(x_{p + q},\P_n)^{k_{p + q} })\rho^{(p +
		q)}(x_1,\ldots,x_{p + q}) \label{eqn:mixedmomentn},
	\end{equation}
	similar to those of the input process~$\P$,  approximately factorize into
$$m^{(k_1,...,k_{p})}(x_1,\ldots,x_{p} ; n)  m^{(k_{p +1},...,k_{p + q})}(x_{p + 1},\ldots,x_{p + q} ; n),$$
uniformly in n~$\le\infty$,
up to an additive error decaying faster than any power of the separation distance $s$, defined at \eqref{defs}.
Here $x_1,...,x_{p + q}$ are distinct points in $W_n$ and $k_1,...,k_{p + q}\in\N$. This result, spelled out in Theorem \ref{prop:clustering_gen},	is at the heart of our approach.
We then give two proofs of the central limit theorem (Theorem \ref{thm:main0}) for the purely atomic random measures \eqref{eqn:rand_meas} via the cumulant method and as a corollary derive the asymptotic normality of $H_n^\xi(\P)$ and $\int\!\!f\md\mu_n^\xi$, $f$ a test function, as $n \to \infty$. The proof of expectation and variance asymptotics (Theorem \ref{prop:variance_upper_bound}) mainly relies upon the refined Campbell theorem.

In contrast to the afore-mentioned works, our  proof of the fast decay of correlations of the $\xi$-weighted measures depends heavily on a factorial moment expansion for expected values of functionals of a general point process	$\P$. This expansion, which originates in \cite{Blaszczyszyn95, Bartek97}, is expressed in terms of iterated difference operators 
	of the considered functional on the null configuration of points and integrated against factorial moment measures of the point process. It is valid for  general point processes, in contrast to the  Fock space representation of  Poisson functionals, which involves the same  difference operators but is deeply related to chaos expansions~\cite{LP-Fock}. Further connections with the literature are discussed in the remarks following  Theorems \ref{thm:main1} and \ref{prop:variance_lower_bound}.

Our interest in these issues was stimulated by similarities in the methods of \cite{Malyshev75}, \cite{BY1, Baryshnikov05, SY} and \cite{Nazarov12}. The  articles \cite{Baryshnikov05, SY} prove central limit theorems for stabilizing functionals of Poisson and rarified Gibbsian point processes, respectively, while \cite{Nazarov12} proves central limit theorems for linear statistics $\sum_{x \in \P_n} \xi(x)$ of point processes having fast decay of correlations.
These papers all establish the fast decay of correlations of the  $\xi$-weighted measures as at
 ~\eqref{e.clustering-generalized-mixed-monents} below, and then use the resulting volume order cumulant bounds to show asymptotic normality.
This paper unifies and extends  the results  of \cite{BY1,Baryshnikov05, Nazarov12, SY, Shirai03} to  input having fast decay of correlations.   The idea of using correlation functions to show asymptotic normality via cumulants goes back to \cite{Malyshev75}.	The earlier work of \cite{Martin80} has stimulated our investigation of variance asymptotics.
	
Having described the goals and context of this paper, we  now describe more precisely the assumptions on allowable score and input pairs $(\xi,\P)$ as well as our main results. The generality of allowable pairs $(\xi,\P)$ considered here necessitates several definitions which go as follows.
\subsection{\bf Admissible point processes  having fast decay of correlations}
\label{sec:prelim}
	
Throughout $\P \subset \R^d$ denotes  a simple point process. By a simple
point process we mean a random element taking values in $\mathcal{N}$,
the space of locally finite simple point sets in $\mR^d$ (or
equivalently Radon counting measures $\mu$ such that $\mu(\{x\}) \in
\{0, 1 \}$ for all $x \in \mR^d$)  and equipped with the canonical
$\sigma$-algebra $\mathcal{B}$.  Given a simple point
process $\P$ 
we interchangeably use the following representations of $\P$ :
$$
\P(\cdot) := \sum_i \delta_{X_i}(\cdot) \, \, (\mbox{random measure}) ; \, \, \, \P := \{X_i\}_{i \geq 1} \, \, (\mbox{random set}),$$
where $X_i, i \geq 1,$ are $\mR^d$-valued random variables (given a
measurable numbering  of points, which is irrelevant for the results presented in this paper). Points of $\R^d$ are denoted by $x$ or $y$ whereas points of $(\R^{d})^{k-1}$ are denoted by $\x$ or $\y$.  We let $\0$  denote a point at the origin of $\R^d$.
	
	For a bounded function $f$ on $\mR^d$ and a simple counting measure $\mu$, let $\mu(f)
	:= \langle f,\mu \rangle$ denote the integral of $f$ with respect to
	$\mu$.
	For a bounded set $B
	\subset \mR^d$ we let  $\mu(B) = \mu(\1_B)=\rm{card}(\mu \cap B)$, with
	$\mu$ in the last expression interpreted as the set of its atoms.

	For a simple Radon counting measure $\mu$ and $k \in \N$, its $k$th factorial power is
	$$\mu^{(k)}:=
	\begin{cases}
	\sum_{\text{distinct\,}x_1,\ldots,x_k\in\mu}\delta_{(x_1,\ldots,x_k)}&
	\text{when $\mu(\R^d)\ge k$,}\\
	0&\text{otherwise}.
	\end{cases}
	$$
	Note that $\mu^{(k)}$ is a Radon counting measure  on $(\mR^{d})^k$.
	Consistently, for a set $\X\subset\R^d$, we denote
	$\X^{(k)}:=\{(x_1,\ldots,x_k)\in(\R^d)^k: x_i\in\X, x_i\not=x_j\;\text{for}\; i\not=j\}$.
	The  {\em $k$th order factorial moment measure} of the
	(simple) point process $\P$ is defined as $\alpha^{(k)}(\cdot):=\E (\P^{(k)}(\cdot))$ on
	$(\R^{d})^k$ i.e., $\alpha^{(k)}(\cdot)$ is the intensity measure of the point process $\P^{(k)}(\cdot)$.
	Its  Radon-Nikodyn density $\rho^{(k)}(x_1,...,x_k)$
	(provided it exists) is  the {\em $k$-point correlation
		function}(or $k$th joint intensity)
and is characterized by the relation
	\be \label{alpha-def}
\alpha^{(k)}(B_1\times\dots\times B_k)=
	\E\Bigl(\prod_{1 \leq i \leq k} \P(B_i)\Bigr) = \int_{B_1 \times\dots\times B_k}
	\rho^{(k)}(x_1,...,x_k) \,\md x_1\dots\md x_k,
	\ee
where $B_1,...,B_k$ are mutually disjoint bounded Borel sets in $\R^d$. Since $\P$ is simple, we may put $\rho^{(k)}$ to be zero on the diagonals of $(\R^d)^k$, that is on the subsets of $(\R^d)^k$ where two or more coordinates coincide. The disjointness assumption is crucial as illustrated by the following useful relation: For any bounded Borel set $B \subset \mR^d$ and $k \geq 1$, we have
		\begin{equation}
		\label{eqn:fact_mom}
		\alpha^{(k)}(B^k) = \E \Bigl( \P(B) (\P(B)-1)\ldots (\P(B) - k +1) \Bigr) = \int_{B^k} \rho^{(k)}(x_1,\ldots,x_k) \, \md x_1 \ldots \md x_k.
		\end{equation}
	Heuristically, the {\em $k$th Palm measure} $P_{x_1,\ldots,x_k}$ of $\P$ is the
	probability distribution  of $\P$ conditioned on $\{x_1,\ldots,x_k\}
	\subset \P$. More formally, if $\alpha^{(k)}$ is  locally finite,
	there exists a family of  probability distributions
	$P_{x_1,\ldots,x_k}$ on $(\mathcal{N},\mathcal{B})$,
	unique up to an $\alpha^{(k)}$-null set of   $(\R^d)^k$,
	called the  $k\,$th Palm measures of $\P$, and satisfying the
	disintegration formula
	\be \label{disint}
	\E\Bigl( {\sum_{(x_1,\ldots,x_k) \in \P^{(k)}}\hspace{-1.8em}f(x_1,\ldots,x_k;\P)}\Bigr)
	= \int_{\R^{dk}}\int_{\mathcal{N}}f(x_1,\ldots,x_k;\mu)\,
	P_{x_1,\ldots,x_k}(\md \mu)\alpha^{(k)}(\md x_1\ldots \md x_k)
	\ee
	for any (say non-negative) measurable function $f$ on  $(\R^{d})^k \times\mathcal{N}$. Formula \eqref{disint} is also known as the refined Campbell theorem.
	
	To simplify notation,  write $\int_{\mathcal{N}}f(x_1,\ldots,x_k;\mu) P_{x_1,\ldots,x_k}(\md\mu)=\allowbreak
		\sE_{x_1,\ldots,x_k}(f(x_1,\ldots,x_k;\P))$, where
		$\sE_{x_1,\ldots,x_k}$ is the expectation corresponding to the
		Palm probability $\mP_{x_1,\ldots,x_k}$ on a
		 canonical probability space on which $\P$ is also defined.
	To further simplify notation,   denote by $\mP^!_{x_1,\ldots,x_k}$   the
	{\em reduced  Palm probabilities} and their expectation by $\sE^!_{x_1,\ldots,x_k},$
	which satisfies
	$\sE^!_{x_1,\ldots,x_k}(f(x_1,\ldots,x_k;\P))=\sE_{x_1,\ldots,x_k}(f(x_1,\ldots,x_k;\P\setminus\{x_1,\ldots,x_k\}))$
	\footnote{\label{fn:Palm} It can be shown that $\mP_{x_1,\ldots,x_k}(x_1,\ldots,x_k\in\P)=1$ for $\alpha^{(k)}$ a.e. $x_1,\ldots,x_k \in \mR^d$.}.
	
All Palm probabilities (expectations) are meaningfully defined only for $\alpha^{(k)}$ almost
all  $x_1,\ldots,x_k \in \mR^d$. Consequently, all expressions 	involving these measures should be understood in the $\alpha^{(k)}$ a.e. sense and   suprema
should likewise be understood as {\em essential} suprema with respect to~$\alpha^{(k)}$.
	
The following definition is reminiscent of the so-called weak exponential decrease of correlations introduced in \cite{Malyshev75} and subsequently used in \cite{BY1, Martin80,Nazarov12}.
\begin{definition}[$\omega$-mixing correlation functions]
\label{defn:admissible}
The correlation functions of a point process $\P$ are {\em $\omega$-mixing} if
there exists a decreasing function $\omega : \mathbb{N} \times \mR^+ \to \mR^+$ such that for all $n \in\mathbb{N}$, $\lim_{x \to \infty}\omega(n,x) = 0$ and for all $p,q \in \mathbb{N}, x_1,\ldots,x_{p+q} \in \mR^d$, we have
\[ | \rho^{(p+q)}(x_1,\ldots ,x_{p+q}) - \rho^{(p)}(x_1,\ldots ,x_p)\rho^{(q)}(x_{p+1},\ldots ,x_{p+q}) | \leq \omega(p+q,s),\]
where $s := d(\{x_1,\ldots,x_p\},\{x_{p+1},\ldots,x_{p+q}\})$ is as  at \eqref{defs}.
\end{definition}

By an admissible point process $\P$ on  $\R^d, d \geq 2$, we mean that $\P$ is simple, stationary (i.e., $\P+x \stackrel{d}{=} \P$ for
all $x\in\R^d$, where $\P+x$ denotes the translation of $\P$ by the
vector $x$), with non-null and finite intensity $\rho^{(1)}(\0)=\EXP{\P(W_1)}$, and has $k$-point correlation functions of all orders $k \in \N$.
By a {\em fast decreasing function}  $\phi: \R^+ \to [0,1]$ we mean $\phi$ satisfies $\lim_{x \to \infty}x^m\phi(x) = 0$ for all $m \geq 1$.

\begin{definition}[Admissible point process having fast decay of correlations]\label{def.admissiblePP}  Let $\P$ be an admissible point process.  $\P$ is said to have  {\em fast decay of correlations} if
its correlation functions are $\omega$-mixing as in  Definition~\ref{defn:admissible}
with  $\omega(n,x)=C_{n} \phi(c_{n}x)$ for some  {\em correlation decay constants} $c_{n} > 0$ and  $C_{n} < \infty$
and a fast decreasing function $\phi: \R^+ \to [0,1]$,  called {\em a correlation decay function}.
\end{definition}

More explicitly, an admissible point process has {\em  fast decay of correlations}, if for all $p, q \in \N$ and all $(x_1,\ldots ,x_{p+q}) \in (\mR^d)^{p+q}$
\begin{equation}
\label{eqn:clustering_condition}
| \rho^{(p+q)}(x_1,\ldots ,x_{p+q}) - \rho^{(p)}(x_1,\ldots ,x_p)\rho^{(q)}(x_{p+1},\ldots ,x_{p+q}) | \leq C_{p+q} \phi(c_{p+q}s),
\end{equation}
where $s := d(\{x_1,\ldots ,x_p\},\{x_{p+1},\ldots,x_{p+q}\})$ is as at \eqref{defs} and $C_k,c_k,\phi$ are as in
Definition~\ref{def.admissiblePP}. Without loss of generality, we assume that $c_k$ is non-increasing in $k$, and that $C_k \in [1, \infty)$ is non-decreasing in $k$. 
As a by-product of our proof of the asymptotic normality
of $\mu_n^{\xi}$ in~\eqref{eqn:rand_meas}, we establish that
 the fast decay of correlations of $\P$ implies that it is Brillinger mixing; cf Remark (vi) in Section \ref{sec:main_results} and  Remarks at the end of
 Section~\ref{sec:ursell_bounds}.

Admissible point processes having fast decay of correlations are ubiquitous and include certain determinantal, permanental,  and Gibbs point processes, as explained in  Section \ref{sec:ex_pp}. The $k$-point correlation functions of an admissible point process having fast decay of correlations are bounded i.e.,
	\begin{equation}
	\label{eqn:corr_bounded}
	\sup_{(x_1,\ldots,x_k) \in (\R^{d})^k } \rho^{(k)}(x_1,\ldots,x_k) \leq \kappa_k < \infty, \ 
	\end{equation}
	for some constants $\kappa_k$, which without loss of generality are assumed non-decreasing in~$k$. Also without loss of generality, assume $\kappa_0 :=\max \{\rho^{(1)}(\0),1\}$. For stationary $\P$ with intensity $\rho^{(1)}(\0) \in (0, \infty)$ we have 	that \eqref{eqn:clustering_condition} implies~\eqref{eqn:corr_bounded} with
\be \label{kapbound} \kappa_k \leq (\rho^{(1)}(\0))^k + \sum_{i = 2}^k C_i   (\rho^{(1)}(\0))^{k - i} \leq k C_k \kappa_0^k.
\ee
The bound \eqref{kapbound} helps to determine when point processes having fast decay of correlations also have exponential moments, as in Section 2.1.

	
	\subsection{\bf Admissible score functions}
	\label{sec:score}
	
Throughout we restrict to translation-invariant score functions $\xi : \mR^d \times \mathcal{N} \to \mR,$
i.e., those which are measurable in each coordinate,  $\xi(x,\X) = 0$ if $x \notin \X \in \mathcal{N}$, and for all $y \in \mR^d$, satisfy $\xi(\cdot+ y,\cdot + y)= \xi(\cdot,\cdot)$.
	
We introduce  classes (A1) and (A2) of {\em admissible} score and input pairs $(\xi, \P)$.  Specific examples of admissible input pairs of both classes are provided in Sections \ref{sec:ex_pp} and \ref{sec:applications2}. The first class allows  for admissible input $\P$ as in Definition~\ref{def.admissiblePP} whereas the second
considers  admissible input $\P$ having fast decay of correlations \eqref{eqn:clustering_condition}, subject to $c_k \equiv 1$ and growth conditions
	on the decay constants $C_k$ and the decay function $\phi$.

	\begin{definition}[Class (A1) of admissible score and input pairs $(\xi, \P)$] \label{def.A1}
		Admissible input $\P$ consists of admissible point processes having fast decay of correlations as in Definition~\ref{def.admissiblePP}. Admissible score functions are of the form		
		\begin{equation}
		\label{eqn:score_U_statistic}
		\xi(x,\X) := \frac{1}{k!} \sum_{\x \in \X^{(k-1)}}h(x,\x),
		\end{equation}
		for some $k\in \N$ and a  symmetric, translation-invariant function
		$h : \R^d \times (\mR^{d})^{k - 1} \to \mR$
		such that $h(x_1,\ldots,x_k) = 0$ whenever either $\max_{2 \leq i \leq k}
		|x_i-x_1| > r$ for some given $r>0$  or  when $x_i=x_j$ for some $i\not=j$. When $k = 1$, we set $\xi(x,\X) = h(x)$.
		Further, assume
		$$\|h \|_{\infty} := \sup_{\x \in \mR^{d(k-1)}} |h(\0,\x)| < \infty. $$		
	\end{definition}
	
	The interaction range for $h$ is at most $r$, showing that the functionals $H_n^{\xi}$ defined at  \eqref{eqn:total_mass} generated via scores \eqref{eqn:score_U_statistic}
		are  local {\em $U$-statistics} of order $k$ as in \cite{Reitzner13}.
	Before introducing a more general class of score functions,
	we recall \cite{Baryshnikov05, LPS, PeEJP, PY1, PY5} a few definitions formalizing  the notion of the local
	dependence of $\xi$ on its input. Let $B_r(x):= \{ y:|y-x|\le r\}$ denote the ball of radius $r$ centered at $x$ and $B^c_r(x)$ its complement.
\begin{definition}[Radius of stabilization]\label{def.Stab.Radius}
Given a score function $\xi$, input $\X$, and $x \in \X$,  define the radius of stabilization $R^\xi(x,\X)$ to be the smallest $r \in \N$ such that
$$
\xi(x, \X \cap B_r(x)) = \xi(x, (\X \cap B_r(x)) \cup ( {\cal A} \cap B^c_r(x))) \,
$$
for all ${\cal A} \subset \R^d$  locally finite. If no such finite $r$ exists, we set $R^\xi(x,\X)=\infty$.
\end{definition}
	
	If  $\xi$ is  a translation invariant  function then so
	is $R^\xi(x,\X)$. Score
	functions~\eqref{eqn:score_U_statistic} of class (A1) have
	radius of stabilization upper-bounded by~$r$.
	\begin{definition}[Stabilizing score function]\label{def.stabilizing_score}
		We say that  $\xi$  {\em is stabilizing on $\P$} if for all $l \in \N$ there are constants $a_l > 0$, such that
		\be \label{stab}
		\sup_{1 \leq n \leq \infty} \sup_{x_1,\ldots,x_l \in W_n} \mP_{x_1,\ldots,x_l}
		\bigl(R^\xi(x_1,\P_n) > t\bigr) \leq \varphi(a_lt)
		\ee
		with $\varphi(t) \downarrow 0$ as $t \to \infty$. Without loss of generality the $a_l$ are non-increasing in $l$ and $0 \leq \varphi \leq 1.$
			In \eqref{stab} and elsewhere, we adopt the convention that
			$W_\infty := \R^d$ and $\P_{\infty} := \P$. The second sup in  \eqref{stab}  is understood as
			$\text{ess}\sup$ with respect to the measure $\alpha^{(l)}$ at \eqref{alpha-def}.
	\end{definition}
	\begin{definition}[Exponentially stabilizing score function]
		We say that $\xi$ is {\em exponentially stabilizing} on $\P$ if $\xi$ is  stabilizing on $\P$
		as in Definition~\ref{def.stabilizing_score} with 	$\varphi$ satisfying
		\be \label{varphibd}
		\liminf_{t \to \infty} {\log \varphi(t) \over t^c } \in (-\infty,  0)
		\ee
		for some $c \in (0, \infty)$.
	\end{definition}
	
	We define a general class of score functions exponentially stabilizing on their input.
	
	\begin{definition}[Class (A2) of admissible score and input pairs $(\xi, \P)$]\label{def.A2}
		Admissible input $\P$  consists of admissible point processes having fast decay of correlations as in Definition \ref{def.admissiblePP} with correlation decay constants satisfying  $c_k \equiv 1$,
		\be \label{eqn:sum}
		C_k = O(k^{ak}),
		\ee
		for some $a\in[0,1)$ and correlation decay function $\phi$ satisfying the exponential decay condition
		
		\be \label{phibd}
		\liminf_{t \to \infty} {\log \phi(t) \over t^b } \in (-\infty,  0)
		\ee
		for some constant $b \in (0, \infty)$. 
		Admissible score functions  $\xi$ for this class are exponentially
		stabilizing on the input $\P$ and satisfy a {\em power growth condition}, namely there exists $\hat{c} \in [1, \infty)$ such that for all $r \in (0, \infty)$
			\begin{equation} \label{eqn:xibd}
			|\xi(x,\X \cap B_r(x))|\1[{\rm card}(\X\cap B_r(x)) = n] \leq (\hat{c} \max (r,1) )^n.
			\end{equation}
	\end{definition}
The condition $c_k \equiv 1$ is equivalent to $c_* := \inf c_k > 0$. This follows since we may replace the fast
decreasing function $\phi(.)$ by 
$\phi(c_* \times \cdot)$, with $c_k \equiv 1$ for this new fast decreasing function. Score functions of class (A1) also  satisfy the power growth condition \eqref{eqn:xibd} since in this case the left hand side of \eqref{eqn:xibd} is at most $\|h \|_{\infty} n^{(k-1)}/k$. Thus the generalization from (A1) to  (A2) consists in replacing local U-statistics by
exponentially stabilizing score functions satisfying the power growth condition. This is done at the
price of imposing stronger conditions on the input process, requiring in
particular that it has finite  exponential moments, as explained in Section \ref{momsection}.
	
\subsection{\bf Fast decay of correlations of the $\xi$-weighted  measures}
\label{sec:strong_clustering_mix_mom}
	
	The following  $p$-moment condition involves the score function $\xi$ and the input $\P$.
	We shall describe in Section \ref{momsection} ways to control the $p$-moments of input pairs of
	class (A1) and (A2).
	
	\begin{definition}[Moment condition]
	\label{d.p-moment}
	\Comment{increasing  $\tM_p$ assumption moved to the Def; BB}
		Given $p \in [1, \infty)$, say that the pair $(\xi, \P)$ satisfies the $p$-moment condition if
		\be \label{eqn:mom}
		\sup_{1 \leq n \leq \infty} \sup_{1 \leq p' \leq \lfloor p \rfloor } \sup_{x_1,...,x_{p'} \in W_n} \sE_{x_1,\ldots,x_{p'}}
		\max\{|\xi(x_1, \P_n )|, 1\}^{p}   \le \tM_p < \infty
		\ee
		for some constant $\tM_p:=\tM_p^\xi$, where  $\sup$ signifies
		$\text{ess}\sup$ with respect to $\alpha^{(p)}$. Without loss of generality we assume that $\tM_p$ is increasing in $p$ for all $p$ such that~\eqref{eqn:mom} holds.
	\end{definition}
	
 We next consider the decay of the functions at~\eqref{eqn:mixedmomentn}, the so-called		
correlation functions of the $\xi$-weighted measures  at \eqref{xwm}.
These functions indeed play the same role as the $k$-point correlation functions of the simple point process $\P.$  When $\xi \equiv 1$ they obviously reduce to the correlation functions of $\P$.
For general $\xi$ and  $k_i \equiv 1$   they are densities (`mixed moment densities' in the language of \cite{Baryshnikov05}) of the higher-order  moment measures  of the
$\xi$-weighted measures  with all distinct arguments.
In the case of repeated arguments, the moment measures of a simple point process  ``collapse'' to appropriate lower dimensional ones.
This is neither  the case for non-simple point processes nor  for our $\xi$-weighted measures,
where general exponents $k_i$ are  required to properly take into account repeated arguments.

 When $k_i \equiv 1$ for all $1 \leq i \leq p$, we write $m_{(p)}(x_1,...,x_p;n)$ instead of $m^{(1,...,1)}(x_1,...,x_p;n)$. Abbreviate
 $m^{(k_1,\ldots,k_p)}(x_1,\ldots,x_p;\infty)$ by $m^{(k_1,\ldots,k_p)}(x_1,\ldots,x_p).$	These functions exist whenever \eqref{eqn:mom} is satisfied for $p$ set to  $k_1+\ldots+k_p$ and provided the $p$-point correlation function $\rho^{(p)}$ exists.
As for the input process $\P$ we consider mixing properties and fast decay of correlations for the $\xi$-weighted measures at \eqref{xwm}.

\begin{definition}[$\tilde\omega$-mixing  correlation functions of  $\xi$-weighted measures]
\label{defn:mixing-compound}	
The correlation functions~\eqref{eqn:mixedmomentn} 
are said to be  {\em $\omega$-mixing} if
there exists a decreasing function $\tilde\omega : \mathbb{N} \times \mR^+ \to \mR^+$ such that for all $p \in\mathbb{N}$, $\lim_{x \to \infty}\tilde\omega(p,x) = 0$ and for all $p,q \in \mathbb{N}$, distinct $x_1,\ldots,x_{p+q} \in \mR^d$ and $n \in \N \cup \{\infty\}$
\begin{align}\nonumber
		\Bigl|m^{(k_1,\ldots,k_{p+q})}(x_1,\ldots,x_{p+q};n)- m^{(k_1,\ldots,k_{p})}(x_1,\ldots,x_{p};n)\, m^{(k_{p+1},\ldots,k_{p+q})}(x_{p+1},\ldots,&x_{p+q};n)\Bigr|\\
&\le
		\tilde\omega(K,s)\,,  \label{e.clustering-generalized-omega-mixed-monents}
\end{align}
		where $K:=\sum_{i=1}^{p+q}k_i$ and  $s := d(\{x_1,\ldots,x_p\},\{x_{p+1},\ldots,x_{p+q}\})$ is as at \eqref{defs}.
\end{definition}
\begin{definition}[Fast decay of correlations of the $\xi$-weighted measures]
\label{def.strong-clustering}
The $\xi$-weighted measures are said to have  {\em fast decay of correlations}
if their correlations functions are $\tilde\omega$-mixing as in  Definition~\ref{defn:mixing-compound}
with  $\tilde\omega(n,x)=\tilde C_{n} \tilde\phi(\tilde c_{n}x)$ for some fast decreasing function $\tilde\phi: \R^+ \to [0,1]$ and some constants  $\tilde c_{n} > 0$ and  $\tilde C_{n} < \infty$.
\end{definition}
More explicitly, the $\xi$-weighted measures \eqref{xwm} have fast decay of correlations
 if there exists a fast-decreasing function $\tilde\phi$ and constants $\tilde C_k < \infty$, $\tilde c_k > 0, k \in \N$ such that for all $n \in \N \cup \{\infty\}$, $p, q \in \N$ and any collection of  positive integers $k_1,\ldots,k_{p+q}$, we have
		\begin{align}\nonumber
		\Bigl|m^{(k_1,\ldots,k_{p+q})}(x_1,\ldots,x_{p+q};n)- m^{(k_1,\ldots,k_{p})}(x_1,\ldots,x_{p};n)\, m^{(k_{p+1},\ldots,k_{p+q})}(x_{p+1},\ldots,&x_{p+q};n)\Bigr|\\
&\le		\tilde C_{K}\tilde\phi(\tilde c_{K} s)\,,  \label{e.clustering-generalized-mixed-monents}
		\end{align}
		where $x_1,\ldots,x_{p+q},K$ and $s$ are as in Definition~\ref{defn:mixing-compound}.
	
Our first theorem shows that the fast decay of correlations is inherited from the input process~$\P$ by the $\xi$-weighted measures for a wide class of score functions and input. This key result forms the starting point of our approach.
	
\begin{theorem}
\label{prop:clustering_gen}
Let $(\xi,\P)$ be an admissible score and input pair of class (A1) or (A2) such that the $p$-moment condition  \eqref{eqn:mom} holds for all $p \in (1,\infty)$. Then the correlations of the $\xi$-weighted measures  decay fast as at~\eqref{e.clustering-generalized-mixed-monents}.
\end{theorem}
	
We prove this theorem in Section \ref{sec:proofs_clustering}, where it is also shown that it subsumes more specialized results of \cite{Baryshnikov05, SY}.
	\subsection{\bf Main results}
	\label{sec:main_results}
	We give the limit theory for the measures $\mu_n^\xi, n \geq 1,$ defined at \eqref{eqn:rand_meas}.
	Given a score function $\xi$  on admissible input $\P$ we set  \footnote{For a stationary point process $\P$, its Palm expectation $\E_{\0}$ (and consequently $m_{(1)}(\0)$, $m_{(2)}(\0,x)\md x$) is meaningfully defined e.g. via the Palm-Matthes approach.}

	\be  \label{eqn:sigdef}
		\sigma^2(\xi) :=
		\E_{\0}\xi^2(\0, \P) \rho^{(1)}(\0) + \int_{\R^d}(m_{(2)}(\0,x) -  m_{(1)}(\0)^2 )\,\md x.
	\ee
	The following result provides expectation and variance asymptotics for ${\mu}^{\xi}_n(f)$, with $f$
	belonging to the space ${\cal B}(W_1)$ of bounded measurable functions on $W_1$.
	\begin{theorem}
		\label{prop:variance_upper_bound}
		Let $\P$ be an admissible point process on $\R^d$. \\
		\noindent (i) If $\xi$ satisfies exponential stabilization \eqref{varphibd} and the $p$-moment condition \eqref{eqn:mom} for some $p \in (1, \infty)$ then for all $f \in  {\cal B}(W_1)$
		\be
		\label{expasy}
		\Big| n^{-1} \E \mu_n^\xi(f) - \E_{\0} \xi(\0, \P) \rho^{(1)}(\0)  \int_{W_1}f(x)\,\md x\Big| = O(n^{-1/d}).
		\ee
		If $\xi$ only satisfies stabilization \eqref{stab} and the $p$-moment condition \eqref{eqn:mom} for some $p \in (1, \infty)$, then the right hand side of \eqref{expasy} is $o(1)$. \\
		\noindent (ii) Assume that the second correlation function $\rho^{(2)}$ of $\P$ exists  and is bounded as in~\eqref{eqn:corr_bounded},
		that $\xi$ satisfies \eqref{stab},
		and  that $(\xi, \P)$  satisfies 
		the $p$-moment   condition \eqref{eqn:mom} for some $p \in (2, \infty)$.  If the second-order  correlations of  the $\xi$-weighted measures decay fast,
 i.e. satisfy ~\eqref{e.clustering-generalized-mixed-monents}
		with $ p =  q = k_1 = k_2 = 1$ and all $n \in \N \cup\{\infty\}$, then  for all $f \in  {\cal B}(W_1)$
		\be
		\label{eqn:var}
		\lim_{n \to \infty} n^{-1} \Var \mu_n^{\xi}(f) =  \sigma^2(\xi) \int_{W_1} f(x)^2 \,\md x \in [0,\infty),
		\ee
		whereas for all $f, g \in {\cal B}(W_1)$
		\be
		\label{eqn:Cov}
		\lim_{n \to \infty} n^{-1}  {\rm{Cov}} ( \mu_n^{\xi}(f) , \mu_n^{\xi}(g) )  =  \sigma^2(\xi) \int_{W_1} f(x) g(x) \,\md x.
		\ee
	\end{theorem}
	We remark that \eqref{expasy} and \eqref{eqn:var}  together show convergence in probability
	$$n^{-1} \mu_n^\xi(f)
	\ \ {\stackrel{{P}}{\longrightarrow}} \ \  \E_{\0} \xi(\0, \P) \rho^{(1)}(\0) \int_{W_1}f(x)\,\md x \, \, \, \text{as $n \to \infty$.}
	$$
	
	The proof of variance asymptotics \eqref{eqn:var} requires fast decay of the second-order
correlations of the $\xi$-weighted measures. Fast decay of {\em all} higher-order correlations as in Definition~\ref{def.strong-clustering}
 yields Gaussian fluctuations of  $\mu_n^\xi, n \geq 1,$ under moment conditions on the atom sizes (i.e., under moment conditions on $\xi$) and a variance lower bound. Let $N$ denote a mean zero normal random variable with variance $1$.
We write $f(n)=\Omega(g(n))$ when
		$g(n)=O(f(n))$, i.e., when  $\lim\inf_{n\to\infty} |f(n)/g(n)|>0$.
	\begin{theorem}
		\label{thm:main0}
		Let $\P$ be an admissible  point process on $\R^d$ and let the pair $(\xi,\P)$  satisfy the $p$-moment condition \eqref{eqn:mom} for all $p \in (1, \infty)$. If the correlations of the
$\xi$-weighted measures at \eqref{xwm} decay fast as in Definition~\ref{def.strong-clustering}
and if  $f \in  {\cal B}(W_1)$ satisfies
		\begin{equation}
		\label{vlb0}
		\Var{\mu^{\xi}_n(f)} = \Omega(n^{\nu})
		\end{equation}
		for some $\nu \in (0, \infty)$, then as $n \to \infty$
		\begin{equation}
		\label{mainclt0}
		(\Var{\mu^{\xi}_n(f)})^{-1/2} ( \mu^{\xi}_n(f) - \E \mu^{\xi}_n(f))  \tod N.
		\end{equation}
	\end{theorem}
	
	\vskip.3cm
	
	Combining Theorem \ref{prop:clustering_gen} and Theorem \ref{thm:main0} yields the following theorem, which is well-suited for off-the-shelf use in applications, as  seen in Section \ref{sec:applications2}.
	\begin{theorem}
		\label{thm:main1}
		Let $(\xi,\P)$ be an admissible  pair of class (A1) or (A2) such that the $p$-moment condition \eqref{eqn:mom} holds for all $p \in (1,\infty)$.
		If $f \in  {\cal B}(W_1)$ satisfies condition \eqref{vlb0} for some $\nu \in (0, \infty)$, then $\mu^{\xi}_n(f)$ is asymptotically normal as in \eqref{mainclt0}, as $n \to \infty$.
	\end{theorem}
Theorems \ref{prop:variance_upper_bound} and \ref{thm:main0} are proved in Section 4. We next compare our results with those in the literature. Point processes mentioned below
are defined in Section \ref{sec:ex_pp}.
   \paragraph{\bf Remarks}
	\begin{enumerate}[itemindent=0.8cm,labelwidth=*,label=(\roman*),leftmargin=0cm]
		\item {\em Theorem \ref{prop:variance_upper_bound}.}  In the case of
		Poisson and binomial input $\P$, the limits \eqref{expasy}
		and \eqref{eqn:var} are  shown in \cite{PY4} and \cite{Baryshnikov05, PeEJP}, respectively (the binomial point processes are not the restriction of an infinite point process to windows, but rather a re-scaled binomial point process on $[0,1]^d$).
		In the case of Gibbsian input, the limits \eqref{expasy}
		and \eqref{eqn:var} are established in \cite{SY}. Theorem \ref{prop:variance_upper_bound} shows these limits hold for general stationary input.
The paper \cite{AY1} gives a weaker version of Theorem \ref{prop:variance_upper_bound} for specific $\xi$ and for $f = \1[x \in W_1]$. In full generality, the convergence   rate \eqref{expasy} is new.
		
		\item {\em Theorems \ref{thm:main0} and \ref{thm:main1}.}   Under condition  \eqref{vlb0},Theorems \ref{thm:main0} and \ref{thm:main1}
		provide a central limit theorem for non-linear statistics of  either $\alpha$-determinantal and $\alpha$-permanental input ($|\alpha|^{-1} \in \N$) with a fast-decaying kernel as at
		\eqref{fastdk}, the zero set $\P_{GEF}$ of a Gaussian entire function, or rarified Gibbsian input.  When $\xi \equiv 1$,  then $\mu^{\xi}_n(f)$ reduces to the  {\em linear statistic} $\sum_{x \in \P_n} f(x)$.
		These theorems  extend the central limit theorem for linear statistics of $\P_{GEF}$ as established in \cite{Nazarov12}.
		When  the input is determinantal with a fast decaying kernel as at \eqref{fastdk}, then  Theorems \ref{thm:main0} and \ref{thm:main1}
		also extend the main result of  Soshnikov \cite{So}, whose pathbreaking paper gives a central limit theorem for linear statistics for any determinantal input, provided the  variance grows as least as fast as a power of the expectation. Proposition 5.7 of \cite{Shirai03} shows central limit theorems for linear statistics of $\alpha$-determinantal point processes with $\alpha = -1/m$ or $\alpha$-permanental point processes with $\alpha = 2/m$ for some $m \in \N$.  During the revision of this article, we noticed the recent work \cite{PDL}.  This paper shows that when the kernel satisfies \eqref{fastdk} with $\omega(s) = o(s^{-(d + \epsilon)/2 })$ and when $|\xi|$ is bounded with a deterministic radius of stabilization, then $H_n^\xi$ at \eqref{eqn:total_mass}
is asymptotically normal.  The generality of the score functionals and point processes considered in our article necessitates  assumptions on the determinantal kernel which are more restrictive than those of \cite{PDL, So}.
				
\item {\em Variance lower bounds.}
To prove asymptotic normality it is customary to require  variance lower bounds as at  \eqref{vlb0};  \cite{Nazarov12} and  \cite{So} both require assumptions of this kind.	Showing condition \eqref{vlb0} is a separate problem and it fails in general. For example, the variance of the point count of some determinantal point processes, including the GUE point process, grows at most logarithmically. This phenomena is especially pronounced in dimensions $d = 1, 2$.  Additionally, given input $\P_{GEF}$ and  $\xi \equiv 1$, the bound \eqref{vlb0} may fail even when $f$ is a smooth cut-off that equals one in a neighborhood of the origin (cf. Prop. 5.2 of \cite{NS11}).
On the other hand, if $\xi \equiv 1$, and if the kernel $K$ for a determinantal point process satisfies $\int_{\R^d} |K(\0, x)|^2 \md x < K(\0,\0) = \rho^{(1)}(\0)$, then
recalling the definition of $\sigma^2(\xi)$ at \eqref{eqn:sigdef}, we have $\sigma^2(\xi) = \sigma^2(1) = \rho^{(1)}(\0)  - \int_{\R^d}  |K(\0, x)|^2 \md x > 0$. In the case of rarified Gibbsian input, the bound \eqref{vlb0} holds with $\nu = 1$, as shown in of \cite[Theorem 1.1]{XY}. Theorem \ref{thm:main1} allows for surface-order variance growth, which arises for linear statistics $\sum_{x \in \P_n} \xi(x)$ of determinantal point processes; see \cite[(4.15)]{Forrester99}.
		
\item {\em Poisson, binomial, and Gibbs input}.  When $\P$ is Poisson or binomial input and when $\xi$ is a functional
which stabilizes exponentially fast as at \eqref{varphibd}, then  $\mu^{\xi}_n$ is asymptotically normal \eqref{mainclt0} under  moment conditions on $\xi$;	see the survey \cite{Yukich12}. When $\P$ is a rarified Gibbs point process with `ancestor clans' decaying exponentially fast, and when $\xi$ is an exponentially stabilizing functional, then $\mu^{\xi}_n$ satisfies normal convergence \eqref{mainclt0} as established in \cite{SY, XY}.
		
\remove{  Drawing on some ideas of \cite{BiscioLav}, one can prove that the random measures $\mu_n^\xi$ are Brillinger mixing. }

\item {\em Mixing conditions.}  Central limit theorems for geometric functionals of mixing point processes (random fields) are established in \cite{Baddeley80, Bulinski12, Ivanoff82,Heinrich94,Heinrich99,Heinrich13,PDL}. The geometric functionals considered in these papers are different than the ones considered here; furthermore the relation between the mixing conditions in these papers and $\omega$-mixing correlation functions as in Definition \ref{defn:admissible} is unclear. Though correlation functions are simpler than mixing coefficients, which depend on $\sigma$-algebras generated by the point processes, our decay rates appear more restrictive than those needed in afore-mentioned papers. 
    A careful investigation of the relations between the various notions of mixing and fast decay of correlations lies beyond the scope of our limit results and will be treated in a separate paper. In the case of point processes on discrete spaces, such a study is easier, c.f. \cite{Ramreddy17}.
		
\item {\em Brillinger mixing and fast decay of correlations.} Brillinger mixing \cite[Section 3.5]{Ivanoff82} is defined via finiteness of integrals of the {\em reduced cumulant measures} (see Section \ref{sec:prop_cumulants}). 
           The very definition of Brillinger mixing implies volume-order growth of cumulants; the converse follows using the ideas in the proof of \cite[Theorem 3.2]{BiscioLav}. The key to proving our announced central limit theorems  is to show that the
fast decay of correlations of the $\xi$-weighted measures  \eqref{xwm}  implies volume-order growth of cumulants and hence  Brillinger mixing; see the remarks at the beginning of Section \ref{sec:proof_clt1}  and also those and at the end of Section~\ref{sec:ursell_bounds}.

\item  {\em Multivariate central limit theorem.}
 We may use the Cram\'er-Wold device to extend  Theorems \ref{prop:variance_upper_bound} and \ref{thm:main1} to the multivariate setting as follows.	
		Let $(\xi, \P)$ be a pair satisfying the hypotheses of Theorems \ref{prop:variance_upper_bound} and \ref{thm:main1}.
		If $f_i \in  {\cal B}(W_1),  1 \leq i \leq k,$
		satisfy the variance limit \eqref{eqn:var}  with $\sigma^2(\xi) > 0$,
		then as $n \to \infty$ the fidis
		$$
		\left( \frac{ \mu^{\xi}_n(f_1) - \E{\mu^{\xi}_n(f_1)} }  { \sqrt{ n  } } ,
		\ldots  ,\frac{\mu^{\xi}_n(f_k) - \E{\mu^{\xi}_n(f_k)} } { \sqrt{ n  } } \right)
		$$
		converge to that of a centred Gaussian field having covariance kernel $f,g \mapsto \sigma^2(\xi) \int_{W_1} f(x)\allowbreak g(x) \md x$.
				
		\item {\em Deterministic radius of stabilization}.
		It may be shown that our main results go through without the condition~\eqref{phibd} if the radius of stabilization $R^\xi(x, \P)$ is bounded by a non-random (deterministic) constant	and if \eqref{eqn:sum} and \eqref{eqn:xibd} are satisfied.   However we are unable to find any interesting examples of point processes satisfying \eqref{eqn:clustering_condition} but not \eqref{phibd}.

		\item {\em Fast decay of the correlation of the $\xi$-weighted measures; Theorem \ref{prop:clustering_gen}.} Though the cumulant method is common to \cite{Baryshnikov05, SY, Nazarov12} and this article,
a distinguishing and novel feature of our approach  is the proof of fast decay of correlations of the $\xi$-weighted measures  \eqref{e.clustering-generalized-mixed-monents}, and consequently  their Brillinger mixing,  for a wide class of functionals and point processes. As mentioned in the introduction, the proof of this result is via factorial moment expansions, which  differs from the approach of \cite{Baryshnikov05, SY, Nazarov12} (see the remarks at the beginning of Section \ref{sec:proofs_clustering}). Fast decay of correlations of the $\xi$-weighted measures \eqref{e.clustering-generalized-mixed-monents} appears to be of independent interest.  It features  in the proofs of moderate deviation principles  and laws of the iterated logarithms for stabilizing functionals of Poisson point process \cite{Baryshnikov08}, \cite{ERS}.
Fast decay of correlations \eqref{e.clustering-generalized-mixed-monents} yields volume order cumulant bounds, useful in establishing  concentration inequalities as well as moderate deviations, as explained in \cite[Lemma 4.2]{Grote16}.

\item  {\em Normal approximation.} Difference operators (which appear in our factorial moment expansions) are also a key tool in the Malliavin-Stein method \cite{Nourdin12,Peccati15}. This method yields presumably optimal rates of normal convergence for various statistics (including many considered in Section \ref{sec:applications2}) in stochastic geometric problems \cite{Lachieze15, LPS, Reitzner13,Lachieze17}. However, these methods currently apply only to functionals defined on Poisson and binomial  point processes. It is an open question whether a refined use of these methods  would yield rates of convergence in our central limit theorems.
\item  {\em Cumulant bounds.} As mentioned, we establish that the $k$th order cumulants for $\langle f, \mu_n^\xi \rangle$ grow {\em at most linearly} in $n$ for $k \geq 1$.  Thus, under assumption \eqref{vlb0}, the cumulant $C_n^k$ for $(\Var \langle f, \mu_n^\xi \rangle)^{-1/2} \langle f,  \mu_n^\xi \rangle$ satisfies $C_n^k \leq D(k) n^{1 - (\nu k/2)}$, with $D(k)$ depending only on $k$.
For $k = 3,4,...$ and $\nu > 2/3$, we have  $C_n^k \leq D(k)/( \Delta(n))^{k - 2}$, where $\Delta(n):= n^{(3\nu - 2)/2 }$.  When $D(k)$ satisfies $D(k) \leq (k!)^{1 + \gamma}$, $\gamma$ a constant,  we obtain the Berry-Esseen bound (cf. \cite[Lemma 4.2]{Grote16})
$$
\sup_{t \in \R} \left | \sP \left( \frac{ \mu_n^\xi(f) - \E \mu_n^\xi(f)} { \sqrt{ \Var \mu_n^\xi(f) } } \leq t \right) - \sP(N \leq t) \right | = O( \Delta(n)^{-1/(1 + 2 \gamma) } ).
$$
Determining conditions on input pairs $(\xi, \P)$ insuring the bounds $\nu > 2/3$ and  $D(k) \leq (k!)^{1 + \gamma}$, $\gamma$ a constant, is beyond the scope of this paper. When $\P$ is Poisson input, this issue is addressed by \cite{ERS}.

\end{enumerate}

We next consider the case when the fluctuations of $H_n^\xi(\P)$ are not of volume-order, that is to say $\sigma^2(\xi) = 0$. Though this may appear to be a degenerate condition, interesting examples involving determinantal point processes or zeros of GEF in fact satisfy  $\sigma^2(1) = 0$. Such point processes are termed `super-homogeneous point processes' \cite[Remark 5.1]{Nazarov12}.
	Put
	\be
	\label{eqn:arand_meas}
	\H_n^\xi(\P):= \sum_{x \in \P_n} \xi(x, \P).
	\ee
	
The summands in  $\H_n^{\xi}(\P)$,	in contrast to those of $H_n^\xi(\P)$,
are not sensitive to boundary effects. We shall show that under volume-order scaling the asymptotic variance of $\hat{H}_n^{\xi}(\P)$  also equals $\sigma^2(\xi)$. However, when $\sigma^2(\xi) = 0$ we derive surface-order variance asymptotics for $\hat{H}_n^{\xi}(\P)$. Though a similar result should plausibly hold for $H_n^{\xi}(\P)$, a  proof seems beyond the scope of the current paper.
Letting $\Vol_d$ denote the $d$-dimensional Lebesgue volume,
for $y \in \R^d$ and $W \subset \R^d$, put
\be  \label{gamm}
\gamma_W(y): = \Vol_{d} (W \cap (\mR^d \setminus W  -y)).
\ee
By  \cite[Lemma 1(a)]{Martin80}, we are justified in writing $ \gamma(y) := \limn  \gamma_{W_n}(y)/ n^{(d-1)/d}.$

\begin{theorem}
			\label{prop:variance_lower_bound}
			Under the assumptions of Theorem~\ref{prop:variance_upper_bound}(ii)
			suppose also that the pair $(\xi,\P)$ exponentially stabilizes as in~\eqref{varphibd}. Then
			\be \label{propeq1a}  \lim_{n \to \infty} n^{-1}\Var{H_n^{\xi}(\P)} = \sigma^2(\xi). \ee
			If moreover $\sigma^2(\xi)=0$ in~\eqref{eqn:var}
			then
			\be \label{propeq}  \lim_{n \to \infty} n^{-(d-1)/d} \Var{\H_n^{\xi}(\P)} = \sigma^2(\xi, \gamma) := \int_{\mR^d}(m_{(1)}(\0)^2 - m_{(2)}(\0,x))\gamma(x)\,\md x \in [0,\infty).\ee
\end{theorem}
			
	\paragraph{\bf Remarks}
	\begin{enumerate}[itemindent=0.8cm,labelwidth=*,label=(\roman*),leftmargin=0cm,topsep=-2ex]
		\item Checking positivity of $\sigma^2(\xi, \gamma)$	is not always straightforward, though we note if $\xi$ has the form \eqref{eqn:score_U_statistic}, then the disintegration formula
		\eqref{disint} yields
		\begin{eqnarray}
	& &	\sigma^2(\xi,\gamma) = \sum_{j=0}^k\left(j!(k-j-1)!(k-j-1)!\right)^{-1}\int_{\mR^d} \md x \, \gamma(x) \left( \int_{(B_r(\0) \cap B_r(x))^j \times B_r(\0)^{k-j-1} \times B_r(x)^{k-j-1}} \right. \no \\
	&  & \left. h(\0,\y,\z)h(x,\x,\z)  \left[\rho^{(k)}(\0,\y,\z)\rho^{(k)}(x,\x,\z) - \rho^{(2k-j)}(\0,\y,\z,x,\x) \right] \,\md \z \md \y \md \x \right) . \no 
		\end{eqnarray}
		%
\item  Theorem \ref{prop:variance_upper_bound} and  Theorem \ref{prop:variance_lower_bound} extend \cite[Propositions 1 and 2]{Martin80}, which are valid only for $\xi \equiv 1$, to general functionals. If an admissible  pair $(\xi, \P)$ of type  (A1) or (A2) is such that $\H_n^{\xi}(\P)$ does not have volume-order variance growth, then Theorems \ref{prop:variance_upper_bound} and  \ref{prop:variance_lower_bound} show that $\H_n^{\xi}(\P)$ has at most surface-order variance growth.
		
\end{enumerate}	
	
\section{\bf Examples and applications}
\label{sec:applications}
	
	Before providing examples and applications of our general results, we briefly discuss the moment assumptions involved in our main theorems.
	\subsection{\bf Moments of point processes having fast decay of correlations}  \label{momsection}
	\Comment{Sect. 2.1 revised.  JY}
	We say that $\P$ has exponential moments if for all bounded Borel $B
	\subset \R^d$ and all $t \in\R^+$ we have
	\be \label{expmom}
	\E[t^{\P(B)}]<\infty \,.
	\ee
	Similarly,  say that $\P$ has all moments if for all bounded Borel $B \subset \R^d$ and all $k \in \N$, we have
	\be \label{allmom}
	\E[\P(B)^k] < \infty \,.
	\ee
	
	\paragraph{\bf Remarks}
	\begin{enumerate}[itemindent=0.8cm,labelwidth=*,label=(\roman*),leftmargin=0cm,topsep=-2ex]
		\item  \label{rem:exp_mom-i} The point process $\P$ has exponential moments whenever	$\sum_{k=1}^\infty \kappa_k t^k/k!<\infty$   for all
		$t \in\R^+$ with $\kappa_k$ as in~\eqref{eqn:corr_bounded} (cf. the expansion
		of the probability generating function of a random variable in terms
		of factorial moments~\cite[Proposition
		5.2.III.]{DVJ}). By \eqref{kapbound}
		an admissible point process having fast decay of correlations has exponential moments
		provided
		\begin{equation}\label {eqn:corr_bounded_exponential-clustering}
		\sum_{k=1}^\infty \frac{C_k t^k}{(k - 1)!} < \infty, \  t \in \R^+.
		\end{equation}
		Note that input of type (A2) has exponential moments
		since by \eqref{eqn:sum}, we have $C_k = O(k^{ak}), a \in [0, 1)$, making \eqref{eqn:corr_bounded_exponential-clustering}
		summable.
		For pairs $(\xi, \P)$  of type (A2) with  radius of stabilization bounded by $r_0 \in [1, \infty)$, by \eqref{eqn:xibd} the $p$-moment in \eqref{eqn:mom} is consequently controlled  by a finite exponential moment, i.e., for $x_1,...,x_{p'} \in W_n$,
		\begin{align}
 \sE_{x_1,\ldots,x_{p'}} \max\{|\xi(x_1, \P_n )|, 1\}^{p} &  \leq
		\E_{x_1,\ldots,x_{p'}} (\max\{\hat{c}r_0,1\}^{p\P(B_{r_0}(x_1))}). \label{eqn:mom-exp}
		\end{align}
		\noindent Finally,  if $\P$ has exponential moments
		under its stationary probability
		$\mP$, the same is true under $\mP_{x_1,\ldots,x_k}$ for
		$\alpha^{(k)}$ almost all $x_1,\ldots,x_k$ \footnote{\label{fn.palm-exponential-moments}
			Indeed, if
				$\E_{x_1,\ldots,x_k}[\rho^{\P(B_r(x_1))}]=\infty$ for
				$x_1,\ldots,x_k\in B'$ for some bounded $B'\in \R^d$
				such that  $\alpha^{(k)}(B'^k)>0$ then
				$\E_{x_1,\ldots,x_k}[\rho^{\P(B_r(x_1))}]\le \E_{x_1,\ldots,x_k}[\rho^{\P(B'_r)}]=\infty$
				with $B'_r=B'\oplus B_r(0)=\{y'+y:y'\in B',y\in B_r(0)\}$ the $r$-parallel set of $B'$. Integrating with respect to $\alpha^{(k)}$ in $B'^k$, by the Campbell formula $\E[(\P(B'_r))^k\rho^{\P(B'_r)}]=\infty$, which contradicts  the existence of exponential moments under~$\mP$.}.

		\item \label{rem:exp_mom-ii}
		For pairs  $(\xi, \P)$ of type (A1), the $p$-moment \eqref{eqn:mom} satisfies for $x_1,...,x_{p'} \in W_n$
		\begin{equation} 
		 \sE_{x_1,\ldots,x_{p'}}
		\max\{|\xi(x_1, \P_n )|, 1\}^{p} \leq
		\Bigl(\frac{\|h\|_{\infty}}{k}\Bigr)^p \sE_{x_1,\ldots,x_{p'}} [(\P(B_r(x_1)))^{(k-1)p}]. 	\label{eqn:mom-exp-U}
		\end{equation}
We next show that \eqref{eqn:mom-exp-U} may be controlled by moments of Poisson random variables.
For any Borel set $B \subset (\R^d)^k$, the definition of factorial moment measures gives $\alpha^{(k)}(B) \leq \kappa_k\Vol_{dk}(B)$. Since moments may be expressed as a linear combination of factorial moments, for $k \in \N$ and a bounded Borel subset $B \subset \R^d$, using \eqref{eqn:fact_mom} we have
		\be \label{Yog5}
		\E[(\P(B))^k] = \sum_{j=0}^k \stirling{k}{j} \alpha^{(j)}(B^j) \leq \kappa_k\sum_{j=0}^k \stirling{k}{j} \Vol_{jd}(B)^j = \kappa_k \E({\rm{Po}}(\Vol_{d}(B))^k),
		\ee
		where $\stirling{k}{j}$ stand for the {\em Stirling numbers of the second kind},  ${\rm{Po}}(\lam)$ denotes a Poisson random variable with mean $\lam$ and where  $\kappa_j$'s are non-decreasing in $j$.
		Thus by \eqref{kapbound}, an admissible point process having fast decay of correlations has all moments, as in \eqref{allmom}. If $\P$ has all moments under its stationary probability $\mP$, the same is true under $\mP_{x_1,\ldots,x_k}$ for
		$\alpha^{(k)}$ almost all $x_1,\ldots,x_k$ (by the same arguments as in Footnote~\ref{fn.palm-exponential-moments}).
		
	\end{enumerate}
	
	\subsection{\bf Examples of point processes having fast decay of correlations}
	\label{sec:ex_pp}
The notion of a stabilizing functional is well established  in the stochastic geometry literature but since the notion of fast decay of correlations for point processes
\eqref{eqn:clustering_condition} is less well studied,  we first establish that some well-known point processes enjoy this property.
 For more details on the first five examples, we refer to \cite{HKPV}.
	
	\subsubsection{Class A1 input}
	\label{sec:classA1}
		
	\paragraph{\bf Permanental input} The point process $\P$ is permanental  if its correlation
		functions are defined by $\rho^{(k)}(x_1,...,x_k) := {\rm{per}}(K(x_i, x_j))_{1 \leq i, j \leq k},$
		where the permanent of an $n \times n$ matrix $M$ is ${\rm{per}}(M) := \sum_{\pi \in S_n} \Pi_{i = 1}^n M_{i, \pi(i)} $, with $S_n$ denoting the
		permutation group of the first $n$ integers and $K(\cdot, \cdot)$ is the Hermitian kernel of a locally trace class integral operator $\cK: L^2(\R^d) \to L^2(\R^d)$  \cite[Assumption 4.2.3]{HKPV}. A kernel $K$ is {\em fast-decreasing} if 
		\be \label{fastdk}
		|K(x,y)| \leq \omega(|x-y|), \ \  x, y \in \R^d, \ee
		for some fast-decreasing $\omega : \mR^+ \to \mR^+$.  Lemma
	\ref{lem:clustering_PermPP} in Section \ref{sec:appendix} shows that if a stationary permanental
	point process has a fast-decreasing kernel as at \eqref{fastdk}, then it is an  admissible  point process having fast decay of correlations with decay function $\phi = \omega$ and with correlation decay constants satisfying
\be
\label{eqn:const_perm}
C_k:= kk! ||K||^{k-1}, c_k\equiv1,
\ee
where $||K||:= \sup_{x,y} |K(x,y)|$ and we can choose $\kappa_k = k! \|K\|^k$.  However, a trace	class permanental point process in general does not have exponential moments, i.e., the right-hand side of  \eqref{expmom} might be infinite for some bounded $B$ and $\rho$ large enough.~\footnote{This is because, the number of points of a (trace-class) permanental p.p. in a compact set $B$ is a sum of independent geometric random variables Geo$(1/(1+\lambda))$ where $\lambda$ runs over all eigenvalues of the integral operator defining the process truncated to $B$.}
	
The permanental point process with kernel $K$ may be represented as a Cox point process (see Section \ref{sec:addex}) directed by the random measure $\eta(B) := \int_B (Z_1(x)^2 + Z_2(x)^2) dx$, $B \subset \R^d$, where the intensity $Z_1(x)^2 + Z_2(x)^2$ is a sum of
i.i.d. Gaussian random fields with zero mean and covariance function $K/2$ \cite[Thm 6.13]{Shirai03}.
Thus mean zero Gaussian random fields with a fast decaying covariance function $K/2$ yield a permanental (Cox) point process with kernel $K$ and having fast decay of correlations.

\paragraph{\bf $\alpha$-Permanental point processes} See \cite[Section 4.10]{HKPV}, \cite{McCullagh06},  and \cite{Shirai03} for more details on this class of point processes which generalize permanental point processes. Given $\alpha \geq 0$ and a kernel $K$ which is Hermitian, non-negative definite and locally trace class, a point process $\P$ is said to be {\em $\alpha$-permanental} \footnote{In contrast to terminology in \cite{HKPV,Shirai03}, here we distinguish the two cases (i) $\alpha \geq 0$ ($\alpha$-permanental) and (ii) $\alpha \leq 0$ ($\alpha$-determinantal)} if its correlation functions satisfy
\be
\label{eqn:alpha_det}
\rho^{(k)}(x_1,\ldots,x_k) = \sum_{\pi \in S_k} \alpha^{k - \nu(\pi)}\prod_{i=1}^kK(x_i,x_{\pi(i)})
\ee
where $S_k$ stands for the usual symmetric group and $\nu(.)$ denotes the number of cycles in a permutation. The right hand side is the $\alpha$-permanent of the matrix $((K(x_i,x_j))_{i,j \leq k}$. The special cases  $\alpha = 0$ and $\alpha = 1$ respectively give the Poisson point process with intensity $K(\0,\0)$) and the permanental point process with kernel $K$. In what follows, we assume $\alpha = 1/m$ for $m \in \mathbb{N}$,  i.e. $1/\alpha$ is a positive integer. Existence of such $\alpha$-permanental point processes is guaranteed by \cite[Theorem 1.2]{Shirai03}. The property of these point processes most important to us is that an $\alpha$-permanental point process with kernel $K$ is a superposition of $1/{\alpha}$ i.i.d. copies of a permanental point process with kernel $\alpha K$ (see \cite[Section 4.10]{HKPV}). Also from definition \eqref{eqn:alpha_det}, we obtain
\[ \rho^{(k)}(x_1,\ldots,x_k) \leq \|K\|^k \alpha^k \sum_{\pi \in S_k}(\alpha^{-1})^{\nu(\pi)}, \]
and so we can take $\kappa_k = \prod_{i=0}^{k-1}(j\alpha + 1)\|K\|^k$ for an $\alpha$-permanental point process. The following result is a consequence of the upcoming Proposition \ref{prop:sum_clustering_pp} and the identity \eqref{eqn:const_perm} for decay constants of a permanental point process with kernel $\alpha K$.
\begin{proposition}
\label{eqn:clustering_alpha_per}
Let $\alpha = 1/m$ for some $m \in \mathbb{N}$ and let $\P_{\alpha}$ be the stationary $\alpha$-permanental point process with a kernel $K$ which is Hermitian, non-negative definite and locally trace class. Assume also that $|K(x,y)| \leq \omega(|x-y|)$ for some fast-decreasing $\omega$. Then $\P_{\alpha}$ is an admissible point process having fast decay of correlations with decay function $\phi = \omega$ and correlation decay constants $C_k = km^{1-k(m-1)} m!(k!)^m\|K\|^{km-1}, c_k = 1$.
\end{proposition}

\paragraph{\bf Zero set of Gaussian entire function (GEF)} A Gaussian entire function $f(z)$ is the sum $\sum_{j \geq 0} X_j \frac{z^j}{ \sqrt{j!}}$
where $X_j$ are i.i.d. with the standard normal density on the complex plane. The zero set $f^{-1}(\{0\})$ gives rise to the point process  $\P_{GEF}:=  \sum_{x \in f^{-1}(\{0\}) } \delta_x$ on $\R^2$.
	The point process $\P_{GEF}$ is an admissible point
	process having fast decay of correlations \cite[Theorem 1.4]{Nazarov12}, exhibiting local
	repulsion of points. Though $\P_{GEF}$  satisfies condition \eqref{phibd}, it is unclear whether \eqref{eqn:sum} holds. By \cite[Theorem 1]{Krishnapur06}, $\P_{GEF}(B_r(\0))$ has exponential moments.
	
\vskip.3cm	
	
\noindent{\em Moment conditions.} For $p \in [1, \infty)$, we show that the $p$-moment condition \eqref{eqn:mom} holds when $\xi$ is such that the pair $(\xi, \P_{GEF})$ is of class (A1). By \cite[Theorem 1.3]{Nazarov12}, given  $\P:= \P_{GEF}$, there exists constants $\tilde{D}_k$ such that
		\be \label{Yog1}
		\tilde{D}_k^{-1} \prod_{i < j} \min \{|y_i -y_j|^2,1\} \leq  \rho^{(k)}(y_1,\ldots,y_k) \leq \tilde{D}_k \prod_{i < j} \min \{|y_i -y_j|^2,1\}.
		\ee
		Recall from ~\cite[Lemma 6.4]{Shirai03} (see also \cite[Theorem 1]{Hanisch82}, \cite[Proposition 2.5]{Blaszczyszyn95}),
		that the existence of correlation functions of any point process  implies
		existence of reduced Palm correlation functions
		$\rho^{(k)}_{x_1,\ldots,x_p}(y_1,\ldots,y_k)$, which satisfy the following useful multiplicative identity:
		For Lebesgue a.e. $(x_1,\ldots,x_p)$ and $(y_1,\ldots,y_k)$, all distinct,
		\begin{equation}
		\label{eqn:palm_correlation}
		\rho^{(p)}(x_1,\ldots ,x_p)\rho^{(k)}_{x_1,\ldots,x_p}(y_1,\ldots,y_k) = \rho^{(p+k)}(x_1,\ldots ,x_p,y_1,\ldots,y_k).
		\end{equation}
		Combining \eqref{Yog1} and \eqref{eqn:palm_correlation}, we get for Lebesgue a.e. $(x_1,\ldots,x_p)$ and $(y_1,\ldots,y_k)$,  that
		\be \label{Yog3} \rho^{(k)}_{x_1,\ldots,x_p}(y_1,\ldots,y_k) \leq D_{p+k} \rho^{(k)}(y_1,\ldots,y_k),\ee
		where $D_{p+k} := \tilde{D}_{p+k}\tilde{D}_p\tilde{D}_k$.
		Thus we have shown  there exists constants $D_j, j \in \mathbb{N},$ such that for any bounded Borel subset $B$, $k \in \mathbb{N}$ and Lebesgue a.e. $(x_1,\ldots,x_p) \in (\R^{d})^p$,
		we have
		\be
		\label{eqn:palm_mom_gef}
		\E^!_{x_1,\ldots,x_p}(\P^{(k)}(B^k)) \leq D_{p+k} \E(\P^{(k)}(B^k)).
		\ee	
		By \eqref{eqn:mom-exp-U}, \eqref{eqn:palm_mom_gef}, and \eqref{Yog5} in this order, along with stationarity of $\P_{GEF}$, we have for any $p \in [1, \infty)$,
		\begin{eqnarray}
		& & \sup_{1 \leq n \leq \infty} \sup_{1 \leq p' \leq \lfloor p \rfloor } \sup_{x_1,...,x_{p'} \in W_n} \sE_{x_1,\ldots,x_{p'}} \max\{|\xi(x_1, \P_n )|, 1\}^{p}  \no \\
		&\leq &  \left(\frac{\|h\|_{\infty}}{k}\right)^p\kappa_{(k-1)p}D_{kp} \E[({\rm{Po}}(\Vol_{d} (B_r(\0)))+p)^{(k-1)p}] < \infty,
		\label{eqn:mom-gef}
		\end{eqnarray}
		where as before ${\rm{Po}}(\lam)$ denotes a Poisson random variable with mean $\lam$. Thus the $p$-moment condition \eqref{eqn:mom} holds
		for pairs $(\xi, \P_{GEF})$ of class (A1) for all $p \in [1, \infty)$.
	
\subsubsection{Class A2 input}
	\label{sec:classA2}
	
	\paragraph{\bf Determinantal input}
	The point process $\P$ is determinantal  if its correlation
	functions are defined by $\rho^{(k)}(x_1,\ldots,x_k) = {\rm{det}}(K(x_i, x_j))_{1 \leq i, j \leq k},$
	where $K(\cdot, \cdot)$ is again the Hermitian kernel of a locally trace class integral operator $\cK: L^2(\R^d) \to L^2(\R^d)$. Determinantal point processes exhibit local repulsivity and their structure is preserved when restricting  to subsets of $\R^d$ and as well as when considering their reduced Palm versions. These facts facilitate our analysis of determinantal input;
 the Appendix (Section \ref{sec:appendix}) provides  lemmas further illustrating their tractability.  If a stationary determinantal point process has a fast-decreasing kernel as at \eqref{fastdk}, then Lemma \ref{lem:clustering_DPP} in Section \ref{sec:appendix} shows that it is an admissible point process having fast decay of correlations satisfying   \eqref{eqn:sum} with decay function $\phi = \omega$, with $\omega$ as at \eqref{fastdk}, and correlation decay constants
\be
\label{eqn:const_det}
 C_k:= k^{1 +(k/2)} ||K||^{k-1}, c_k\equiv1.
\ee
Consequently, $\phi$  satisfies the requisite exponential decay \eqref{phibd} whenever $\omega$ itself satisfies \eqref{phibd}. 
	
	The Ginibre ensemble of eigenvalues of $N \times N$ matrices with independent standard complex Gaussian entries is a leading example of a determinantal point process. The limit of the Ginibre ensemble as $N \to \infty$ is the Ginibre point process (or the infinite Ginibre ensemble), here denoted $\P_{GIN}$. It is the prototype of a stationary determinantal point process and has  kernel
      $$
		K(z_1,z_2) := \exp(z_1 \bar{z_2}) \exp\left( - \frac{  |z_1|^2 + |z_2|^2 } {2} \right) = \exp\left(i {\rm{Im}}(z_1 \bar{z_2}) - \frac{  |z_1 - z_2|^2} {2} \right) , \ \ z_1, \ z_2 \in \mathbb{C}.
		$$
	More generally, for $0 < \beta \leq 1,$ the $\beta$-Ginibre (determinantal) point process (see \cite{Goldman10}) has kernel
	$$
	K_{\beta}(z_1,z_2) := \exp(\frac{1}{\beta}z_1 \bar{z_2}) \exp\left( - \frac{ |z_1|^2 + |z_2|^2} { 2\beta } \right), \ \ z_1, \ z_2 \in \mathbb{C}.
	$$
	When  $\beta = 1$, we obtain  $\P_{GIN}$ and as  $\beta \to 0$ we
	obtain the Poisson point process. Thus the $\beta$-Ginibre  point
	process interpolates between the Ginibre and Poisson point processes.
Identifying the complex plane with $\R^2$ we see that all  $\beta$-Ginibre point processes are  admissible point
	processes having fast decay of correlations satisfying \eqref{eqn:sum} and \eqref{phibd}.
	
	\vskip.1cm
	
		\noindent{\em Moment Conditions.} Let $p \in [1, \infty)$  and let $\P$ be a stationary determinantal point process with a continuous and fast-decreasing kernel.
		We now show that the $p$-moment condition \eqref{eqn:mom} holds for pairs $(\xi,\P)$ of class (A1) or (A2), provided $\xi$ has a deterministic
		radius of stabilization, say $r_0 \in [1, \infty)$. First,  for all $(x_1,\ldots,x_p) \in (\mR^d)^p$, all increasing  $F : \mathbb{N} \to \mR^+$ and all
		bounded Borel sets $B$ we have \cite[Theorem 2]{Goldman10}
				$$ \E^!_{x_1,\ldots,x_p}(F(\P(B))) \leq  \E(F(\P(B))).$$
		%
		Thus using \eqref{eqn:mom-exp},
		the above inequality and stationarity of $\P$, we get that for any bounded stabilizing score function $\xi$ of class (A2),
		\be
		\label{eqn:mom-dpp}
		\sup_{1 \leq n \leq \infty} \sup_{1 \leq p' \leq \lfloor p \rfloor } \sup_{x_1,...,x_{p'} \in W_n} \sE_{x_1,\ldots,x_{p'}} \max\{|\xi(x_1, \P_n )|, 1\}^{p} \leq \E(\max\{\hat{c}r_0,1\}^{p\P(B_{r_0}(\0))+p^2}) < \infty.
		\ee
		The finiteness of the last term follows from the fact that determinantal input considered here is of class (A2) and, by Remark~\ref{rem:exp_mom-i} at the beginning of  Section~\ref{momsection}, such input has finite exponential moments.
		
\paragraph{\bf $\alpha$-Determinantal point processes}		

Similar to permanental point processes, we  generalize determinantal point processes to include their $\alpha$-determinantal versions, by requiring that the correlation functions satisfy \eqref{eqn:alpha_det} for some $\alpha \leq 0$. In what follows, we shall assume that $\alpha = -1/m, m \in \mathbb{N}$. Existence of such $\alpha$-determinantal point processes again follows from \cite[Theorem 1.2]{Shirai03}. Likewise, 
an $\alpha$-determinantal point process with kernel $K$ is a superposition of $-1/\alpha$ i.i.d. copies of a determinantal point process with kernel $-\alpha K$ (\cite[Section 4.10]{HKPV}). By \cite[Proposition 4.3]{Shirai03}, we can take $\kappa_k = K(\0,\0)^k$ for an $\alpha$-determinantal point process. Analogously to Proposition \ref{eqn:clustering_alpha_per}, the next result follows from Proposition \ref{prop:sum_clustering_pp} below and the identity  \eqref{eqn:const_det} for correlation decay constants of a determinantal point process with kernel $-\alpha K$.
\begin{proposition}
	\label{prop:clustering_alpha_det}
	Let $\alpha = -1/m$ for some $m \in \mathbb{N}$ and $\P_{\alpha}$ be the stationary $\alpha$-determinantal point process with a kernel $K$ which is Hermitian, non-negative definite and locally trace class. Assume also that $|K(x,y)| \leq \omega(|x-y|)$ for some fast-decreasing function $\omega$. Then $\P_{\alpha}$ is an admissible point process having fast decay of correlations with decay function $\phi = \omega$ and correlation decay constants $C_k = m^{1-k(m-1)} m!{K(\0,\0)}^{k(m-1)}k^{1+(k/2)}\|K\|^{k-1}, c_k = 1$. Further if $w$ satisfies   \eqref{phibd}, then $\P_{\alpha}$ is an admissible input of type (A2).
\end{proposition}
From \eqref{eqn:mom-dpp} and \eqref{eqn:mom_bd_superp}, we have that for $\P_{\alpha}, -1/\alpha \in \N$ as above, and for any bounded stabilizing score function $\xi$ of class (A2),
\be
\label{eqn:mom-alphadpp}
\sup_{1 \leq n \leq \infty} \sup_{1 \leq p' \leq \lfloor p \rfloor } \sup_{x_1,...,x_{p'} \in W_n} \sE_{x_1,\ldots,x_{p'}} \max\{|\xi(x_1, \P_n )|, 1\}^{p} < \infty.
\ee
\paragraph{\bf Rarified Gibbsian input}  Consider the class ${\bf
		\Psi}$ of Hamiltonians consisting of pair potentials without  negative part, area interaction Hamiltonians,
	hard core Hamiltonians, and  potentials generating a truncated Poisson point process (see \cite{SY} for  further details of such potentials).
	For $\Psi \in {\bf \Psi}$ and $\beta \in (0, \infty)$, let $\P^{\beta \Psi}$ be the Gibbs point process having Radon-Nikodym derivative
	$\exp(- \beta \Psi( \cdot) )$ with respect to a reference homogeneous Poisson point process on $\R^d$ of intensity
	$\tau \in (0, \infty)$.
	There is a range of inverse temperature and activity parameters
	($\beta$ and $\tau$) such that $\P^{\beta \Psi}$ has fast decay of correlations;
	see the introduction to	Section 3 and \cite{SY} for further details. These rarified
	Gibbsian point processes are admissible  point processes having fast decay of correlations and  satisfy the input conditions~\eqref{eqn:sum} and~\eqref{phibd} of
	class (A2). Setting $\xi(.,.) \equiv 1$ in  Lemma 3.4 of \cite{SY} shows that \eqref{eqn:clustering_condition} holds with
	$C_k$ a scalar multiple of $k$ and $c_k$ a constant.
	
	\subsubsection{Additional input examples}
	\label{sec:addex}

\vskip.1cm

Here we provide a non-exhaustive list of examples of admissible point processes having fast decay of correlations.

\vskip.1cm
	
\noindent{\bf Cox point processes}.  A point process $\P$ is said to be a Cox point process with (random) {\em intensity measure} $\Lam(.)$ if conditioned on $\Lam(.)$, $\P$ is a Poisson point process with intensity measure $\Lam(.)$. We shall assume that the measure $\Lam(.)$ has a (random) density $\lam(.)$ with respect to the Lebesgue measure, called the {\em intensity field}, i.e., $\Lam(B) = \int_B \lam(x) \md x$ for all Borel sets $B$. In such a case, the correlation functions of the Cox point process $\P$ are given by $\rho^{(k)}(x_1,\ldots,x_k) = \E(\prod_{i=1}^k\lam(x_i))$ for $x_1,\ldots,x_k$ distinct. Hence, fast decay of correlations of the random field $\{\lam(x)\}_{x \in \mR^d}$  induce fast decay of correlations of the point process $\P$. Also, if $\lam(.)$ is a stationary random field, then $\P$ is a stationary point process. We have already seen one class of Cox point processes having fast decay of correlations in the permanental point process and we shall see below another class in thinned Poisson point processes. Another tractable class of Cox point processes, called the {\em shot-noise Cox point process} and studied in \cite{Moller03}, includes examples of admissible point processes having fast decay of correlations.

\vskip.3cm
	
\noindent{\bf  Finite-range dependent point processes.} Correlation functions of these processes, when assumed locally finite,  (trivially) have fast decay
with the decay function $\phi(s) = 0$ for all $s \in (r_0, \infty)$, for some $r_0 \in (0, \infty)$, where $r_0$ is the range of dependence, and with constants $c_k=1$.	Whether correlation decay constants $C_k$  satisfy  condition~\eqref{eqn:sum} depends
on the local properties of the correlation functions.~\footnote{A point process consisting of sufficiently heavy-tailed random number of points distributed
independently (and say uniformly) in each hard-ball of a hard-core (say Mat\'ern) model   will  not satisfy~\eqref{eqn:sum}.}
Examples of finite range dependent point processes include  perturbed lattices \cite{Blaszczy14}, Mat\'ern cluster point processes or Mat\'ern hard-core point processes with finite dependence radius.

\vskip.3cm
	
\noindent{\bf Thinned Poisson point processes.} Suppose $\P$ is a Poisson point process, $\xi(.,\P) \in \{0,1\}$, and consider the thinned point process $\tilde{\P} := \sum_{x \in \P}\xi(x,\P)\delta_x$. If  $\xi$ stabilizes exponentially fast then Lemma 5.2 of \cite{Baryshnikov05} shows that $\tilde{\P}$
	is an admissible point process having fast decay of correlations. 
This set-up includes finite-range dependent point processes $\P$ as well as  Mat\'ern cluster and Mat\'ern hard-core point process with exponentially decaying dependence radii. Tractable procedures generating  thinnings of Poisson point processes are in \cite{Baccelli12}.
	
\vskip.3cm
	
\noindent{\bf Thinned general point processes.} Suppose $(\xi, \P)$
is an admissible pair of class A2 and suppose further that $\xi(.,\P) \in \{0,1\}$.
Then $\mu_n^{\xi}$ is a thinned point process and 
the correlation functions of $\mu_n^{\xi}$ coincide with $m^{(1,\ldots,1)}(x_1,\ldots,x_k;\infty)$  
in~\eqref{eqn:mixedmomentn}.  Theorem~\ref{prop:clustering_gen} implies
these functions have fast decay of correlations and hence  $\mu_n^{\xi}$ is an admissible point process having fast decay of correlations. For similar examples and generalizations, termed  {\em generalized shot-noise Cox point process}, see \cite{Moller05}.

\vskip.3cm

\noindent{\bf Superpositions of i.i.d. point processes.} Apart from thinning another natural operation on point processes generating new point processes consists of  independent superposition. We show that this operation preserves fast decay of correlations. 

Let  $\P_1,\ldots,\P_m, m \in \N,$ be i.i.d. copies of an admissible point process $\P$ with correlation functions $\rho$ and having fast decay of correlations. Let $\rho_0$ denote the correlation functions of the point process $\P_0 := \cup_{i=1}^m\P_i$. For any $k \geq 1$ and distinct $x_1,\ldots,x_k \in \R^d$ the following relation holds
\be
\label{eqn:corr_sum_pp}
\rho_0^{(k)}(x_1,\ldots,x_k) = \sum_{\sqcup_{i=1}^mS_i = [k]}\prod_{i=1}^m\rho(S_i),
\ee
where $\sqcup$ stands for disjoint union and where we abbreviate $\rho^{(|S_i|)}(x_j : j \in S_i)$ by $\rho(S_i)$.
Here $S_i$ may be empty, in which case we set $\rho(\emptyset) = 1$.
It follows from \eqref{eqn:corr_sum_pp} that $\P_0$ is an admissible point process with intensity $m\rho^{(1)}(\0)$. Further, we can take $\kappa_{k}(\P_0) = (\kappa_k)^mm^k$.
The proof of the next proposition, which shows that $\P_0$ has fast decay of correlations, is in the Appendix.

\begin{proposition}
\label{prop:sum_clustering_pp}
Let $m \in \N$ and $\P_1,\ldots,\P_m$ be i.i.d. copies of an admissible point process $\P$ having fast decay of correlations with decay function $\phi$ and correlation decay constants $C_k$ and $c_k$. Then $\P_0 := \cup_{i=1}^m\P_i$ is an admissible point process having fast decay of correlations with decay function $\phi$ and correlation decay constants $m^k m! (\kappa_k)^{m-1} C_k$ and $c_k$. Further, if $\P$ is admissible input of type (A2) with $\kappa_k \leq \lam^k$ for some $\lam \in (0,\infty)$, then $\P_0$ is also admissible input of type (A2).
\end{proposition}
We have already used this proposition in the context of fast decay of correlations of $\alpha$-permanental and determinantal point processes.

\subsection{\bf Applications}
\label{sec:applications2}

	Having provided examples of admissible point processes, we  now use Theorems \ref{prop:variance_upper_bound} and \ref{thm:main1}
 to deduce 	the limit theory for geometric and topological statistics of these point processes.  Examples include statistics arising in combinatorial and  differential topology,  integral geometry, and  computational geometry.  When the input is one of the Gibbs point processes described in Section \ref{sec:ex_pp}, then these results can be deduced from \cite{SY,XY}.  The examples are not exhaustive and indeed include functionals in stochastic geometry already discussed in e.g. \cite{Baryshnikov05, PY5}.
 There are further applications to (i) random packing models on input having fast decay of correlations (extending \cite{PY2}), (ii) statistics of percolation models
(extending e.g. \cite{LP, PY1}), and (iii) statistics of extreme points of input having fast decay of correlations (extending \cite{BX1,XY}). Details are left to the reader.

  We shall need to assume  moment bounds and variance lower bounds in the upcoming applications in Sections \ref{sec:examples}- \ref{sec:Boolean}.
  In view of \eqref{eqn:mom-gef} and \eqref{eqn:mom-alphadpp},  the moment conditions are valid for input pairs $(\xi, \P)$ of class (A1) when $\P$ is either $\alpha$-determinantal with $-1/\alpha \in \N$ or $\P_{GEF}$. Similarly, in view of \eqref{eqn:mom-alphadpp}, the moment conditions are valid for input pairs $(\xi, \P)$ of class (A2) when $\xi$ has a bounded radius of stabilization and when $\P$ is $\alpha$-determinantal with $-1/\alpha \in \N$.


	
	\subsubsection{Statistics of simplicial complexes}
	\label{sec:examples} 	
	
	A nonempty family $\Del$ of finite subsets of a set $V$ is {\em an abstract simplicial} complex if $\cY \in \Del$ and $\cY_0 \subset \cY$ implies that $\cY_0 \in \Del$. Elements of $\Del$ are called {\em faces/simplices} and the {\em dimension of a face} is one less than its cardinality. The $0$-dimensional faces are {\em vertices}. The collection of all faces of $\Del$ with dimension less than $k$ is a sub-complex called the {\em $k$-skeleton} of $\Del$ and denoted by $\Del^{\leq k}$. The $1$-skeleton of a simplicial complex is a graph whose vertices are $0$-dimensional faces and whose edges are $1$-dimensional faces. The simplicial complex, or `complex' for short, represents a combinatorial generalization of a graph and is a fundamental object in combinatorial as well as computational topology \cite{EH}

	Given a finite point set $\cX$ in $\mR^d$ (or generally, in a metric space) there are various ways to define a complex that captures some of the geometry/topology of $\cX$. One such complex is the \Cech complex.  Recall that if $\cX \subset\mR^d$ is a finite set of points and $r \in (0, \infty)$, then the	\Cech complex of radius $r$ is the abstract complex
$$\C (\cX ,r) := \{ \sigma\subset\cX: \bigcap_{x\in\sigma} B_r(x)\neq\emptyset \}.$$
	By the nerve theorem \cite[Theorem 10.7]{Bjorner95}, the \Cech complex  is homotopy equivalent (in particular, same topological invariants) to the classical germ-grain model
	\begin{equation}
	\label{defn:boolean}
	\cCB(\cX,r) := \bigcup_{x \in \cX}B_{x}(r).
	\end{equation}
	The $1$-skeleton of the \Cech complex, $\C(\cX,r)^{\leq 1},$ $\X$ random, is the  {\em random geometric graph} \cite{Penrose03}, denoted by $G(\cX,r)$. 
	One can study many geometric or topological statistics similar to those described below for the \Cech complex for other geometric complexes (for example, see \cite[Section 3.2]{Zomorodian12}) or geometric graphs. Indeed, motivated by problems in topological data analysis, random geometric complexes on Poisson or binomial  point processes were studied in \cite{Kahle11} and later were extended to stationary point processes in \cite{AY1}.
	
	We next establish the limit theory for statistics of random \Cech complexes. {\em The central limit theorems are applicable whenever the input $\P$ is either $\alpha$-determinantal ($|\alpha| = \frac{1}{m}, m \in \N$) with kernel as at \eqref{fastdk}, $\P_{GEF}$,  or rarified Gibbsian input.} In all that follows we fix $r \in (0, \infty)$.
	
	\vskip.5cm
	
	\noindent{\bf Simplex counts or clique counts}.	\label{ex:simplex_count}
	Let $\Gamma$ be a complex on $k$-vertices such that $\Gamma^{\leq 1}$ is a connected graph.
	For $x \in \mR^d$ and $\x := (x_1,\ldots,x_{k-1}) \in (\mR^d)^{k-1}$, let
	\[ h^{\Gamma}(x,\x) := \1[\C(\{x,x_1,\ldots,x_{k-1}\},r) \cong \Gamma],\]
	where  $\cong$ denotes simplicial isomorphism. For an admissible point process $\P$ as in Definition \ref{def.admissiblePP},
	we put
$$ \gamma^{(k)}(x,\P) := \frac{1}{k!}\sum_{\x \in (\P \cap B_{r}(x) )^{k-1}}h^{\Gamma}(x,\x),$$
that is $(\gamma^{(k)},\P)$ is an admissible pair of type (A1). If $\Gamma$ denotes  the $(k-1)$-simplex, then $H^{\gamma^{(k)}}_n(\P)$ is the number of $(k-1)$-simplices in $\C(\P_n,r)$  and for $k = 2$, $H^{\gamma^{(k)}}_n(\P)$ is the {\em edge count} in the random geometric graph $G(\P_n,r)$. Theorem 3.4  of \cite{AY1}  establishes expectation asymptotics for  $n^{-1} H^{\gamma^{(k)}}_n(\P)$ for stationary input.  The next result
establishes variance asymptotics and asymptotic normality of $n^{-1} H^{\gamma^{(k)}}_n(\P)$. It is an immediate consequence of Theorem \ref{prop:variance_upper_bound}(ii) and  Theorem \ref{thm:main1}.  Let  $\sigma^2(\gamma^{(k)})$ be as at \eqref{eqn:sigdef},  with $\xi$ put to be $\gamma^{(k)}$.

\begin{theorem} \label{thm:main1simplex}  Let $k \in \N$.
If $\P$ is an admissible point process having fast decay of correlations as at \eqref{eqn:clustering_condition}
and the pair $(\gamma^{(k)},\P)$ satisfies the moment condition \eqref{eqn:mom} for all $p \in (1,\infty)$, then
$$ \lim_{n \to \infty} n^{-1} \Var H^{\gamma^{(k)}}_n(\P) = \sigma^2(\gamma^{(k)}).$$
Additionally,  if $\Var H^{\gamma^{(k)}}_n(\P) = \Omega(n^{\nu} )$ for some $\nu \in (0, \infty)$, then as $n \to \infty$
%
$$		 (\Var H^{\gamma^{(k)}}_n(\P))^{-1/2} \left(H^{\gamma^{(k)}}_n(\P) - \E H^{\gamma^{(k)} }_n(\P)\right) \stackrel{\cD}{\to}  N.
$$
\end{theorem}

Up to now, the central limit theorem for clique counts has been restricted to binomial or Poisson input, cf.
\cite{DFR,  EicTha, Lachieze15, Penrose03, Reitzner13}.   Theorem  \ref{thm:main1simplex} shows that asymptotic normality holds for more general input.

\vskip.5cm	

\noindent{\bf Edge lengths.}	
	For $x,y \in \mR^d$, let $h(x,y) :=  |x - y| {\bf 1}[|x-y| \leq r ] .$
	The $U$-statistic
	$$
	\xi^L(x,\X) := \frac{1}{2}\sum_{y \in \X \cap B_{r}(x) }h(x,y),
	$$
	is of  generic type \eqref{eqn:score_U_statistic} and $H^{\xi^L}_n(\P)$ is the {\em total edge length} of the geometric graph $G(\P_n,r)$.  The following is an immediate
	consequence of Theorems \ref{prop:variance_upper_bound} and \ref{thm:main1}.

	\begin{theorem}
		\label{thm:main1edge}
		For any  admissible point process $\P$ having fast decay of correlations as at \eqref{eqn:clustering_condition}
 with the pair $(\xi^L,\P)$ satisfying the moment condition \eqref{eqn:mom} for all $p \in (1,\infty)$, we have
			$$
			| n^{-1} \E H^{\xi^L}_n(\P) - \E_{\0}\xi^L(\0, \P) \rho^{(1)}(\0) | = O(n^{-1/d}), \, \mbox{and} \,
\lim_{n \to \infty} n^{-1} \Var H^{\xi^L}_n(\P) = \sigma^2(\xi^L).
			$$
Moreover, if $\Var H^{\xi^L}_n(\P) = \Omega(n^{\nu} )$ for some $\nu \in (0, \infty)$  then as $n \to \infty$
$$
		(\Var H^{\xi^L}_n(\P))^{-1/2}  \left( \H^{\xi^L}_n(\P) - \E H^{\xi^L}_n(\P) \right) \tod N.
		$$
	\end{theorem}
The central limit theory for $H^{\xi^L}_n(\cdot)$ for Poisson or binomial  input is a consequence of \cite{DFR,  EicTha, Lachieze15, Penrose03, Reitzner13}.
Theorem \ref{thm:main1edge} shows that $H^{\xi^L}_n(\cdot)$ still satisfies a central limit theorem when Poisson and binomial input is replaced by more general input having fast decay of correlations.

	\vskip.5cm
	
	\noindent{\bf Degree  counts}. 
	Define the (down) degree of a $k$-simplex to be the number of $k$-simplices with which the given simplex has a common $(k-1)$-simplex.
	For $x \in \mR^d$ and $\x := (x_1,\ldots,x_{k+1}) \in (\mR^d)^{k+1}$, define the indicator that $(x_1,\ldots,x_k)$ is the common $(k-1)$ simplex between two $k$-simplices:
$$ h(x,\x) := {\bf 1}[C(x,x_1,\ldots,x_k) {\rm{ \ is \ a \ k-simplex}}] \ {\bf 1}[C(x_1,\ldots,x_{k+1}) {\rm{\ is \ a \ k-simplex}}].$$
%
	The total (down) degree of order $k$ of a complex is the sum of the degrees of the constituent $k$-simplices. Consider the $U$-statistic
$$
\xi^{(k)}(x,\X) := \frac{1}{(k+2)!}\sum_{\x \in (\X \cap B_{r}(x) )^{k+1}}h(x,\x)
$$
which is of generic type \eqref{eqn:score_U_statistic}. Then $H^{\xi^{(k)}}_n(\P)$ is the total down degree (of order $k$) of the geometric complex $\C(\P_n,r)$.  Note that $(\xi^{(k)}, \P)$ is of type (A1) whenever $\P$ is admissible in the sense of Definition \ref{def.admissiblePP}. Theorems \ref{prop:variance_upper_bound} and \ref{thm:main1} yield the following limit theory for $H^{\xi^{(k)}}_n(\P)$.
	\begin{theorem}	\label{thm:main1degree}  Let $k \in \N$.
	For any  admissible point process $\P$ having fast decay of correlations as at \eqref{eqn:clustering_condition}
with the pair $(\xi^{(k)},\P)$ satisfying the moment condition \eqref{eqn:mom} for all $p \in (1,\infty)$, we have
$$
| n^{-1} \E H^{\xi^{(k)}}_n(\P) - \E_{\0}\xi^{(k)}(\0, \P) \rho^{(1)}(\0)| = O(n^{-1/d}), \, \mbox{and}
\lim_{n \to \infty} n^{-1} \Var H^{\xi^{(k)}}_n(\P) = \sigma^2(\xi^{(k)}).
$$
Moreover if $\Var H^{\xi^{(k)}}_n(\P) = \Omega(n^{\nu} )$ for some $\nu \in (0, \infty)$  then as $n \to \infty$
$$
 ( \Var H^{\xi^{(k)}}_n(\P))^{-1/2} \left( H^{\xi^{(k)}}_n(\P) - \E H^{\xi^{(k)} }_n(\P) \right) \tod N.
$$
\end{theorem}
	
	
\subsubsection{Morse critical points}
	\label{sec:isol_u_stat}
	
	Understanding the topology of a manifold via smooth functions on the manifold is a classical topic
	in differential topology known as Morse theory \cite{Milnor63}. Among the various extensions of Morse theory to non-smooth functions, the one of interest to us is the `min-type' Morse theory developed in
	\cite{Gershkovich97}. This theory was exploited to study the topology
	of random \v{C}ech complexes on Poisson and binomial  point processes \cite{Bobrowski14} and later on stationary point processes \cite{AY1}.
	
	As above, $\X \subset \R^d$ denotes a locally finite point set and  $k \in \N$. Given $\z \in (\mR^d)^{k+1}$ in general position, let $C(\z)$ denote the center of the unique $(k-1)$-dimensional sphere  containing the points
	of $\z$ and let $R(\z)$ be the radius of this unique ball. The point $C(\z)$ is an index $k$ critical point iff
	$$
	C(\z) \in \mathring{\rm{co}}(\z) \, \, \, \, \mbox{and} \, \, \, \,   \X( B_{R(\z)}(C(\z))) = \z,
	$$
	where $\rm{co}(\z)$ is the convex hull of the points comprising $\z$ and $\mathring{A}$ stands for the interior of a Euclidean set $A$. We are interested in critical points $C(\z)$
	distant at most $r$ from $\X$, that is to say $R(\z) \in (0,r]$.
	To this end,  for $(x, \x) \in (\R^{d})^{k + 1}$ define
	$$g_r(x,\x) := \1[C(x,\x) \in \mathring{\rm{co}}(x,\x)]\1[R(x,\x) \leq r] \, ; \, \, Q(x,\x) := B_{R(\x)}(C(\x))\1[R(x,\x) \leq r].$$
	Thus $Q(x,\x) \subset B_{2r}(x)$. Now, for $\x \in (\R^{d})^k$, we set $h(x,\x) := (k+1)^{-1}g_r(x,\x)$, and, in keeping with \eqref{eqn:score_U_statistic},  define	the Morse score function
	\begin{equation}
	\label{eqn:isolated_U_statmorse}
	\xi^M(x,\X):= \frac{1}{k!}\sum_{\x \in (\X \cap B_{2r}(x))^{k}}h(x,\x) \1[\X(Q(x,\x) \setminus \{x,\x\}) = 0].
	\end{equation}
Then $\xi^M$ satisfies the power growth condition ~\eqref{eqn:xibd} and
	is of type (A2) (and nearly of type (A1)), with the understanding that $h(x,x_1,\ldots,x_k) = 0$ whenever $\max_{1 \leq i \leq k} |x_i-x|  \in (2r, \infty)$.	The statistic $H^{\xi^M}(\P)$ is simply the  number $N_k(\P_n,r)$ of index $k$ Morse critical points generated by $\P_n$ which are within a distance $r$ of $\P_n$, whereas $\mu_n^{\xi^M}$ is the random measure generated by index $k$ critical points. Index $0$ Morse critical points are trivially the points of $\X$ and so in this case $N_0(\X) = {\rm{card}}(\X)$. Thus we shall be interested in asymptotics for  Morse critical points only of index $k \in \N$.

The next result establishes variance asymptotics and asymptotic normality of  $N_k(\P_n,r)$, valid when $\P$  is of class (A2). It is an immediate consequence of Theorem \ref{prop:variance_upper_bound}(ii),  Theorem \ref{thm:main1}, and the fact that for input $\P$ of class (A2),  $(\xi^M,\P)$ is an input pair of class (A2).  Let  $\sigma^2(\xi^{M})$ be as at \eqref{eqn:sigdef},  with $\xi$ put to be $\xi^{M}$.
	
	\begin{theorem}
		\label{thm:main1Morse} For all $k \in \{1,\ldots,d\}$ and class (A2) input $\P$ with the pair $(\xi^M,\P)$ satisfying the moment condition \eqref{eqn:mom} for all $p \in (1,\infty)$, we have
		$$
		\lim_{n \to \infty} n^{-1} \Var  N_k(\P_n,r) = \sigma^2(\xi^{M}).
		$$
		Moreover  if $\Var N_k(\P_n,r) = \Omega(n^{\nu} )$ for some $\nu \in (0,\infty)$, then as $n \to \infty$
$$		(\Var N_k(\P_n,r))^{-1/2} \left( N_k(\P_n,r) - \E N_k(\P_n,r) \right) \tod N.
		$$
	\end{theorem}
	
	\noindent{\bf Remarks.}
	\begin{enumerate}[itemindent=0.8cm,labelwidth=*,label=(\roman*),leftmargin=0cm,topsep=-2ex]
		\item Theorem 5.2 of  \cite{AY1}  establishes expectation asymptotics for  $n^{-1} N_k(\P_n,r)$ for stationary input, though without a rate of convergence and \cite{Bobrowski14} establishes a central limit theorem but only for the case of Poisson or binomial  point processes.
		
		\item The Morse inequalities relate the Morse critical points (local functionals) to the Betti numbers (global functionals) of the Boolean model and in particular, imply that the changes in the homology of the Boolean model $\cCB(\P,r)$ occurs at radii $r = R(\x)$ whenever $C(\x)$ is a Morse critical point. A trivial consequence is that the $k$th Betti number of $\C(\P_n,r)$ is upper bounded by $N_k(\P_n,r)$.
		
		\item Examples of similar score functions satisfying a modified version of (A1) (similar to \eqref{eqn:isolated_U_statmorse}) include
		component counts of random geometric graphs \cite[Chapter 3]{Penrose03},  simplex counts of degree $k$ in the \Cech complex, and simplex counts in
		either an alpha complex \cite[Sec 3.2]{Zomorodian12} or an appropriate discrete Morse complex on $\P_n$ \cite{Forman02}.
	\end{enumerate}
	
	\subsubsection{Statistics of germ-grain  models}
	\label{sec:Boolean}
	
We furnish two more applications of Theorem \ref{thm:main1} when $\xi$ has a deterministic radius of stabilization. The applications concern the germ-grain model, a classic model in stochastic geometry \cite{Schneider08}.
	
	\vskip.3cm
	
	\noindent{\bf $k$-covered region of the germ-grain model.}	The following is a statistic of interest in coverage processes \cite{Hall88}. For locally-finite $\X \subset \R^d$ and  $x \in \X$, define the score function
	
	$$ \beta^{(k)}(x,\X) := \int_{y \in B_r(x)}\frac{\1[\X(B_r(y)) \geq k]}{\X(B_r(y))} \md y.
	$$
	Clearly, $\beta^{(k)}$ is an exponentially stabilizing score function as in Definition 1.1 
	with stabilization radius $2r$. Define the $k$-covered region of the germ-grain model  $\cCB(\P_n,r)$
	 at \eqref{defn:boolean} by $\cCB^k(\P_n,r) = \{y : \P_n(B_r(y)) \geq k\}$.  Thus $H_n^{\beta^{(k)}}(\P)$ is the volume of $\cCB^k(\P_n,r)$. When $k = 1$, $H_n^{\beta^{(k)}}(\P)$ is the volume of the germ-grain model having germs in $\P_n$. Clearly  $\beta^{(k)}$ is bounded by the volume of a radius $r$ ball and so $\xi$ satisfies the power growth condition ~\eqref{eqn:xibd}.
The following is an immediate
consequence of Theorems \ref{prop:variance_upper_bound} and \ref{thm:main1} and the fact that if $\P$ is of class (A2) then the input pair $(\beta^{(k)}, \P)$ is also of class (A2).
\begin{theorem}
\label{thm:main1Boolean}
For all $k \in \N$ and any point process $\P$ of class $(A2)$ with the pair $(\beta^{(k)},\P)$ satisfying the moment condition \eqref{eqn:mom} for all $p \in (1,\infty)$, we have
$$ | n^{-1} \E\Vol_{d} ( \cCB^k(\P_n,r))  - \E_{\0}\beta^{(k)}(\0, \P) \rho^{(1)}(\0) | = O(n^{-1/d}),$$
and
$$ \lim_{n \to \infty} n^{-1} \Var\Vol_{d} ( \cCB^k(\P_n,r)) = \sigma^2(\beta^{(k)}).$$
Moreover, if $\Var\Vol_{d} ( \cCB^k(\P_n,r)) = \Omega(n^{\nu} )$ for some $\nu \in (0, \infty)$,  then as $n \to \infty$
\begin{equation}
\label{maincltgerm-grain}
\frac{\Vol_{d} ( \cCB^k(\P_n,r)) - \E\Vol_{d} ( \cCB^k(\P_n,r)) } { \sqrt{ \Var\Vol_{d} ( \cCB^k(\P_n,r)) } } \tod N.
\end{equation}
\end{theorem}
%
In the case of Poisson input and $k = 1$, \cite{Hall88} establishes a central limit theorem for $\cCB^1(\P_n,r)$.
For general $k$,
the  central limit theorem for Poisson input can be deduced from the general results in \cite{PY1,Baryshnikov05} with  presumably optimal bounds following from \cite[Proposition 1.4]{LPS}.
	
	\vskip.5cm	
	
\noindent{\bf Intrinsic volumes of the germ-grain model.}
Let $\mathcal{K}$ denote the set of {\em convex bodies}, i.e., compact, convex subsets of  $\mR^d$. The {\em intrinsic volumes} $V_0,\ldots,V_d$ are non-negative functionals on $\cK$ which satisfy {\em Steiner's formula}
$$ V_d(K \oplus B_r(\0)) = \sum_{j=0}^dr^{d-j}\theta_{d-j}V_j(K), \ K \in \mathcal{K}, $$
where $r > 0$, $\oplus$ denotes  Minkowski addition, $V_j$ denotes the $j$-dimensional Lebesgue measure, and $\theta_j := \pi^{j/2}/\Gamma(j/2 + 1)$ is the volume of the unit ball in $\mR^j$. Intrinsic volumes satisfy translation invariance and additivity, i.e., for $K_1,\ldots,K_m \in \cK$,
$$
V_j(\cup_{i=1}^m K_i)  =  \sum_{k=1}^m (-1)^{k+1}\sum_{1 \leq i_1 < \ldots < i_k \leq m}V_j(\cap_{l=1}^kK_{i_l}), \, \, \, \, j \in \{0,\ldots,d\}.
$$
This identity allows an extension of intrinsic volumes to real-valued functionals on 
the family of finite unions of convex bodies \cite[Ch. 14]{Schneider08}. 
Intrinsic volumes coincide  with {\em quer-mass integrals} or {\em Minkowski functional} up to a normalization.
The $V_j$'s define certain $j$-dimensional volumes of $K$, independently of the ambient space. $V_d$ is the $d$-dimensional volume, $2V_{d-1}$ is the surface measure, and $V_0$ is the {\em Euler-Poincar\'{e} characteristic} which may also be expressed as an alternating sum of simplex counts, Morse critical points, or Betti numbers \cite[Sections IV.2, VI.2]{EH}. Save for $V_d$ and $V_{d-1}$, the remaining $V_j$'s may assume negative values on unions of convex bodies.
	
For finite $\X$ and $r > 0$, we express $V_j(\cCB(\X,r)), j \in \{0,\ldots,d\}$, as a sum of
bounded stabilizing scores, which goes as follows. 
For $x_1 \in \X$, define the score
$$ \xi_j(x_1,\X) :=  \sum_{k=1}^{\infty}(-1)^{k+1}\sum_{(x_2,\ldots,x_k) \in (\X \cap B_{2r}(x_1))^{(k)}}\frac{V_j(B_r(x_1) \cap \ldots \cap B_r(x_k))}{k!} .$$
The score $\xi_j$ is translation invariant with radius of stabilization $R^{\xi_j}(x_1, \P) \leq \lceil 3r \rceil$. By additivity, we have  $V_j(\cCB(\P_n,r)) = H^{\xi_j}_n(\P)$ for $j \in \{0,\ldots,d\}$. The homogeneity relation $V_j(B_r(\0)) = r^jV_j(B_1(\0)) = r^j{d \choose j}\frac{\theta_d}{\theta_{d-j}}$ and the monotonicity of $V_j$ on $\cK$ yield
$$ |\xi_j(x_1,\X)|\1[\X(B_{2r}(x_1)) = l] \leq r^j {d \choose j}\frac{\theta_d}{\theta_{d-j}}\sum_{k=1}^l{l \choose k} \leq 2^l r^j{d \choose j}\frac{\theta_d}{\theta_{d-j}}.$$
In other words, $\xi_j, 0 \leq j \leq d,$ satisfy the power growth condition \eqref{eqn:xibd} and are exponentially stabilizing. Theorems \ref{prop:variance_upper_bound} and \ref{thm:main1} yield the following limit theorems, where we note that $(\xi_j, \P)$ is of class (A2) whenever $\P$ is of class (A2).
	
\begin{theorem}
\label{thm:intvolume} Fix $r \in (0, \infty)$ as above. For all  $j \in \{0,\ldots,d\}$ and any point process $\P$ of class $(A2)$ with the pair $(\xi_j,\P)$ satisfying the $p$-moment condition \eqref{eqn:mom} for all $p \in (1,\infty)$, we have
$$ | n^{-1} \E V_j(\cCB(\P_n,r))  - \E_{\0}\xi_j(\0, \P)\rho^{(1)}(\0) | = O(n^{-1/d}), \, \mbox{and} \,
\lim_{n \to \infty} n^{-1} \Var V_j(\cCB(\P_n,r)) = \sigma^2(\xi_j).
$$
Moreover, if $\Var V_j(\cCB(\P_n,r)) = \Omega(n^{\nu} )$ for some $\nu \in (0, \infty)$  then as $n \to \infty$
%
$$
\frac{ V_j(\cCB(\P_n,r)) - \E V_j(\cCB(\P_n,r)) } { \sqrt{ \Var V_j(\cCB(\P_n,r)) } } \tod N.
$$
\end{theorem}

\noindent{\bf Remarks.}
\begin{enumerate}[itemindent=0.8cm,labelwidth=*,label=(\roman*),leftmargin=0cm,topsep=-2ex]
\item Theorem \ref{thm:intvolume} extends the analogous central limit theorems of \cite{Hug15}, which are confined to Poisson input, to any point process of class (A2).
		
\item We may likewise prove central limit theorems for other functionals of germ-grain models, including mixed volumes, integrals of surface
area measures \cite[Chapters 4 and 5]{Schneider93}, and total measures of translative integral geometry \cite[Section 6.4]{Schneider08}.  These functionals, like intrinsic volumes, are expressed as sums of bounded stabilizing scores and thus, under suitable assumptions, the limit theory for these functionals follows from Theorems \ref{prop:variance_upper_bound} and \ref{thm:main1}.
		
\end{enumerate}		

\subsubsection{Edge-lengths in $k$-nearest neighbors graph}
\label{sec:knnd}
	
We now use the full force of Theorems \ref{prop:variance_upper_bound} and \ref{thm:main1}, applying them to sums of score functions whose radius of stabilization has an exponentially decaying tail.
	
Statistics of the Voronoi tessellation as well as of graphs in computational geometry such as the $k$-nearest neighbors graph and sphere of influence graph  may be expressed as sums of exponentially stabilizing score functionals \cite{PY1} and hence via Theorems \ref{prop:variance_upper_bound} and \ref{thm:main1}, we may  deduce the limit theory for these statistics. To illustrate, we establish a weak law of large numbers, variance asymptotics, and a central limit theorem for the total edge-length of the $k$-nearest neighbors graph on a $\alpha$-determinantal point process $\P_{\alpha}$ with $-1/\alpha \in \N$ and a fast-decreasing kernel as in \eqref{fastdk}. As noted in Proposition \ref{prop:clustering_alpha_det} , such an $\alpha$-determinantal point process is of class (A2) as in Definition \ref{def.A2}.
	
As shown in Corollary \ref{cor:palm-void-aDPP}, we may explicitly upper bound void probabilities for $\P$, allowing us to deduce exponential stabilization for score functions on $\P$. This is a recurring phenomena, and it is often the case that to show exponential stabilization of statistics, it suffices to control the Palm probability content of large Euclidean balls. This opens the way towards showing that other relevant statistics of random graphs exhibit exponential stabilization on $\P$.  This includes intrinsic volumes of faces of Voronoi tessellations \cite[Section 10.2]{Schneider08}, edge-lengths in a radial spanning tree  \cite[Lemma 3.2]{Schulte16}, proximity graphs including the Gabriel graph, and global Morse critical points i.e., critical points as defined in Section \ref{sec:isol_u_stat} but without the restriction ${\bf 1}[R(x,\x) \leq r]$.
	
	Given locally finite $\X \subset \R^d$ and $k \in \N$, the (undirected) $k$-nearest neighbors graph $NG(\X)$ is the graph with vertex set $\X$ obtained by including  an edge $\{x,y\}$  if $y$ is one of the $k$ nearest neighbors of $x$ and/or $x$ is one of the $k$ nearest neighbors of $y$. In the case of a tie
	we may break the tie via some pre-defined total order (say lexicographic order) on $\mR^d$. For any finite $\X \subset \R^d$ and $x \in \X$, we let
	${\cal E}(x)$ be the edges $e$ in  $NG(\X)$ which are incident to $x$.  Defining
$$ \xi_L(x, \X) :=  {1 \over 2}  \sum_{e \in {\cal E}(x)} |e|, $$
we write the total edge length of $NG(\X)$ as  $L(NG(\X)) = \sum_{x \in \X} \xi_L(x, \X).$  Let  $\sigma^2(\xi_L)$ be as at \eqref{eqn:sigdef}, with $\xi$ put to be $\xi_L$.
\begin{theorem}
\label{thm:knng}
Let $\P := \P_{\alpha}$ be a stationary $\alpha$-determinantal point process on $\mR^d$ with $-1/\alpha \in \N$ and a fast-decreasing kernel $K$ as at \eqref{fastdk}.
We have
$$
| n^{-1} \E L(NG(\P_n))  - \E_{\0}\xi_L(\0, \P) K(\0, \0)| = O(n^{-1/d}), \, \mbox{and} \,
\lim_{n \to \infty} n^{-1} \Var L(NG(\P_n)) = \sigma^2(\xi_L).
$$
If $\Var L(NG(\P_n)) = \Omega(n^{\nu} )$ for some $\nu \in (0, \infty)$  then as $n \to \infty$
$$
\frac{ L(NG(\P_n)) - \E L(NG(\P_n)) } { \sqrt{ \Var L(NG(\P_n)) } } \tod N.
$$
\end{theorem}  		
	\noindent{\bf Remark}. Theorem \ref{thm:knng} extends Theorem 6.4 of \cite{PeEJP} which is confined to Poisson input.  In this context, the work  \cite{LPS} provides  a rate of
	normal approximation.
	\vskip.3cm
	\noindent{\em  Proof.}  We want to show that $(\xi_L, \P)$ is an admissible score and input pair of type (A2) and then apply
	Theorem \ref{thm:main1}. Note that $\P$ is an admissible point process which has fast decay of correlations satisfying  \eqref{eqn:sum}
	and \eqref{phibd}.  Thus we only need to show that $\xi_L$ is exponentially stabilizing, that $\xi_L$ satisfies
	the power growth condition \eqref{eqn:xibd}, and the $p$-moment condition \eqref{eqn:mom}.
	When $d = 2$, we show exponential stabilization of $\xi_L$ by closely following the proof of Lemma 6.1  of \cite{PY4}.  This goes as follows.
	For each $t > 0$, construct six disjoint equal triangles $T_j(t), 1 \leq j \leq 6$, such that $x$ is a vertex of each triangle and each edge
	has length $t$.  Let the random variable $R$ be the minimum $t$ such that $\P_n(T_j) \geq k + 1$ for all $1 \leq j \leq 6$.  Notice that
	$R \in [r, \infty)$ implies that there is a ball inscribed in some $T_j(t)$ with center $c_j$ of radius $\gamma r$ which does not contain
	$k + 1$ points. Combining Corollary \ref{cor:palm-void-aDPP} in the Appendix and the fact that $\P$ has kernel $K$, the probability of this event satisfies
	$$\mP_{x_1,\ldots,x_p}[R > r] \leq 6 \mP_{x_1,\ldots,x_p}[\P(B_{\gamma r}(c_1)) \leq k-1] \leq 6 \mP^!_{x_1,\ldots,x_p}[\P(B_{\gamma r}(c_1)) \leq k-1]  $$ $$
	\leq 6e^{m(2k+p-2)/8}e^{- K(\0,\0)  \pi \gamma^2 r^2/8},$$
	that is to say that $R$ has exponentially decaying tails.  As in Lemma 6.1 of \cite{PY4}, we find that $R^\xi(x, \P_n):=4 R$ is a radius of
	stabilization for $\xi_L$, showing that \eqref{varphibd}  holds with $c = 2$.
	For $d > 2$, we may extend these geometric arguments (cf. the proof of Theorem 6.4 of \cite{PeEJP})  to define a random variable $R$ serving as a radius of
	of stabilization.  Mimicking the above arguments we may likewise show that $R$ has exponentially decaying tails.

	For all $r \in (0,  \infty)$ and $l \in \N$  we notice that \eqref{eqn:xibd} holds because
		$$ |\xi_L(x,\X \cap B_r(x) )|\1[\X(B_r(x)) = l] \leq r \cdot \min(l, 6) \leq (cr)^l.$$
		  Since vertices in the $k$-nearest neighbors graph have degree bounded by
		$kC(d)$ as in Lemma 8.4 of \cite{Yubook}, and since each edge incident to $x$ has length at most $4R$, it follows
		that $|\xi_L(x, \P_n)| \leq k \cdot C(d) \cdot 4R.$ Since $R$ has moments of all orders,
		$(\xi_L, \P)$ satisfies the $p$-moment condition \eqref{eqn:mom} for all $p \geq 1.$
		Thus $\xi_L$ satisfies all conditions of Theorem \ref{thm:main1} and
		we deduce Theorem \ref{thm:knng} as desired.  \qed

  \section{\bf Proof of the fast decay  ~\eqref{e.clustering-generalized-mixed-monents} for correlations of the $\xi$-weighted measures}
	\label{sec:proofs_clustering}
	
We show the  decay bound ~\eqref{e.clustering-generalized-mixed-monents} via a factorial moment expansion for the expectation of functionals of
point processes. 	Notice	that \eqref{e.clustering-generalized-mixed-monents}	holds for any exponentially stabilizing score function  $\xi$ satisfying the $p$-moment condition \eqref{eqn:mom}	for all $p \in [1, \infty)$ on a Poisson point process~$\P$.  Indeed if  $x, y \in \R^d$ and $r_1, r_2>0$ satisfy $r_1+ r_2<|x-y|$ then 	$\xi(x, \P)\1[R^\xi(x, \P)\le r_1]$ and $\xi(y, \P)\1[R^\xi(y, \P)\le r_2]$
are independent random variables.  This yields the fast decay \eqref{e.clustering-generalized-mixed-monents}	with $k_1=\ldots = k_{p+q}=1$ and	$\tilde{C}_{n} \leq c_1^{n}$  with  $c_1$ a constant,  as in \cite[Lemma 5.2]{Baryshnikov05}.	On the other hand, if $\P$ is rarified Gibbsian input and $\xi$ is exponentially stabilizing, then  \cite[Lemma 3.4]{SY} shows the fast decay bound \eqref{e.clustering-generalized-mixed-monents}
with $k_1=\ldots = k_{p+q}=1$. These methods depend on quantifying the region of spatial dependencies of Gibbsian points via exponentially decaying diameters of their ancestor clans.  Such methods apparently neither extend to determinantal input nor to the zero set $\P_{GEF}$ of a Gaussian entire function.  On the other hand, for $\P_{GEF}$  and for $\xi \equiv 1$,
the paper \cite{Nazarov12} uses the Kac-Rice-Hammersley formula and complex analysis tools to show \eqref{e.clustering-generalized-mixed-monents} with $k_1=\ldots = k_{p+q}=1$.
 All three proofs are specific to either the underlying point process or to the score function $\xi$. The following more general and considerably different approach includes these results as special cases.
	
\subsection{\bf Difference operators and factorial moment expansions}
\label{ss.FME}	
	We  introduce some notation and collect auxiliary results required for an application of the much-needed factorial moment expansions \cite{Blaszczyszyn95,Bartek97} for general point processes. Equip $\mR^d$ with a total order $\prec$ defined using the lexicographical ordering of the polar	co-ordinates. For $\mu \in \cN$ and $x \in \mR^d$, define the measure
	$\mu_{|x}(.) := \mu(. \cap \{y : y \prec x\})$. Note that since $\mu$
	is a locally finite measure and the ordering is defined via polar
	co-ordinates $\mu_{|x}$ is a finite measure for all $x \in
	\mR^d$. Let $o$ denote the null-measure i.e., $o(B) = 0$ for all
	Borel subsets $B$ of $\mR^d$.
	For a measurable function $\psi : \cN \to
	\mR$, $l \in \N \cup \{0\}$,
	and $x_1,...,x_l \in \R^d$, we define the factorial moment expansion
	(FME) kernels \cite{Blaszczyszyn95,Bartek97} as follows.
For $l \geq 1$,
	\begin{equation}
	\label{eqn:FME_Kernels1}
	D^{l}_{x_1,\ldots,x_l}\psi(\mu) = \sum_{i=0}^l(-1)^{l-i}\sum_{J \subset \binom{[l]}{i}}\psi(\mu_{|x_*}
	+ \sum_{j \in J}\delta_{x_j}) = \sum_{J \subset [l]}(-1)^{l-|J|}\psi(\mu_{|x_*} + \sum_{j \in J}\delta_{x_j}),
	\end{equation}
	where $\binom{[l]}{j}$ denotes the collection of all subsets of $[l]: =
	\{1,\ldots,l\}$ with cardinality $j$ and  $x_* := \min
	\{x_1,\ldots,x_l\}$, with  the minimum  taken with respect to
	the order $\prec$. 	For $l = 0$, put $D^{0}\psi(\mu) := \psi(o)$. Note that
	$D^{(l)}_{x_1,\ldots,x_l}\psi(\mu)$ is a symmetric function of
	$x_1,\ldots,x_l$.\footnote{For $x_l \prec x_{l-1} \prec \ldots \prec x_1$ the functional $D^{l}_{x_1,\ldots,x_l}\psi(\mu)$
		is equal to the iterated difference operator:
		$D^1_{x_1} \psi(\mu)= \psi(\mu_{|x_1}+ \delta_{x_1})- \psi(\mu_{|x_1})$,  $D^{l}_{x_1,\ldots,x_l}\psi(\mu)=
		D^1_{x_l}(D^{l-1}_{x_1,\ldots,x_{l-1}}  \psi(\mu) )$.}
	
We say that $\psi$ is $\prec$-continuous at $\infty$ if for all $\mu \in \cN$ we have
	$$
	\lim_{x \uparrow \infty }  \psi(\mu_{|x} ) = \psi(\mu).
	$$
	We first recall the FME expansion proved
	in~\cite[cf. Theorem 3.2]{Blaszczyszyn95} for dimension one and then extended to
	higher-dimensions  ~\cite[cf. Theorem 3.1]{Bartek97}. Recall that
	$\sE^{!}_{y_1,\ldots,y_l}$ denote expectations with respect to reduced Palm probabilities.
\begin{theorem}
\label{FMEthm}
Let $\P$ be a simple point process and let $\psi : \cN \to \mR$ be $\prec$-continuous at $\infty$. Assume that for all $l \geq 1$
\be \label{hyp1}
		\int_{\mR^{dl}}\E^!_{y_1,\ldots,y_l}[| D^{l}_{y_1,\ldots,y_l}\psi(\P)|]\rho^{(l)}(y_1,\ldots,y_l)\,\md y_1\ldots \md y_l < \infty
\ee
and
\be \label{hyp2}
		\frac{1}{l!}\int_{\mR^{dl}}\E^!_{y_1,\ldots,y_l}[D^{l}_{y_1,\ldots,y_l}\psi(\P)]\rho^{(l)}(y_1,\ldots,y_l)\,\md y_1\ldots \md y_l \to 0 \ {\rm{as}} \ l \to \infty.
\ee
Then $\E[\psi(\P)]$ has the following {\em factorial moment expansion }
\be
\label{FME}
		\E[\psi(\P)] = \psi(o) +  \sum_{l=1}^{\infty}\frac{1}{l!}\int_{ \mR^{dl} }D^{l}_{y_1,\ldots,y_l}\psi(o)\rho^{(l)}(y_1,\ldots,y_l)\,\md y_1\ldots \md y_l.
\ee
\end{theorem}
Consider now admissible pairs $(\xi, \P)$ of type (A1) or (A2) and $x_1,\ldots,x_p\in\R^d$.	The proof of~\eqref{e.clustering-generalized-mixed-monents} given in the next sub-section is based on the FME expansion for $\mE_{x_1,\ldots,x_{p}}[\psi(\P_n)]$, where $\psi(\mu)$ is  the following product of the score functions
	\begin{equation}\label{JYpsi}
	\psi(\mu):= \psi_{k_1,\ldots,k_p}(x_1,\ldots,x_p;\mu):=\prod_{i=1}^p\xi(x_i,\mu)^{k_i}
	\end{equation}
with $k_1,\ldots,k_p\ge 1$. However, under $\mP_{x_1,\ldots,x_{p}}$ the point process $\P_n$ has fixed atoms at $x_1,\ldots,x_{p}$, which complicates the form of its
	factorial moment measures. It is more handy to consider these points as parameters of the following  modified functional
	\begin{equation}\label{e.psip}
	\psip(\mu):=\psip_{k_1,\ldots,k_p}(x_1,\ldots,x_p;\mu):=
	\prod_{i=1}^p\xi\Bigl(x_i,\mu+\textstyle{\sum_{j=1}^p}\delta_{x_j}\Bigr)^{k_i}\,
	\end{equation}
	and to not count points $x_1,\ldots,x_{p}$ in $\P$, i.e., consider $\P$ under the reduced Palm probabilities $\mP^!_{x_1,\ldots,x_{p}}$.
	Obviously $\E_{x_1,\ldots,x_p}[\psi(\P_n)]=\E^!_{x_1,\ldots,x_{p}}[\psip(\P_n)]$
	and the latter expectation is more suitable for FME expansion with respect to the correlation functions
	$\rho^{(l)}_{x_1,\ldots,x_p}(y_1,\ldots,y_l)$ 
	of $\P$ with respect to the	Palm probabilities $\mP^!_{x_1,\ldots,x_p}$. The following consequence of Theorem \ref{FMEthm} allows us to use  FME expansions to prove \eqref{e.clustering-generalized-mixed-monents}.
	
\begin{lemma}\label{l.FMEp}  Assume that either (i) $(\xi, \P)$ is  an admissible score and input pair  of type  (A1)
		or (ii) $(\xi, \P)$ satisfies the power growth
		condition~\eqref{eqn:xibd}, with $\xi$ having a radius of
		stabilization satisfying
		$ \sup_{x \in \P} R^\xi(x,\P) \leq r$ a.s. for some $r \in (1, \infty)$ and with $\P$ having exponential moments. Then for distinct $x_1,\ldots,x_p\in \R^d$, non-negative integers $k_1,\ldots,k_p$ and $n\le\infty$
		the functional $\psip$ at \eqref{e.psip} admits the FME
		\begin{align}\nonumber
		&\mE_{x_1,\ldots,x_p}[\psi_{k_1,\ldots,k_p}(x_1,\ldots,x_{p};\P_n)] = \mE^!_{x_1,\ldots,x_p}[\psip_{k_1,\ldots,k_p}(x_1,\ldots,x_p;\P_n)] \no \\
		&= \psip_{k_1,\ldots,k_p}(x_1,\ldots,x_p;o)\nonumber \\
		&\;+\sum_{l=1}^{\infty}\frac{1}{l!}\int_{\mR^{dl}} D^{l}_{y_1,\ldots,y_l}\psip_{k_1,\ldots,k_p}(x_1,\ldots,x_p;o)
		\rho^{(l)}_{x_1,\ldots,x_p}(y_1,\ldots,y_l)\,\md y_1\ldots \md y_l.
		\label{FMEpsip}
		\end{align}
		When $(\xi, \P)$ is of type (A1), the series \eqref{FMEpsip} has at
		most $(k-1)\sum_{i=1}^pk_i$ non-zero terms, where $k$ is as in \eqref{eqn:score_U_statistic}.
	\end{lemma}
	
	\begin{proof}
Throughout we fix non-negative integers $k_1,\ldots,k_p$	and suppress them when writing  $\psip$; i.e.,		  $\psip(x_1,\ldots,x_p;\P_n):=\psip_{k_1,\ldots,k_p}(x_1,\ldots,x_p;\P_n)$. 	
The bounded radius of stabilization for $\xi$ implies $\psip$ is $\prec$-continuous at~$\infty$.
		
Consider first $\psip$ at~\eqref{e.psip} with $\xi$ as in case (ii); later we  consider the simpler case (i). We  show  the validity
of the expansion~\eqref{FMEpsip} as follows. Let $y_1,\ldots,y_l \in \mR^d$.  The difference operator $D^{l}_{y_1,\ldots,y_l}$ vanishes as soon as $y_k \notin \cup_{i=1}^pB_r(x_i)$ for some $k \in \{1,\ldots,l\}$, that is to say
		\be\label{eqn:FME_local}
		D^{l}_{y_1,\ldots,y_l}\psip(x_1,\ldots,x_p;\mu) = 0.
		\ee
		To prove this, set $\mu_J := \mu|_{y_*}+\sum_{j \in J}\delta_{y_j}$ for
		$J \subset [l]$ and $y_* := \min\{y_1,\ldots,y_l\}$, with  the minimum  taken with respect to $\prec$
		order.
		From \eqref{eqn:FME_Kernels1} we obtain
		\begin{align*}
		&D^{l}_{y_1,\ldots,y_l}\psip(x_1,\ldots,x_p;\mu)\\
		& = \sum_{J \subset [l], k \notin J} (-1)^{l-|J|}\psip(x_1,\ldots,x_p;\mu_J) + \sum_{J \subset [l], k \notin J}(-1)^{l-|J|-1}\psip(x_1,\ldots,x_p;\mu_{J \cup\{k\}}) = 0,
		\end{align*}
where the last equality follows by noting that for $J \subset [l]$
		with $k \notin J$, $\psip(x_1,\ldots,x_p;\mu_J) =
		\psip(x_1,\ldots,x_p;\mu_{J \cup\{k\}})$ because $R^\xi(x,\P) \in [1, r]$ by assumption.

Henceforth we put
\be \label{defKp}
K_p:=\sum_{i=1}^pk_i, \ K_q:=\sum_{i=1}^q k_{p + i},  \ K:= \sum_{i=1}^{p + q} k_i.
\ee
		
Consider now $y_1,\ldots,y_l \in \cup_{i=1}^pB_r(x_{i})$. 
		For $J \subset [l]$, from $1 \leq R^\xi(x,\P)\le r$ and \eqref{eqn:xibd} we have

		\be \label{BD}
		\psip(x_1,\ldots,x_p;\mu_J)  \leq
		(\hat{c}r)^{K_p|J|+pK_p+\sum_{i=1}^pk_i\mu(B_r(x_i))}.\ee
		The term $pK_p$ in the exponent of \eqref{BD}  is due to
		$\sum_{j=1}^p \delta_{x_j}$ in the argument of $\xi$
		in~\eqref{e.psip}.
		Substituting this  bound in \eqref{eqn:FME_Kernels1} yields
		\begin{align}
		|D^{l}_{y_1,\ldots,y_l}\psip(x_1,\ldots,x_p;\mu)|& \leq
		(\hat{c}r)^{pK_p+\sum_{i=1}^pk_i\mu(B_r(x_i))}\sum_{J \subset
			[l]}(\hat{c}r)^{K_p|J|} \no  \\
		\label{e.DLbound} & = (\hat{c}r)^{pK_p+\sum_{i=1}^pk_i\mu(B_r(x_i))}(1+(\hat{c}r)^{K_p})^l.
		\end{align}
Consider $\psip(x_1,\ldots,x_p;\P_n)$, with $\P_n:= \P \cap W_n$ and $\psip$
defined as above. The bound~\eqref{e.DLbound} yields
\begin{align}
		& \frac{1}{l!}\int_{\mR^{dl}} (\sE^!_{x_1,\ldots,x_p})^!_{y_1,\ldots,y_l}[ |D^{l}_{y_1,\ldots,y_l}\psip(x_1,\ldots,x_p;\P_n)|] \rho^{(l)}_{x_1,\ldots,x_p}(y_1,\ldots,y_l)\,\md y_1 \ldots \md y_l  \no\\
		=& \frac{1}{l!}\int_{\mR^{dl}} \sE^!_{x_1,\ldots,x_p,y_1,\ldots,y_l}[ |D^{l}_{y_1,\ldots,y_l}\psip(x_1,\ldots,x_p;\P_n)|]
		\rho^{(l)}_{x_1,\ldots,x_p}(y_1,\ldots,y_l)\,\md y_1 \ldots \md y_l  \no \\		
\le & \frac{(1 + (\hat{c}r)^{K_p})^l(\hat{c}r)^{pK_p}}{l!} \sE^!_{x_1,\ldots,x_p}		\left[(\P_n(\cup_{i=1}^{p} B_r(x_i))^l(\hat{c}r)^{\sum_{i=1}^pk_i\P_n(B_r(x_i))}\right] \no \\
\le & \frac{(1 + (\hat{c}r)^{K_p})^l(\hat{c}r)^{pK_p}}{l!} \sE^!_{x_1,\ldots,x_p} \left[(\P_n(\cup_{i=1}^pB_r(x_i))^l(\hat{c}r)^{K_p\P_n(\cup_{i=1}^pB_r(x_i))}\right] \nonumber \\
\le & \frac{(1 + (\hat{c}r)^{K_p})^l}{l!} \sE_{x_1,\ldots,x_p}
\left[(\P_n(\cup_{i=1}^pB_r(x_i))^l(\hat{c}r)^{K_p\P_n(\cup_{i=1}^pB_r(x_i))}\right]\,,\label{e.l-bound}
\end{align}
where the last inequality follows since the distribution of $\P$ under $\mP_{x_1,\ldots,x_p}$ is equal to
		that of $\P+\sum_{i=1}^p\delta_{x_i}$ under $\mP^!_{x_1,\ldots,x_p}$.  Defining  $N:=\P_n(\cup_{i=1}^pB_r(x_i))$,  we bound \eqref{e.l-bound}  by
\begin{align*}
\sE_{x_1,\ldots,x_p}\left[(\hat{c}r)^{K_pN}
\sum_{m=l}^\infty \frac{(1 + (\hat{c}r)^{K_p})^l}{l!}N^l\right]\le
\sE_{x_1,\ldots,x_p}\left[ (\hat{c}r)^{(1 + (\hat{c}r)^{K_p}+K_p)N}\right]<\infty\,,
\end{align*}
where the last inequality follows since $\P$ has exponential moments under the Palm measure as well (see Remark~\ref{rem:exp_mom-i} at the beginning of Section~\ref{momsection}).  Consequently, by the Lebesgue dominated convergence theorem, the expression~\eqref{e.l-bound} converges to $0$ as	$l\to\infty$. Thus conditions \eqref{hyp1} and  \eqref{hyp2} hold and \eqref{FMEpsip} follows by Theorem  \ref{FMEthm}.

Now we consider case (i), that is to say $\psip$ is as at~\eqref{e.psip} with $\xi$ a U-statistic of type (A1). By Lemma~\ref{l.U-stats}, with $k$ as in \eqref{eqn:score_U_statistic}, $\psip$ is a sum of $U$-statistics of orders not larger than~$K_p(k-1)$. Consequently, for $l \in (K_p(k-1), \infty)$ we have
		\be
		\label{eqn:FME_U}
		D^{l}_{y_1,\ldots,y_l}\psip(x_1,\ldots,x_p;\mu) = 0  \, \,   \, \forall y_1,\ldots,y_l \in \mR^d,
		\ee
as shown in \cite[Lemma 3.3]{Reitzner13} for Poisson point processes (the proof for general simple counting measures $\mu$ is identical).
This implies that conditions ~\eqref{hyp1} for  $l \in (K_p(k-1), \infty)$ and  \eqref{hyp2} are trivially satisfied for $\psip$  as at \eqref{e.psip}. Now, we need to verify the condition~\eqref{hyp1} for $l \in [1,  K_p(k-1)]$. For $y_1,\ldots,y_l \in \mR^d$, set as before $\mu_J = \mu|_{y_*}+\sum_{j \in J}\delta_{y_j}$ for
$J \subset [l]$ and $y_* := \min\{y_1,\ldots,y_l\}$, with  the minimum  taken with respect to the order $\prec$.  Since $\xi$ has a bounded stabilization radius, by \eqref{eqn:FME_local} and \eqref{eqn:mom-exp-U}, we have

\begin{eqnarray}
\psip(x_1,\ldots,x_p;\mu_J) & \leq & \prod_{i=1}^p \|h\|_{\infty}^{k_i}(\mu(\cup_{i=1}^pB_r(x_i))+|J|+p)^{k_i(k-1)} \no \\
& \leq & \|h\|_{\infty}^{K_p}(\mu(\cup_{i=1}^pB_r(x_i))+|J|+p)^{K_p(k-1)}. \label{BD_U}
\end{eqnarray}
The number of subsets of $[l]$ is $2^l$ and so by \eqref{eqn:FME_Kernels1}, we obtain
\begin{align} |D^{l}_{y_1,\ldots,y_l}\psip(x_1,\ldots,x_p;\mu)|  & \leq  \|h\|_{\infty}^{K_p} \sum_{J \subset [l]} (\mu(\cup_{i=1}^pB_r(x_i))+|J|+p)^{K_p(k-1)}	\no \\
& \leq  \|h\|_{\infty}^{K_p} 2^l (\mu(\cup_{i=1}^pB_r(x_i))+l+p)^{K_p(k-1)}.  \label{e.DLbound_U}
\end{align}
Consider $\psip(x_1,\ldots,x_p;\P_n)$ with $\psip$ defined as above. Using the refined Campbell theorem \eqref{disint}, the bound~\eqref{e.DLbound_U}, and following the calculations as in \eqref{e.l-bound}, we obtain
\begin{eqnarray*}
				& & \frac{1}{l!}\int_{\mR^{dl}} (\sE^!_{x_1,\ldots,x_p})^!_{y_1,\ldots,y_l}[ |D^{l}_{y_1,\ldots,y_l}\psip(x_1,\ldots,x_p;\P_n)|] \rho^{(l)}_{x_1,\ldots,x_p}(y_1,\ldots,y_l)\,\md y_1 \ldots \md y_l
							\end{eqnarray*}
										\begin{eqnarray*}
				& \leq &  \|h\|_{\infty}^{K_p} 2^l \E_{x_1,\ldots,x_p}[\P(\cup_{i=1}^pB_r(x_i))^l(\P(\cup_{i=1}^pB_r(x_i))+l+p)^{K_p(k-1)}].
			\end{eqnarray*}
Since $\P$ has all moments under the Palm measure (see Remark~\ref{rem:exp_mom-ii} at the beginning of Section~\ref{momsection}), the finiteness of the last term and hence the validity of the condition~\eqref{hyp1} for  $l \in [1,  K_p(k-1)]$ follows. This justifies the FME expansion~\eqref{FMEpsip}, with finitely many non-zero terms, when $\psip$ is the product of score functions of class (A1).
\end{proof}		
\subsection{\bf Proof of  Theorem~\ref{prop:clustering_gen}} \label{ss.clustering-proof}
First assume that $(\xi, \P)$ is of class (A2). Later we consider the simpler case that $(\xi, \P)$ is of class (A1).
	For fixed $p,q,k_1,\ldots,k_{p+q} \in \mathbb{N}$,
	consider correlation functions $m^{(k_1,\ldots,k_{p+q})}(x_1,\ldots,x_{p+q};n)$, $m^{(k_1,\ldots,k_{p})}(x_1,\allowbreak\ldots,x_{p};n)$, $m^{(k_{p+1},\ldots,k_{p+q})}\allowbreak(x_{p+1},\ldots,x_{p+q};n)$
of the $\xi$-weighted measures at \eqref{eqn:mixedmomentn}.
		We abbreviate $\psi_{k_1,\ldots,k_p}(x_1,\allowbreak \ldots,x_p;\mu)$ by
		$\psi(x_1,\allowbreak \ldots,x_p;\mu)$
		as at  \eqref{JYpsi}, and similarly for $\psi(x_{p+1},\allowbreak \ldots,x_{p+q};\mu)$
		and $\psi(x_{1},\ldots,x_{p+q};\mu)$.
	
Given  $x_1,\ldots,x_{p+q} \in W_n$ we recall  $s:= d(\{x_1,\ldots,x_p\},\{x_{p+1},\ldots,x_{p+q}\})$.  Without loss of generality we assume $s \in (4, \infty)$.
	Recalling the definition of $b$ at  \eqref{phibd} and that of $K$ at \eqref{defKp}, we may assume without loss of generality that $b \in (0, d)$.
	Put
	\be \label{defsprime}
	t:=t(s):= (\frac{ s} {4})^{b(1-a)/(2(K + d))},
	\ee
	where  $a \in [0,1)$ is at \eqref{eqn:sum}. Since $s \in (4, \infty)$ and $K \geq 2$,  we easily have $t \in (1, s/4)$.	
	Given stabilization radii $R^\xi(x_i,\P_n), 1 \leq i \leq p + q$, we put  $$\tx(x_i,\P_n)
	:= \xi(x_i, \P_n \cap B_{R^\xi(x_i,\P_n)}(x))\1[R^\xi(x_i,\P_n) \leq
	t]$$
	considered under $\mE_{x_1,\ldots,x_p}$. We denote by $\tm^{(k_1,\ldots,k_p)}$ the correlation functions of the $\tx$-weighted atomic measure, that is
	$$
	\tm^{(k_1,\ldots,k_p)}(x_1,\ldots,x_p ; n) := \sE_{x_1,\ldots,x_p} [\tx(x_1,\P_n)^{k_1}\ldots\tx(x_p,\P_n)^{k_p}] \rho^{(p)}(x_1,\ldots,x_p).
	$$
Write
\be \label{psitilde}
\tps(x_1,\ldots,x_{p};\P_n) = \psi(x_1,\ldots,x_p;\P_n) \1[ \max_{ i \leq p} R^\xi(x_i, \P_n) \leq t] = \prod_{i=1}^p \tx(x_i,\P_n)^{k_i}.\ee
Next, write $\E_{x_1,\ldots,x_{p}}\psi(x_1,\ldots,x_{p};\P_n)$ as a sum of
	$$
	\E_{x_1,\ldots,x_p}[\psi(x_1,\ldots,x_{p};\P_n) \1[ \max_{ i \leq p} R^\xi(x_i, \P_n) \leq t ] ]
	$$
	and
	$$
	\E_{x_1,\ldots,x_{p}} [\psi(x_1,\ldots,x_{p};\P_n)\1 [ \max_{ i \leq p} R^\xi(x_i, \P_n) > t] ].$$
	 The bounds  \eqref{eqn:corr_bounded}, \eqref{stab}, the moment
		condition \eqref{eqn:mom}, H\"older's inequality, and $p\le \sum_{i=1}^p k_i = K_p$ give for Lebesgue almost all  $x_1,\ldots,x_{p}$
	\begin{align}
	&  \Bigl| \E_{x_1,\ldots,x_{p}}\psi(x_1,\ldots,x_{p};\P_n) - \E_{x_1,\ldots,x_{p}}
	\tps(x_1,\ldots,x_{p};\P_n) \Bigr|\rho^{(p)}(x_1,\ldots,x_{p}) \no\\
	& \le p \kappa_{p}(\tM_{K_p +
		1})^{K_p/(K_p+1)}
	\varphi(a_{p}t)^{1/(K_p  + 1)} \no \\
	& \le K_p  \kappa_{K_p}(\tM_{K_p +	1})^{K_p/(K_p+1)}
	\varphi(a_{K_p}t)^{1/(K_p  + 1)} \le  c_1(K_p)  \varphi(a_{K_p}t)^{1/(K_p+1)}.
	\label{e.psi-psit}
	\end{align}
	Here $c_1(m):= m \kappa_{m}\tM_{m + 1} \ge m \kappa_{m}(\tM_{m+ 1})^{m/(m+1)}$,	as $\tM_m\ge 1$ by assumption.
	Similarly, condition \eqref{eqn:mom} yields
	$|\E_{x_1,\ldots,x_{p}}\psi(x_1,\ldots,x_{p};\P_n)|\rho^{(p)}(x_1,\ldots,x_{p})
	\le c_1(K_p) \,.$		
	Using~\eqref{e.psi-psit} with $p$ replaced by $p+q$, we find 	$m^{(k_1,\ldots,k_{p+q})}(x_1,\ldots,x_{p+q};n)$
	differs from $\tm^{(k_1,\ldots,k_{p+q})}(x_1,\ldots,x_{p+q};n)$ by
	$c_1(K)\varphi(a_{K}t)^{1/(K+1)}$,	which is  fast-decreasing by \eqref{varphibd}.
	
	For any reals $A, B, \tA, \tB$, with $|\tB|\le |B|$ we have
	$|AB - \tA \tB| \leq |A(B - \tB)| + |(A - \tA)\tB|\le(|A|+|B|)(|B-\tB|+|A-\tA|)$.
	Hence, it follows that
	\begin{align*}
	& |m^{(k_1,\ldots,k_{p})}(x_1,\ldots,x_p;n)m^{(k_{p+1},\ldots,k_{q})}(x_{p+1},\ldots,x_{p+q};n)\\
	& -\tm^{(k_1,\ldots,k_{p})}(x_1,\ldots,x_p;n)\tm^{(k_{p+1},\ldots,k_{q})}(x_{p+1},\ldots,x_{p+q};n)| \\
	&\leq  (c_1(K_p)+c_1(K_q))\Big(c_1(K_p)\varphi(a_{K_p}t)^{1/(K_p+1)}+c_1(K_q)\varphi(a_{K_q}t)^{1/(K_q+1)}\Big) \\
	&\leq  c_2(K)\varphi(a_{K}t)^{1/(K+1)},
	\end{align*}
	with $c_2(m):=4(c_1(m))^2$  and where we note that $\varphi(a_{m}t)^{1/(m+1)}$
	is also fast-decreasing by \eqref{varphibd}.  The difference of correlation functions of the $\xi$-weighted measures is thus bounded by
	\begin{eqnarray}
	&  &\Bigl|m^{(k_1,\ldots,k_{p+q})}(x_1,\ldots,x_{p+q};n)- m^{(k_1,\ldots,k_{p})}(x_1,\ldots,x_{p};n)\, m^{(k_{p+1},\ldots,k_{p+q})}(x_{p+1},\ldots,x_{p+q};n) \Bigr| \no \\
	& \le & (c_1(K) + c_2(K))\varphi(a_kt)^{1/(K+1)} \no \\
	&  & + |\tm^{(k_1,\ldots,k_{p+q})}(x_1,\ldots,x_{p+q};n)-\tm^{(k_1,\ldots,k_p)}(x_1,\ldots,x_{p};n)\tm^{(k_{p+1},\ldots,k_{p+q})}(x_{p+1},\ldots,x_{p+q};n)|. \no \\ \label{eqn:bound_m_tm}
	\end{eqnarray}
The rest of the proof consists of bounding $|\tm^{(k_1,\ldots,k_{p+q})}-\tm^{(k_1,\ldots,k_p)}\tm^{(k_{p+1},\ldots,k_{p+q})}|$ by 
a  fast-decreasing function of $s$. In this regard we will consider the expansion ~\eqref{FMEpsip}
with $\psi(x_1,\ldots,x_{p};\P_n)$ replaced by $\tps(x_1,\ldots,x_{p};\P_n)$ as at \eqref{psitilde}
and similarly for $\tps(x_{p+1},\ldots,x_{p + q};\P_n)$ and $\tps(x_1,\ldots,x_{p+q};\P_n)$.
By Lemma~\ref{l.Radius} in the Appendix, $\tx(x_i,\P_n), 1 \leq i \leq p,$ have radii of stabilization bounded above by $t$
and also satisfy the power-growth condition \eqref{eqn:xibd}  since $|\tx | \leq |\xi|$. Thus	the pair $(\tx, \P)$ satisfies the assumptions of Lemma~\ref{l.FMEp}. 	The corresponding version of $\tps$, accounting for the fixed atoms of $\P_n$ is
\begin{equation*}
\tpsp(x_1,\ldots,x_p;\mu):=\prod_{i=1}^p \tx(x_i,\mu+\textstyle{\sum_{i=1}^p\delta_{x_i}})^{k_i}\,
\end{equation*}
and similarly for $\tpsp(x_{p+1},\ldots,x_{q};\P_n)$ and $\tpsp(x_1,\ldots,x_{p+q};\P_n)$.

Put $B_{t,n}(x_i):= B_t(x_i) \cap W_n$. Applying  \eqref{FMEpsip}, the multiplicative identity \eqref{eqn:palm_correlation} and \eqref{eqn:FME_local}, we obtain
\begin{align*}
	& \tm^{(k_1,\ldots,k_{p+q})}(x_1,\ldots,x_{p+q}) = \E^!_{x_1,\ldots,x_{p+q}} [\tpsp(x_1,\ldots,x_{p+q};\P_n)] \rho^{(p+q)}(x_1,\ldots,x_{p+q}) \no \\
	&= \sum_{l=0}^{\infty}\frac1{l!}\int_{(W_n)^l}
	D^{l}_{y_1,\ldots,y_l}\tpsp(o) \rho^{(l+p+q)}(x_1,\ldots,x_{p+q},y_1,\ldots,y_l)\, \md y_1\ldots \md y_l \no \\
	&= \sum_{l=0}^{\infty}\frac1{l!}\int_{(\cup_{i=1}^{p+q}
		B_{t,n}(x_i))^l}
	D^{l}_{y_1,\ldots,y_l}\tpsp(o)
	\rho^{(l+p+q)}(x_1,\ldots,x_{p+q},y_1,\ldots,y_l) \,\md y_1\ldots \md y_l\,.
\end{align*}
Applying  \eqref{eqn:FME_Kernels1} when $\mu$ is the null measure, this gives for $\alpha^{(p+q)}$ almost all $x_1,\ldots,x_{p+q}$
\begin{align}
	& \tm^{(k_1,\ldots,k_{p+q})}(x_1,\ldots,x_{p+q}) \no  \\ 
	&=  \sum_{l=0}^{\infty}\frac1{l!}\sum_{j=0}^l\frac{l!}{j!(l-j)!}\int_{(\cup_{i=1}^p B_{t,n}(x_i))^j \times(\cup_{i=1}^q B_{t,n}(x_{p+i}))^{l-j}}
	D^{l}_{y_1,\ldots,y_l}\tpsp(x_1,\ldots,x_{p+q};o)  \no \\
	&\hspace{6em}\times \rho^{(l+p+q)}(x_1,\ldots,x_{p+q},y_1,\ldots,y_l) \md y_1\ldots \md y_l  \no \\
	&=\sum_{l=0}^{\infty}\sum_{j=0}^l\frac{1}{j!(l-j)!}\int_{(\cup_{i=1}^p
		B_{t,n}(x_i))^j \times(\cup_{i=1}^q B_{t,n}(x_{p+i}))^{l-j}}
	\no \\[0.5ex]
	&\hspace{2em}\times\sum_{J \subset [l]}(-1)^{l-|J|}\tpsp(x_1,\ldots,x_{p+q};\textstyle{\sum_{j \in J}\delta_{y_j}}) \times \rho^{(l+p+q)}(x_1,\ldots,x_{p+q},y_1,\ldots,y_l) \md y_1 \ldots \md y_l. \label{eqn:m_p+q}
	\end{align}

To compare the $(p+q)$th correlation functions of the $\xi$-weighted measures with the product of their $p$th and $q$th correlation functions, we shall use the fact  that
	$R^{\tilde\xi}(x_i, \P_n) \in (0, t]$ (cf. Lemma~\ref{l.Radius}) implies the following factorization, which holds for  $y_1,\ldots,y_j \in \cup_{i=1}^pB_{t}(x_i)$ and $y_{j+1},\ldots,y_l \in \cup_{i=1}^qB_{t}(x_{p+i})$,
	with $t \in (1,s/4)$ (making $\cup_{i=1}^pB_{t}(x_i)$ and $\cup_{i=1}^qB_{t}(x_{p+i})$ disjoint):
	\begin{equation}
	\tpsp(x_1,\ldots,x_{p+q};\textstyle{\sum_{i=1}^l\delta_{y_i}}) = \tpsp(x_1,\ldots,x_p;\textstyle{\sum_{i=1}^j\delta_{y_i}})\tpsp(x_{p+1},\ldots,x_{p+q};\textstyle{\sum_{i=j+1}^l\delta_{y_i}})\,. \label{eqn:prod_psi}
	\end{equation}
	
	Using the expansion \eqref{FMEpsip} along with \eqref{eqn:prod_psi}, we next derive an expansion for the product of $p$th and $q$th correlation functions of the  $\xi$-weighted measures.
	Recalling the multiplicative identity \eqref{eqn:palm_correlation} as well as the identity  $\E_{x_1,\ldots,x_p}[\psi(\P_n)]=\E^!_{x_1,\ldots,x_{p}}[\psip(\P_n)]$ (cf. \eqref{e.psip}), we obtain
	\begin{align}
	& \tm^{(k_1,\ldots,k_{p})}(x_1,\ldots,x_p)\tm^{(k_{p+1},\ldots,k_{q})}(x_{p+1},\ldots,x_{p+q}) \no\\
	&=  \E^!_{x_1,\ldots,x_p}[\tpsp(x_1,\ldots,x_p;\P_n)]
	\E^!_{x_{p+1},\ldots,x_{p+q}}[\tpsp(x_{p+1},\ldots,x_{p+q};\P_n)] \no \\
	&\hspace{20em}\times \rho^{(p)}(x_1,\ldots,x_p)\rho^{(q)}(x_{p+1},\ldots,x_{p+q}) \no
	\end{align}
\begin{align}
	& =  \sum_{l_1,l_2 = 0}^{\infty}\frac1{l_1!l_2!}\int_{(\cup_{i=1}^pB_{t,n}(x_i))^{l_1} \times(\cup_{i=1}^qB_{t,n}(x_{p+i}))^{l_2}}
	\hspace{-3em}	D^{l_1}_{y_1,\ldots,y_{l_1}}\tpsp(x_1,\ldots,x_p;o) \ D^{l_2}_{z_1,\ldots,z_{l_2}}\tpsp(x_{p+1},\ldots,x_{p+q};o) \no \\[0.5ex]
	&\hspace{1em}\times \rho^{(l_1+p)}(x_1,\ldots,x_p,y_1,\ldots,y_{l_1})\rho^{(l_2+q)}(x_{p+1},\ldots,x_{p+q},z_1,\ldots,z_{l_2})\, \md y_1\ldots \md y_{l_1} \md z_1 \ldots \md z_{l_2}\,. \no \\
	\label{e.tmp-series}	
	\end{align}
Applying  \eqref{eqn:FME_Kernels1} once more for  $\mu$ the null measure, this gives
\begin{align}
	& \tm^{(k_1,\ldots,k_{p})}(x_1,\ldots,x_p)\tm^{(k_{p+1},\ldots,k_{q})}(x_{p+1},\ldots,x_{p+q})\no \\
	& = \sum_{l_1,l_2=0}^{\infty}\frac{1}{l_1!l_2!}\int_{(\cup_{i=1}^pB_{t,n}(x_i))^{l_1} \times(\cup_{i=1}^qB_{t,n}(x_{p+i}))^{l_2}}\sum_{J_1 \subset [l_1],J_2 \subset [l_2]}(-1)^{l_1+l_2-|J_1|-|J_2|}\no \\[0.2ex]
	&\hspace{1em} \times \tpsp(x_1,\ldots,x_p;\textstyle{\sum_{i\in J_1}\delta_{y_i}}) \tpsp(x_{p+1},\ldots,x_{p+q};\textstyle{\sum_{i\in J_2}\delta_{z_i}})\no \\
	&\hspace{1em} \times  \rho^{(l_1+p)}(x_1,\ldots,x_p,y_1,\ldots,y_{l_1}) \rho^{(l_2+q)}(x_{p+1},\ldots,x_{p+q},z_1,\ldots,z_{l_2})\, \md y_1\ldots \md y_{l_1} \md z_1 \ldots \md z_{l_2} \no \\
		& = \sum_{l=0}^{\infty}\sum_{j=0}^l\frac{1}{j!(l-j)!}\int_{(\cup_{i=1}^pB_{t,n}(x_i))^j \times(\cup_{i=1}^qB_{t,n}(x_{p+i}))^{l-j}}\sum_{J_1 \subset [j],J_2 \subset [l]\setminus[j]}(-1)^{l-|J_1|-|J_2|}\no \\[0.5ex]
	&\hspace{1em}\times \tpsp(x_1,\ldots,x_p;\textstyle{\sum_{i \in J_1}\delta_{y_i}}) \tpsp(x_{p+1},\ldots,x_{p+q};\textstyle{\sum_{i \in J_2}\delta_{y_i}}) \no \\
	&\hspace{1em}\times \rho^{(j+p)}(x_1,\ldots,x_p,y_1,\ldots,y_j)\rho^{(l-j+q)}(x_{p+1},\ldots,x_{p+q},y_{j+1},\ldots,y_l)\, \md y_1\ldots \md y_l \no \\
	& =  \sum_{l=0}^{\infty}\sum_{j=0}^l\frac{1}{j!(l-j)!}\int_{(\cup_{i=1}^pB_{t,n}(x_i))^j \times(\cup_{i=1}^qB_{t,n}(x_{p+i}))^{l-j}}\sum_{J \subset [l]}(-1)^{l-|J|}\tpsp(x_1,\ldots,x_{p+q};\textstyle{\sum_{i \in J}\delta_{y_i}}) \no \\[0.5ex]
	&\hspace{2em}\times \rho^{(j+p)}(x_1,\ldots,x_p,y_1,\ldots,y_j)\rho^{(l-j+q)}(x_{p+1},\ldots,x_{p+q},y_{j+1},\ldots,y_l)\, \md y_1\ldots \md y_l, \label{eqn:m_p,q}
\end{align}
where we have used \eqref{eqn:prod_psi} in the last equality.

Now we estimate the difference of \eqref{eqn:m_p+q} and \eqref{eqn:m_p,q}. Applying~\eqref{eqn:clustering_condition} and replacing $B_{t,n}(x_i)$ with $B_{t}(x_i)$, we obtain
\begin{align}
	&|\tm^{(k_1,\ldots,k_{p+q})}(x_1,\ldots,x_{p+q})-\tm^{(k_1,\ldots,k_{p})}(x_1,\ldots,x_p)\tm^{(k_{p+1},\ldots,k_{q})}(x_{p+1},\ldots,x_{p+q})| \no \\
	& \leq  \phi({s\over 2} )\sum_{l=0}^{\infty}\sum_{j=0}^l\frac{C_{l+p+q}}{j!(l-j)!}\no \\
	&\hspace{2em}\times\int_{(\cup_{i=1}^pB_{t}(x_i))^j \times(\cup_{i=1}^qB_{t}(x_{p+i}))^{l-j }}\sum_{J \subset [l]}|\tpsp(x_1,\ldots,x_{p+q};\textstyle{\sum_{i \in J}\delta_{y_i}})|\, \md y_1\ldots \md y_l.
\end{align}
Recalling \eqref{eqn:prod_psi}  and \eqref{BD}, as well as the definitions of $K_p, K_q,$ and $K$ at \eqref{defKp},
we bound \\ $\sum_{J \subset [l]}|\tpsp(x_1,\ldots,x_{p+q};\sum_{i \in J}\delta_{y_i})|$  by
	$2^l(\hat{c}t)^{jK_p+(l-j)K_q+K}$, where $\hat{c}t  \in [1, \infty)$  holds since $\hat{c}\in [1, \infty)$ in \eqref{eqn:xibd}.  This gives
	\begin{align}
	&|\tm^{(k_1,\ldots,k_{p+q})}(x_1,\ldots,x_{p+q})-\tm^{(k_1,\ldots,k_{p})}(x_1,\ldots,x_p)\tm^{(k_{p+1},\ldots,k_{q})}(x_{p+1},\ldots,x_{p+q})| \no
\end{align}
\begin{align}
	& \leq \phi({s\over 2} )\sum_{l=0}^{\infty}\sum_{j=0}^l\frac{C_{l+p+q}}{j!(l-j)!}\int_{(\cup_{i=1}^pB_{t}(x_i))^j
		\times(\cup_{i=1}^qB_{t}(x_{p+i}))^{l-j}}2^l(\hat{c}t)^{jK_p+(l-j)K_q+K}\, \md y_1\ldots \md y_l.
		\end{align}
Consequently, bounding $K_p$ and $K_q$ by $K$ we obtain
\begin{align}
	&|\tm^{(k_1,\ldots,k_{p+q})}(x_1,\ldots,x_{p+q})-\tm^{(k_1,\ldots,k_{p})}(x_1,\ldots,x_p)\tm^{(k_{p+1},\ldots,k_{q})}(x_{p+1},\ldots,x_{p+q})| \no \\
		& \leq \phi({s\over 2} )\sum_{l=0}^{\infty}C_{l+p+q}2^l(\hat{c}t)^{(l+1)K}((p+q)\theta_dt^d)^l\sum_{j=0}^l\frac{1}{j!(l-j)!} \no \\[1ex]
	& \leq \phi({s\over 2}
	)\sum_{l=0}^{\infty}\frac{C_{l+p+q}}{l!}4^l(\hat{c}t)^{(l+1)K}((p+q)\theta_dt^d)^l  \leq \phi({s\over 2}  )\sum_{l=0}^{\infty}\frac{C_{l+K}}{l!}4^l(\hat{c}t)^{(l+1)K}(K\theta_dt^d)^l,
	\label{eqn:clustering_moments_proof}
	\end{align}
where $\theta_d := \pi^{d/2}/\Gamma(d/2 + 1)$ is the volume of the unit ball in $\R^d$ 
and where the last inequality uses $p+q\le K$.   The bound \eqref{eqn:sum} yields $C_{l + K} = O((l + K)^{a(l + K)})$.  Thus there
are constants $c_1, c_2$ and $c_3$ depending only on $a, d$ and $K$ such that	
	$$
	\sum_{l=0}^{\infty}\frac{C_{l+K}}{l!}4^l(\hat{c}t)^{(l+1)K}(K\theta_dt^d)^l  \leq t^K\sum_{l=0}^{\infty} \frac{ c_1 c_2^l l^{c_3} (t^{K + d})^l \cdot l^{al} }  {l! }.
	$$
	By Stirling's formula,  there are constants $c_4, c_5$
	and $c_6$ depending only on $a, d$ and $K$ 	such that
	$$
	t^K\sum_{l=0}^{\infty}\frac{C_{l+K}}{l!}4^l(\hat{c}t)^{lK}(K\theta_dt^d)^l
	\leq  t^K\sum_{l=0}^{\infty} \frac{ c_4 c_5^l l^{c_6} (t^{K + d})^l }  { (\lfloor l(1 - a) \rfloor )!   },
	$$
	where for   $r \in \R$, $\lfloor r \rfloor$ is the greatest integer less than $r$.   We compute
	\begin{eqnarray}
	&  & t^K	\sum_{l=0}^{\infty}\frac{C_{l+K}}{l!}4^l(\hat{c}t)^{lK}(K\theta_dt^d)^l
	\leq t^K \sum_{n=0}^{\infty}  \sum_{ \{l: \ \lfloor l(1 - a) \rfloor = n  \}}   \frac{ c_4 c_5^l l^{c_6} (t^{K + d})^l  }  {n! }
	\no
	\end{eqnarray}
	\begin{eqnarray}
	& \leq &  t^K \sum_{n=0}^{\infty}   \frac{ c_4 c_5^n n^{c_6} (t^{K + d})^{(n + 1)/( 1 - a) } }  { (1 - a)  n! } \leq c_7 \exp(c_8 t^{(K + d)/( 1 - a) }) \label{3.25}
	\end{eqnarray}
	where $c_7$ and $c_8$ depend only on $a, d$ and $K$.
	
	Recalling from \eqref{defsprime} that $t:= (s/4)^{b(1 - a)/(2(K + d))}$  we obtain
	$$ \sum_{l=0}^{\infty}\frac{C_{l+K}}{l!}4^l(\hat{c}t)^{(l+1)K}(K\theta_dt^d)^l  \leq c_7 \exp\left(c_8 ( {s \over 4})^{ {b \over 2}}\right).$$
	By \eqref{phibd}, there is a constant $c_9$ depending only on $a$ such that for all $s$ we have $\phi(s) \leq c_9 \exp(-s^b/c_9)$. Combining this with \eqref{eqn:clustering_moments_proof}
		and \eqref{3.25} gives
	$$
	|\tm^{(k_1,\ldots,k_{p+q})}-\tm^{(k_1,\ldots,k_p)}\tm^{(k_{p+1},\ldots,k_{p+q})}|
	\leq c_7 c_9 \exp \left( \frac{ - (s/2)^b} {c_9}  +   c_8 ( {s \over 4})^{ {b \over 2}} \right).
	$$
	This along with \eqref{eqn:bound_m_tm} establishes \eqref{e.clustering-generalized-mixed-monents} when $(\xi, \P)$ is an admissible pair of class (A2).
	
Now we establish  \eqref{e.clustering-generalized-mixed-monents} when $(\xi, \P)$ is  of class (A1).  Let $k$ be as in  \eqref{eqn:score_U_statistic}.
 Follow the  arguments for case (A2)  word for word using that $\sup_{x \in \P} R^{\xi}(x,\P) \leq r$.  Notice  that for $l \in ((k-1)K, \infty)$ the summands in ~\eqref{eqn:m_p+q} vanish.  Likewise, when
  $l_1 \in ((k-1)K_p,  \infty)$ and $l_2 \in ((k-1)K_q, \infty),$ the respective summands in \eqref{e.tmp-series} vanish. It follows that for $l \in ((k-1)K, \infty)$ the summands in
\eqref{eqn:clustering_moments_proof} all vanish.
 The finiteness of $\tilde{C}_{K}$ in expression \eqref{e.clustering-generalized-mixed-monents}
  is immediate, without requiring decay rates for $\phi$ or growth-rate bounds on $C_k$.  Thus \eqref{e.clustering-generalized-mixed-monents} holds when $(\xi, \P)$ is of class (A1).	\qed

\section{\bf Proof  of main results}
\label{sec:proofs_main}
	
We provide the proofs of Theorems \ref{prop:variance_upper_bound},  \ref{prop:variance_lower_bound}, and \ref{thm:main0} in this order.
\subsection{\bf Proof of Theorem \ref{prop:variance_upper_bound}}
\label{sec:proof_expectation}

\subsubsection{Proof of expectation asymptotics \eqref{expasy}}
The definition of the Palm probabilities gives
$$
\E \mu_n^\xi(f) = \int_{W_n} f(n^{-1/d} u) \E_u\xi(u,\P_n) \rho^{(1)}(u)  \,\md u.
$$
By the stationarity of $\P$ and translation invariance of $\xi$,
		we have $\E_{\0} \xi(\0,\P) = \E_u \xi(u, \P)$.   Thus,
		\begin{align*}
	& \Big| n^{-1}\E \mu_n^\xi(f) - \E_{\0} \xi(\0, \P) \rho^{(1)}(\0) \int_{W_1}f(x)\,\md x\Big| \no \\[1ex]
	& = \Big| n^{-1} \int_{W_n} f(n^{-1/d} u) \{\E_u \xi(u, \P_n) \rho^{(1)}(u) - \E_{\0} \xi(\0,\P)\rho^{(1)}(\0) \}  \,\md u \Big|\no \\
	& = \Big| n^{-1} \int_{W_n} f(n^{-1/d} u) \E_u[ (\xi(u, \P_n) - \xi(u,\P)) \rho^{(1)}(u) ]  \md u \Big|  \no \\
	& \leq   \|f \|_{\infty}  n^{-1}    \int_{W_n}  \E_u[ |\xi(u, \P_n) - \xi(u,\P)|
        \1[  \max( R^\xi(u, \P),  R^\xi(u, \P_n))     \geq d(u, \partial W_n)]] \rho^{(1)}(u) \md u \no  \\
    & \leq \|f\|_{\infty}   n^{-1}   \int_{W_n}  \md u \,  \rho^{(1)}(u) \,  \times  \no \\
    & \qquad \E_u \bigl[ |\xi(u, \P_n) - \xi(u,\P)| \times \bigl( \1[ R^\xi(u, \P)  \geq d(u, \partial W_n)] +  \1[ R^\xi(u, \P_n)  \geq d(u, \partial W_n)] \bigr) \bigr]  \no \\
	& \leq 4 \kappa_1 \|f \|_{\infty}  n^{-1}  \tM_{p} \int_{W_n}  (\varphi(a_1d(u,\partial W_n)))^{1/q} \md u,
	\end{align*}
	where the last inequality follows from the H\"older inequality, \eqref{stab}, the bound  \eqref{eqn:corr_bounded},  the $p$-moment
		condition~\eqref{eqn:mom} (recall $p \in (1, \infty)$ and $\tilde M_p \in [1, \infty)$) and where $1/p + 1/q = 1$.
	 By \eqref{varphibd}, the bound \eqref{expasy} follows at once from
	$$ \int_{W_n}  (\varphi(a_1d(u,\partial W_n)))^{1/q} \md u 
	= O(n^{(d-1)/d}).$$
	 If $\xi$ satisfies \eqref{stab}, but not \eqref{varphibd}, then by  the bounded convergence theorem, we have
	\begin{align*} 
	 \limsup_{n \to \infty} n^{-1}  \int_{W_n} ( \varphi(a_1 d(u, \partial W_n)))^{1/q} \md u
	 = \limsup_{n \to \infty} \int_{W_1} (\varphi(a_1 n^{1/d} d(z, \partial W_1)))^{1/q} \md z = 0.
	\end{align*}
	Consequently,  we have expectation asymptotics under \eqref{stab} as follows :
	$$ \Big| n^{-1}\E \mu_n^\xi(f) - \E_{\0} \xi(\0, \P) \rho^{(1)}(\0) \int_{W_1}f(x)\,\md x\Big| = o(1).$$

\subsubsection{Proof of variance  asymptotics  \eqref{eqn:var}}
\label{sss.proof-variance}
Recall the definition of the correlation functions~\eqref{eqn:mixedmomentn} of the $\xi$-weighted measures.  We have

	\begin{align}
	& \Var{\mu_n^{\xi}(f)} =
	\E{\sum_{x \in \P_n}f(n^{-1/d} x)^2\xi^2(x,\P_n})\no \\
	&\hspace{1em} + \E{\!\!\!\sum_{x,y \in \P_n, x \neq
			y}f(n^{-1/d}x)f(n^{-1/d}y)\xi(x,\P_n)\xi(y,\P_n)} -
	\Bigl(\E{\sum_{x \in \P_n}f(n^{-1/d}x)\xi(x,\P_n)}\Bigr)^2 \no	\\
		& =  \int_{W_n} f(n^{-1/d}u)^2 \sE_u(\xi^2(u, \P_n ))  \rho^{(1)}(u)  \, \md u
	\label{eqn:variance_exp0} \\
		&\hspace{1em} +
	\int_{W_n \times W_n}f(n^{-1/d}u) f(n^{-1/d}v)\bigl( m_{(2)}(u,v;n) - m_{(1)}(u;n)m_{(1)}(v;n)\bigr) \md u \md v.  \label{eqn:variance_exp}
	\end{align}
	Since $\xi$ satisfies  
		the  $p$-moment condition \eqref{eqn:mom} for $p > 2$, we
		have that $\xi^2$ satisfies the $p$-moment condition for $p > 1$.
		Also, $\xi$ and $\xi^2$ have the same radius of
		stabilization. 
	Thus, the proof of expectation asymptotics, with $\xi$ replaced by $\xi^2$, shows that the first term in~\eqref{eqn:variance_exp0}, multiplied by $n^{-1}$, converges to~
	$$\E_{\0} \xi^2(\0, \P) \rho^{(1)}(\0) \int_{W_1}f(x)^2\,\md x\,;$$
	cf. expectation asymptotics ~\eqref{expasy}. Setting $x = n^{-1/d}u$ and $z =v-u= v-n^{1/d}x$,
the second term in~\eqref{eqn:variance_exp}, multiplied by $n^{-1}$, may  be
	rewritten as
	\begin{align}\label{e.integral_Wnx}
	\int_{W_1} \int_{W_n-n^{1/d}x} &f(x + n^{-1/d}z) f(x)\\
	&\times [m_{(2)}(n^{1/d}x,n^{1/d}x+z;n) -
	m_{(1)}(n^{1/d}x;n)m_{(1)}(n^{1/d}x+z;n)]\,\md z\md x. \no
	\end{align}

	Setting  $\P_n^x:=\P\cap (W_n - n^{1/d}x)$,
	the translation invariance of $\xi$ and stationarity of $\P$ yields	
	
	\begin{align*}
	m_{(2)}(n^{1/d}x,n^{1/d}x+z;n)&=m_{(2)}(\0,z;\P^x_n)\\
	m_{(1)}(n^{1/d}x;n)&= m_{(1)}(\0;\P_n^x)\\
	m_{(1)}(n^{1/d}x+z;n)&=m_{(1)}(z;\P_n^x)\,.
	\end{align*}
	Putting aside for the moment technical details, one expects that the
	above moments  converge to $m_{(2)}(\0,z)$, $m_{(1)}(\0)$ and
	$m_{(1)}(z)=m_{(1)}(\0)$, respectively, when $n\to\infty$.
	Moreover, splitting  the inner integral
	in~\eqref{e.integral_Wnx} into two terms
	\begin{equation}\label{e.z<>M}
	\int_{W_n-n^{1/d}x}(\dots)\md z=
	\int_{W_n-n^{1/d}x }\1[|z|\le M](\dots)\md z+
	\int_{W_n-n^{1/d}x }\1[|z|> M](\dots)\md z\,
	\end{equation}
	for any $M>0$, we see  (at least when  $f$ is continuous)
	that the first term in the right-hand side of~\eqref{e.z<>M} converges  to the desired value
	$$\int_{\R^d}f(x)^2[m_{(2)}(\0,z)-m_{(1)}(\0)^2]\,\md z$$
	when first $n\to\infty$ and then  $M\to\infty$.
	By the fast decay of the second-order correlations of the $\xi$-weighted measures, i.e., by \eqref{e.clustering-generalized-mixed-monents} with  $ p =  q = k_1 = k_2 = 1$ and all $n \in \N \cup\{\infty\}$, the absolute  value of the second term in~\eqref{e.z<>M}
	can be bounded uniformly in $n$ by
	
	$$\|f\|^2_\infty\tilde C_{2}\int_{|z|> M}\tilde \phi(\tilde c_2z)\,\md z\,,$$
	which goes to~0 when  $M\to\infty$ since $\tilde\phi(\cdot)$ is
	fast-decreasing (and thus integrable).
	
	To formally justify the  above statements we need the
	following lemma. Denote
	$$h_n^\xi(x, z):=m_{(2)}(\0,z;\P^x_n)-m_{(1)}(\0;\P_n^x)\,m_{(1)}(z;\P_n^x)\,.$$
	\begin{lemma}\label{l.lemma-convergence-h}
		Assume that translation invariant score function  $\xi$ on the input
		process $\P$ satisfies~\eqref{stab} and the $p$-moment condition~\eqref{eqn:mom} for  $p \in (2, \infty).$
		Then $h_n^\xi(x, z)$ is uniformly bounded
		$$\sup_{n \le\infty}\sup_{x\in W_1}\sup_{z\in W_n-n^{1/d}x}|h_n^\xi(x, z)|\le C_h<\infty$$
		for some constant $C_h$ and
		$$\lim_{n\to\infty} h_n^\xi(x, z)=h_\infty^\xi(x, z)=m_{(2)}(\0,z)-(m_{(1)}(\0))^2\,.$$
	\end{lemma}
	\begin{proof}
		Denote  $X_n := \xi(\0, \P_n^x)$, $Y_n : = \xi(z, \P_n^x)$,
		$X:= \xi(\0, \P )$,  and $Y:= \xi(z, \P )$.
		We shall prove first that all expectations
		$\E_{\0,z}(X_n^2)$, $\E_{\0,z}(Y_n^2)$, $\E_{\0,z} (X^2)$
		$\E_{\0,z}(Y^2)$, $\E_{\0}|X_n|$, $\E_{z}|Y_n|$, $\E_{\0} |X|$ and $\E_{z}|Y|$ are uniformly bounded.
		Indeed, by the H\"older inequality
		\begin{equation}\label{e.3moment}
		\E_{\0,z}(X_n^2)\le (\E_{\0,z}|X_n|^p)^{2/p}=
		(\E_{n^{1/d}x,z} |\xi(n^{1/d}x,\P_n)|^p)^{2/p}
		\le \tilde{M}_p^{2/p}
		\end{equation}
		where in the last inequality we have used  $p$-moment condition~\eqref{eqn:mom} for  $p > 2$.
		Similarly  $\E_{\0,z}(Y_n^2)$ and $\E_{\0,z} (X^2)$, $\E_{\0,z}(Y^2)$ are
		bounded by $\tilde{M}_p^{2/p}$. Again using $p$-moment condition~\eqref{eqn:mom}, we obtain
		$$
		\E_{\0}|X_n|\le(\E_{\0}(X_n)^2)^{1/2}\le
		(\E_{n^{1/d}x}|\xi^2(n^{1/d}x,\P_n)|)^{1/2}
		\le \tilde{M}_p^{1/2}
		$$
		and similarly for $\E_{z}|Y_n|$, $\E_{\0} |X|$ and $\E_{z}|Y|$.
		This proves the uniform bound of $|h_n^\xi(x, z)|$.
		To prove the convergence notice that
		\begin{align}
		|m_{(2)}(\0,z;\P^x_n)-&m_{(2)}(\0,z)|=|\E_{\0,z} (X_n Y_n) - \E_{\0,z}
		(XY)|\rho^{(2)}(\0,z)\\
		&\leq \kappa_2\bigr(\E_{\0,z} |X_n Y_n - X_n Y| +
		\E_{\0,z} |X_nY -XY|\bigl) \no \\
		& \leq \kappa_2(\E_{\0,z}(X^2_n)\E_{\0,z}(Y_n - Y)^2)^{1/2} +
		\kappa_2(\E_{\0,z}(Y^2)\E_{\0,z}(X_n - X)^2)^{1/2}\,, \label{e.m2limit}
		\end{align}
		where $\kappa_2$ bounds the second-order correlation function as at~\eqref{eqn:corr_bounded}.
		We have already proved that   $\E_{\0,z}(X^2_n)$, $\E_{\0,z}(Y^2)$ are
		bounded.
		Moreover
		\begin{align*}
		\E_{\0,z}(X_n - X)^2&=\E_{\0,z}((X_n - X)^2\1[X_n\not=X])\\
		&\le \E_{\0,z}(X_n^2\1[X_n\not=X])+
		2\E_{\0,z}(|X_nX|\1[X_n\not=X])+\E_{\0,z}(X^2\1[X_n\not=X])\,.
		\end{align*}
		The H\"older inequality gives for $p > 2$ and $2/p + 1/q = 1$,
		\begin{align*}
		\E_{\0,z}(X_n^2\1[X_n\not=X])&\le
		(\E_{\0,z}(X_n^p))^{2/p}(\mP_{\0,z}(X_n\not=X))^{1/q} \\
		\E_{\0,z}(|X_nX|\1[X_n\not=X])&\le
		(\E_{\0,z}(X_n^p)\E_{\0,z}(X^p))^{1/p}(\mP_{\0,z}(X_n\not=X))^{1/q}\, \\
		\E_{\0,z}(X^2\1[X_n\not=X])&\le
		(\E_{\0,z}(X^p)^{2/p}(\mP_{\0,z}(X_n\not=X))^{1/q}\,.
		\end{align*}
		The $p\,$th   moment of $X_n$ and $X$ under $\E_{\0,z}$ can be
		bounded by $\tilde{M}_p$ using the $p$-moment condition~\eqref{eqn:mom} with
		$p > 2$ as in~\eqref{e.3moment}.
		Stabilization  \eqref{stab} with $l=2$ gives
		\begin{equation}\label{e:varphi-bound}
		\mP_{\0,z}(X_n\not=X)\le \mP_{\0,z}(  ( \max( R^\xi(u, \P),  R^\xi(u, \P_n))> n^{1/d}
		d(x, \partial W_1))\le 2\varphi (a_2  n^{1/d}
		d(x, \partial W_1))
		\end{equation}
		with the right-hand side converging to~0 for all $x\not\in\partial
		W_1$.
		This proves that $\E_{\0,z}(X_n-X)^2$ and (by the very same arguments)
		$\E_{\0,z}(Y_n-Y)^2$ converge to 0 as $n\to\infty$ for all
		$x\not\in\partial W_1$.
		Concluding this part of the proof, we have shown that the expression in~\eqref{e.m2limit}
		converges to~0 and thus $m_{(2)}(\0,z;\P^x_n)$  converges to $m_{(2)}(\0,z)$. Using similar  arguments, we derive
		\begin{align*}
		|m_{(1)}(\0,\P^x_n)-m_{(1)}(\0)|&=|\E_{\0} (X_n) - \E_{\0}(X)|\rho^{(1)}(\0)\\
		&\le \kappa_1((\E_{\0} (X_n)^2)^{1/2}+(\E_{\0} (X^2))^{1/2})
		(\mP_{\0}(X_n\not=X)))^{1/2},
		\end{align*}
		by the $p$-moment condition~\eqref{eqn:mom} and the stabilization property \eqref{stab} for  $p=1$
		one can show that $m_{(1)}(\0,\P^x_n)$ converges to $m_{(1)}(\0)$
		uniformly in $x$ for all $x\in W_1\setminus\partial W_1$.
		Exactly the same arguments assure convergence of  $m_{(1)}(z,\P^x_n)$
		to $m_{(1)}(z)=m_{(1)}(\0)$. This concludes the proof of
		Lemma~\ref{l.lemma-convergence-h}.
	\end{proof}

	In order to complete the proof of variance asymptotics
	for general $f \in {\cal B} (W_1)$ (not necessarily continuous) we
	use arguments borrowed from the  proof of
	\cite[Theorem~2.1]{PeEJP}.
	Recall that $x \in W_1$  is a Lebesgue point for $f$ if $ ( {\rm{Vol}_d } B_{\epsilon}(x))^{-1} \int_{B_{\epsilon}(x)} |f(z) - f(x)| \md z \to 0$  as
	$\epsilon \to 0$.
	Denote by $\C_f$ all Lebesgue points of $f$ in $ W_1$.
	By the Lebesgue density theorem almost every $x \in W_1$ is a
	Lebesgue point of $f$
	and thus for any $M>0$ and $n$ large enough the double integral in~\eqref{e.integral_Wnx} is equal to
	\begin{align*}
	& \int_{W_1} \1[x\in\C_f] f(x) \int_{W_n-n^{1/d}x} f(x + n^{-1/d}z)h_n^\xi(x,z)\,\md z\md x \\
	&=\int_{W_1} \1[x\in\C_f] f(x) \int_{|z|\le M} f(x + n^{-1/d}z)h_n^\xi(x,z)\,\md z\md x \\
	&\hspace{1em}+
	\int_{W_1}\1[x\in\C_f] f(x) \int_{W_n-n^{1/d}x}1(|z|>M) f(x +
	n^{-1/d}z)h_n^\xi(x,z)\,\md z\md x\,.
	\end{align*}
 As already explained, by the fast decay of the  second-order correlations of the $\xi$-weighted measures,
	the second term converges to~0 as first $n \to \infty$ and then $M
		\to \infty$.
	Considering  the first term, by  the  uniform boundedness of
	$h_n^\xi(x,z)$, using the  dominated
	convergence theorem,  it is enough to prove for any Lebesgue point $x$
	of  $f$ and  fixed $M$ that
	$$\lim_{n \to \infty} \int_{|z|<M } h_n^\xi(x, z) f( x +
	n^{-1/d}z)\,\md z=
	f(x) \int_{|z|<M } h_\infty^\xi(x, z)\,\md z\,.
	$$
	In this regard notice that
	\begin{align*}
	&\int_{|z|<M } |h_n^\xi(x, z)f( x +
	n^{-1/d}z)-  h_\infty^\xi(x, z) f( x)|\,\md z\\
	&\le \int_{|z|<M } C_h\times|f(x + n^{-1/d}z)- f( x)|+
	|h_n^\xi(x, z)-h_\infty^\xi(x,z)|\times\|f\|_\infty
	\,\md z \\
	&\le C_h n \int_{|z|<n^{-1/d}M} |f(x + z)- f( x)|\,\md z+
	\|f\|_\infty\times  \int_{|z|<M }  |h_n^\xi(x,
	z)-h_\infty^\xi(x,z)|\,\md z\,.
	\end{align*}
	Both terms converge to 0 as $n\to\infty$: the first  since $x$ is a
	Lebesgue point of $x$,  the second  by
	the dominated convergence of $h_n^\xi(x,
	z)$; cf. Lemma~\ref{l.lemma-convergence-h}.
	Note that $\int_{W_1}\int_{\R^2}|h^\xi_\infty(x,z)|\,\md z\md
	x<\infty$, which follows again from the fast decay of the second-order correlations of the $\xi$-weighted measure (\eqref{e.clustering-generalized-mixed-monents} with  $ p =  q = k_1 = k_2 = 1$ and all $n \in \N \cup\{\infty\}$). Letting $M$ go 	to infinity	in $\int_{W_1}f^2(x)\int_{|z|<M}h^\xi_\infty(x,z)\,\md z\md x$  completes the proof of
	variance asymptotics.
	\qed
	
	\subsection{\bf Proof of Theorem \ref{prop:variance_lower_bound}}  The proof is inspired by the proofs
of \cite[Propositions 1 and 2]{Martin80}. By the refined Campbell theorem and stationarity of $\P$, we have
	\begin{align}
	n^{-1}\Var \H_n^\xi(\P) & = \int_{W_n} \E_{x}\xi^2(x; \P)\rho^{(1)}(x)  \md x + \int_{W_n} \int_{W_n}[m_{(2)}(x,y) - m_{(1)}(x) m_{(1)}(y)] \md y \md x \no \\
    \label{doub1}
	& = \E_{\0} \xi^2(\0, \P) \rho^{(1)}(\0)
	+ n^{-1} \int_{W_n} \int_{W_n}(m_{(2)}(x,y) - m_{(1)} (x) m_{(1)} (y)) \md y \md x.
	\end{align}
	Writing $c(x,y):= m_{(2)}(x,y) - m_{(1)}(x) m_{(1)}(y)$, the double integral in \eqref{doub1} becomes ($z = y - x$)
	\begin{align*}
	n^{-1}\int_{W_n} \int_{W_n}(m_{(2)}(x,y) - m_{(1)} (x) m_{(1)} (y)) \md y \md x
	& =  n^{-1}\int_{W_n} \int_{\R^d}  c(\0,z) \1[x + z \in W_n] \md z \md x \\
	& =  n^{-1}\int_{W_n} \int_{\R^d}  c(\0,z) \1[x \in W_n - z] \md z \md x.
	\end{align*}	
	Write $\1[x \in W_n - z]$ as $1 - \1[x \in (W_n - z)^c]$ to obtain
	\begin{align*}
	&n^{-1}\int_{W_n} \int_{W_n}(m_{(2)}(x,y) - m_{(1)} (x) m_{(1)} (y)) \md y \md x \\
	& =  \int_{\R^d} c(\0, z) \md z -   n^{-1} \int_{\R^d} \int_{W_n}  c(\0,z) \1[x \in \R^d \setminus (W_n- z)] \md x \md z.
	\end{align*}
	From  \eqref{gamm}, we have that  $\gamma_{W_n}(z) := {\rm{Vol}}_d(W_n \cap  (\R^d \setminus (W_n- z)))$ and thus rewrite \eqref{doub1} as
		\be \label{doub2}
	n^{-1} \Var \H_n^\xi(\P) = \E_{\0} \xi^2(\0, \P) \rho^{(1)}(\0) +  \int_{\R^d}c(\0, z) \md z   - n^{-1} \int_{\R^d} c(\0, z) \gamma_{W_n}(z)\md z.
	\ee
Now we claim that
$$ \lim_{n \to \infty}  n^{-1} \int_{\R^d} c(\0, z) \gamma_{W_n}(z)\md z = 0. $$
	Indeed, as noted in  Lemma 1  of  \cite{Martin80},  for all $z \in \R^d$ we have  $\lim_{n \to \infty}  n^{-1}  \gamma_{W_n}(z) = 0.$
	Since  $n^{-1}c(\0, z) \gamma_{W_n}(z)$ is dominated by the fast-decreasing function $c(\0, z)$, the dominated convergence theorem gives the	claimed limit.  Letting $n \to \infty$ in \eqref{doub2} gives
	\be \label{doub3}
	\lim_{n \to \infty} n^{-1} \Var \H_n^\xi(\P) = \E_{\0} \xi^2(\0, \P)\rho^{(1)}(\0)  +  \int_{\R^d}c(\0, z) \md z = \sigma^2(\xi),
	\ee
where the last equality follows by the definition of $\sigma^2(\xi)$ in \eqref{eqn:sigdef} and the finiteness follows by the fast-decreasing property of $c(\0, z, \P)$ (which follows from the assumption of fast decay of the second mixed moment density).\\
Now if $\sigma^2(\xi) = 0$  then the right hand side of \eqref{doub3} vanishes, i.e.,
$$ \E_{\0} \xi^2(\0, \P) \rho^{(1)}(\0) +  \int_{\R^d}c(\0, z) \md z = 0. $$
	Applying this identity to the right hand side of \eqref{doub2}, then multiplying \eqref{doub2}
	by $n^{1/d}$ and taking limits we obtain
	\be \label{doub4}
	\lim_{n \to \infty} n^{-(d-1)/d} \Var \H_n^\xi(\P) = - \lim_{n \to \infty} n^{-(d-1)/d}    \int_{\R^d} c(\0, z) \gamma_{W_n}(z)\md z .
	\ee
	As in \cite{Martin80},  we have  $n^{-(d-1)/d} \gamma_{W_n}(z) \leq C|z|$, and therefore again, by the fast-decreasing property of
	$c(\0, z)$  we conclude that  $n^{-(d-1)/d}c(\0, z) \gamma_{W_n}(z)$ is dominated by an integrable function of $z$.
	Also, as in  \cite[Lemma 1]{Martin80},  for all $z \in \R^d$ we have  $\lim_{n \to \infty}  n^{-(d-1)/d} \gamma_{W_n}(z) = \gamma(z).$
	The dominated convergence theorem yields \eqref{propeq}  as desired,
	
	$$
	\lim_{n \to \infty} n^{-(d-1)/d} \Var \H_n^\xi(\P) = - \int_{\R^d} c(\0, z) \gamma(z)\md z. \, \, \, \, \, \, \square $$

\subsection{\bf \bf First proof of the central limit theorem}
\label{sec:proof_clt1}

\subsubsection{The method of cumulants}
\label{sec:method_cumulants}
We use the method of cumulants to prove Theorem \ref{thm:main0}. We shall define cumulants precisely in Section \ref{sec:prop_cumulants}. Write  $ \overline{\mu}_{n}^\xi$ for the centered  measure $\mu_n^\xi - \E \mu_n^\xi$ and recall that we write $\langle f, \mu \rangle$ for $\int f d \mu.$
	The guiding principle is that as soon as the $k$th order cumulants $C_n^k$ for $(\Var \langle f, \mu_n^\xi \rangle)^{-1/2} \langle f,  \overline{\mu}_{n}^\xi \rangle $ vanish as $n \to \infty$ for $k$ large, then
	\be \label{1CLT}  (\Var \langle f, \mu_n^\xi \rangle)^{-1/2} \langle f,  \overline{\mu}_{n}
	\rangle \tod N.\ee
	We establish the vanishing of $C_n^k$ for $k$ large by showing that the fast decay of correlation functions for the $\xi$-weighted measures at \eqref{xwm} implies volume order growth (i.e., growth of order $O(n)$) for the
 $k$th order cumulant  for $ \langle f,  \overline{\mu}_{n}^\xi \rangle $,  $k \geq 2$,
	and then use the assumption $\Var \langle f, \mu_n^\xi \rangle = \Omega(n^{\nu})$.
	
	\paragraph{Our approach}  The $O(n)$ growth of the $k$th order cumulant for $ \langle f,  \overline{\mu}_{n}^\xi \rangle $  is established by controlling the growth of $k$th order cumulant measures for $\mu_n^\xi$, here denoted by $c_n^k$, and which are defined analogously to moment measures.
	We first prove a general result (see \eqref{cluster} and \eqref{cumbound} below) showing that integrals of the cumulant measures $c_n^k$ may be controlled by a finite sum of integrals of so-called $(S,T)$ semi-cluster measures, where $(S,T)$ is a generic partition of $\{1,...,k \}$.
	This result holds for any $\mu_n^\xi$ of the form \eqref{eqn:rand_meas} and depends neither on choice of input $\P$ nor on the localization properties of
	$\xi$. Semi-cluster measures for $\mu_n^\xi$ have the appealing property that they involve differences of measures on product spaces with product measures,
	and thus their Radon-Nikodym derivatives involve differences of correlation functions of the $\xi$-weighted measures.
	
	In general, bounds on cumulant measures  in terms of semi-cluster measures are not terribly informative.
	However, when $\xi$, together with $\P$, satisfy moment bounds and
	fast decay of correlations~\eqref{e.clustering-generalized-mixed-monents}, then the situation changes.  First,
	integrals of $(S,T)$ semi-cluster measures on properly chosen subsets $W(S,T)$ of
	$W_n^k$, with $(S,T)$ ranging over partitions of $\{1,...,k \}$, exhibit $O(n)$ growth.
	This is because the subsets $W(S,T)$ are chosen so that the Radon-Nikodym derivative of the $(S,T)$ semi-cluster measure, being a difference of
	the correlation functions of the $\xi$-weighted measures, may be controlled by~\eqref{e.clustering-generalized-mixed-monents} for points $(v_1,...,v_k) \in W(S,T)$.
	Second, it conveniently happens that $W_n^k$ is precisely the union of
	$W(S,T)$, as $(S,T)$ ranges over partitions of $\{1,...,k \}$.  Therefore, combining these observations, we see that every cumulant measure on $W_n^k$ is a sum ranging over partitions $(S,T)$ of $\{1,...,k \}$ of linear combinations of $(S,T)$ semi-cluster measures on $W(S,T)$, each of which exhibits $O(n)$ growth.
	
	Thus cumulant measures $c_n^k$ exhibit growth {\em proportional to $\Vol_d(W_n)$} carrying $\P_n$, namely
	\be \label{cumbounds}
	\langle f^k,c^k_{n} \rangle = O(n), \ \ f \in {\cal B}(W_1), \ \ k = 2, 3,...
	\ee
	The remainder of Section \ref{sec:proof_clt1} provides the details justifying  \eqref{cumbounds}.
		
	\paragraph{Remarks on related work}
	\label{para:rel_work}
	(a) The estimate \eqref{cumbounds} first appeared in \cite[Lemma 5.3]{Baryshnikov05}, but the work of \cite{ERS} (and to some extent \cite{Yukich12}) was the first to rigorously
	control the growth of $c^k_{n}$  on the diagonal subspaces, where two or more coordinates coincide.  In fact Section 3 of \cite{ERS} shows the estimate $\langle f^k,c^k_{n} \rangle \leq L^k (k!)^{\beta} n,$
where $L$ and $\beta$ are constants independent of $n$ and $k$.
We assert that the  arguments behind \eqref{cumbounds} are not restricted to Poisson input, but depend only on
the fast decay of correlations ~\eqref{e.clustering-generalized-mixed-monents} of the $\xi$-weighted measures   and moment bounds \eqref{eqn:mom}.   Since these arguments are not well known we present them in a way which is hopefully  accessible and reasonably self-contained.  Since we do not care about the constants in \eqref{cumbounds},
	we shall suitably adopt the arguments of  \cite[Lemma 5.3]{Baryshnikov05} and \cite{Yukich12},  taking the  opportunity to make those arguments more rigorous. Indeed those arguments did not adequately explain the fast decay of the correlations of the $\xi$-wighted measures of the $\xi$-weighted measures on diagonal subspaces.
	
\noindent(b) The breakthrough paper \cite{Nazarov12} shows that the $k$th order cumulant for the {\em linear statistic} \\ $ ( \Var \langle f,  \sum_{x} \delta_{n^{-1/d} x} \rangle )^{-1/2} \langle f, \sum_{x} \delta_{n^{-1/d} x} \rangle $
	vanishes as $n \to \infty$ and $k$ large. This approach is extended to $\langle f, \mu_n^{\xi} \rangle$ in  Section \ref{sec:clt_proof2} thereby giving a second proof of the central limit theorem.
		
\subsubsection{Properties of  cumulant and semi-cluster measures}
\label{sec:prop_cumulants}
\paragraph{Moments and cumulants}
	For a random variable $Y$ with all finite moments, expanding the logarithm of
	the Laplace transform (in the negative domain) in a formal power series  gives
	\begin{equation}\label{e.cumulants-series}
	\log\E(e^{tY})=\log\bigl(1+\sum_{k=1}^\infty\frac{M_kt^k}{k!}\bigr)=\sum_{k=1}^\infty\frac{S_kt^k}{k!}\,,
\end{equation}
where $M_k=\E(Y^k)$ is the $k\,$th moment of $Y$ and $S_k=S_k(Y)$ denotes the $k$\,th cumulant of $Y$. Both series in~\eqref{e.cumulants-series} can be  considered as formal ones and no additional condition (on exponential moments of $Y$) are required for the cumulants to exist. Explicit relations between cumulants and moments  may be established by formal manipulations of these series, see
	e.g.~\cite[Lemma 5.2.VI]{DVJ}. In particular
	\begin{equation}\label{e.cumulants-moments}
	S_k = \sum_{\gamma\in\Pi[k]} (-1)^{|\gamma|-1}
	(|\gamma|-1)!\prod_{i=1}^{|\gamma|} M^{|\gamma(i)|}\,,
	\end{equation}
where $\Pi[k]$ is the set of all unordered partitions of the set
	$\{1,...,k\}$, and for a partition $\gamma=\{\gamma(1),\ldots,\allowbreak\gamma(l)\}\in \Pi[k]$,
	$|\gamma|=l$ denotes the number of its elements, while $|\gamma(i)|$ the number of elements of
	subset $\gamma(i)$. (Although elements of $\Pi[k]$ are unordered
	partitions, we need to adopt some convention for the labeling of
	their elements: let  $\gamma(1), \ldots,\gamma(l)$ correspond to the ordering of the smallest elements in
	the partition sets.) In view of~\eqref{e.cumulants-moments}  the existence of the $k$th cumulant
	$S_k$ follows from the finiteness of the  moment $M_k$.
	
\paragraph{Moment measures}
	Given a random  measure  $\mu$ on $\R^d$, the $k$-th moment measure
	$M^k = M^k(\mu)$ is the one (Sect 5.4 and Sect 9.5 of \cite{DVJ}) satisfying
	$$
	\langle f_1 \otimes... \otimes f_k, M^k(\mu) \rangle = \E [ \langle f_1, \mu \rangle ...  \langle f_k, \mu \rangle ] = \E [ \sum_{x \in \P_n} f_1( { x \over n^{1/d} }) \xi(x, \P_n) \cdot \cdot \cdot \sum_{x \in \P_n} f_k( { x \over n^{1/d} }) \xi(x, \P_n)]
	$$
	for all $f_1,...,f_k \in \B(\R^d)$,  where $f_1 \otimes... \otimes f_k: \ (\R^d)^k \to \R$
	is given by $ f_1 \otimes... \otimes f_k(x_1,...,x_k) = f_1(x_1) ... f_k(x_k).$
	
	As  on p. 143 of \cite{DVJ}, when $\mu$ is a counting measure,
	$M^k$ may be expressed as a sum of factorial moment measures $M_{[j]}, 1 \leq j \leq k,$
	(as defined on p. 133 of \cite{DVJ}):
$$
	M^k(\md (x_1 \times\dots \times x_k)) = \sum_{j = 1}^k \sum_{\cal V} M_{[j]}( \Pi_{i = 1}^j \md y_i( \cal V)) \delta(\cal V),
$$
	where, to quote from \cite{DVJ}, the inner sum is taken over all partitions $\cal V$ of the $k$ coordinates into $j$ non empty disjoint subsets, the
	$y_i( {\cal V}), 1 \leq i \leq j,$ constitute an arbitrary selection of one coordinate from each subset, and $\delta(\cal V)$ is a $\delta$ function
	which equals zero unless equality holds among the coordinates in each non-empty subset of $\cal V$.
	
	When $\mu$ is  the atomic measure $\mu_n^\xi$, we write $M_n^k$ for $M^k(\mu_n^\xi)$.
	By the Campbell formula,  considering repetitions in the $k$-fold product of $\R^d$, and putting
	$\tilde{y}_i := y_i({\cal V})$ and ${\cal V}:= ( {\cal V}_1, ... ,{\cal V}_j)$
	we have that
\begin{align*}
& \langle f \otimes... \otimes f, M_n^k \rangle  = \E [ \langle f, \mu_n^\xi \rangle ...  \langle f, \mu_n^\xi \rangle ] \\
&= \sum_{j = 1}^k \sum_{\cal V} \int_{ (W_n)^j }  \Pi_{i = 1}^k f( {y_i \over n^{1/d}} ) \E_{\tilde{y}_1....\tilde{y}_j} [ \Pi_{i = 1}^j  \xi^{ | {\cal V}_i | } (\tilde{y}_i, \P_n)] \rho^{(j)}(\tilde{y}_1,...,\tilde{y}_j)  \Pi_{i = 1}^j
	\md y_i( \cal V) \delta( \cal V).
\end{align*}
	In other words, recalling Lemma 9.5.IV of \cite{DVJ} we get
	\be \label{Mk}
	dM_n^k(y_1,...,y_k) = \sum_{j = 1}^k \sum_{\cal V}m^{( |{\cal V}_1|,..., |{\cal V}_j| )}   (\tilde{y}_1,...,\tilde{y}_j;n) \Pi_{i = 1}^j
	\md y_i( {\cal V}) \delta( {\cal V}).
	\ee
	\paragraph{Cumulant measures}
	
The $k$th cumulant measure $c_n^k:= c^k(\mu_n)$ is defined analogously to the $k$th moment measure via
	$$
	\langle f_1 \otimes... \otimes f_k, c^k(\mu_n) \rangle = c( \langle f_1, \mu_n \rangle ...  \langle f_k, \mu_n \rangle )
	$$
	where $c(X_1,...,X_k)$ denotes the joint cumulant of the random variables $X_1,...,X_k$.
		
	The existence of the cumulant measures
	$c_{n}^l, \ l = 1,2,...$ follows from the existence of moment measures
	in view of the representation \eqref{e.cumulants-moments}. Thus, we have the following representation for cumulant measures :
	$$
	c_{n}^l = \sum_{T_1,...,T_p} (-1)^{p-1} (p-1)! M_n^{T_1} \cdots
	M_n^{T_p},
	$$
	where $T_1,...,T_p$ ranges over	all unordered partitions of the set $1,...,l$ (see p. 30 of \cite{MM}).
	Henceforth for $T_i \subset \{1,...,l \}$, let  $M_{n}^{T_i}$ denote a copy of the moment measure $M^{|T_i|}$
	on the product space $W^{T_i}$.  Multiplication denotes the usual product of measures: For $T_1, T_2$ disjoint sets of integers and for
	measurable $B_1 \subset (\R^d)^{T_1}, B_2 \subset (\R^d)^{T_2}$ we have $M_n^{T_1} M_n^{T_2}(B_1 \times B_2) = M_n^{T_1}(B_1) M_n^{T_2}(B_2)$. The first cumulant measure coincides with the expectation measure and the second cumulant measure coincides with the covariance measure.
	
\paragraph{Cluster and semi-cluster measures}
We show that every cumulant measure $c_n^k$  is a linear combination of products of  moment and cluster measures. We first  recall the definition of cluster and semi-cluster measures. A cluster measure $U_n^{S,T}$ on $W_n^S \times W_n^T$ for non-empty $S, T \subset \{1,2,... \}$ is defined by
$$
	U_n^{S,T}(B \times D) = M_n^{S \cup T} (B \times D) - M_n^S(B)
	M_n^T (D)
	$$
for  Borel sets $B$ and $D$ in $W_n^S$ and $W_n^T$, respectively, and where multiplication means product measure.
	
	Let $S_1, S_2$ be a partition of $S$ and let $T_1, T_2$ be a
	partition of $T$.  A product of a cluster measure
	$U_n^{S_1,T_1}$ on $W_n^{S_1} \times W_n^{T_1}$ with products of
	moment measures $M_n^{|S_2|}$ and $M_n^{|T_2|}$ on $W_n^{S_2} \times
	W_n^{T_2}$ is an $(S,T)$ {\em semi-cluster measure}.
	
	For each non-trivial partition $(S,T)$ of $\{1,...,k\}$ the $k$-th
	cumulant $c_n^k$ measure is represented as
	\be \label{cluster}c_n^k =
	\sum_{(S_1,T_1),(S_2,T_2) } \alpha((S_1,T_1),(S_2,T_2) ) U_n^{S_1,T_1}
	M_n^{|S_2|} M_n^{|T_2|}, \ee
	where the sum ranges over  partitions  of
	$\{1,...,k\}$ consisting of pairings $(S_1,T_1)$, $(S_2,T_2)$, where
	$S_1, S_2 \subset S$ and $T_1, T_2 \subset T$,
	where {\em $S_1$ and $T_1$ are non-empty}, and where \\
	$\alpha((S_1,T_1),(S_2,T_2) )$ are integer valued pre-factors. In
	other words, for any non-trivial partition $(S,T)$ of $\{1,...,k\}$,
	$c_n^k$ is a linear combination of $(S,T)$ semi-cluster measures. We prove this exactly as in the proof of
	Lemma 5.1 of \cite{Baryshnikov05}, as that proof involves only combinatorics and does not depend on the nature of the input.
	For an alternate proof, with good growth bounds on the integer
	pre-factors $\alpha((S_1,T_1),(S_2,T_2) )$, we refer to Lemma 3.2 of \cite{ERS}.
	
Let $\Xi(k) $ be the collection of partitions of $\{1,...,k\}$ into two subsets $S$ and $T$.  Whenever $W_n^k$ may be expressed as the union of  sets
$W(S,T), \ (S,T) \in \Xi(k)$, then we may write
\begin{align}
	& | \langle f^k, c_n^k \rangle | \leq \sum_{(S,T) \in \Xi(k) } \int_{W(S,T)} |f(v_1)...f(v_k) | |dc_n^k(v_1,...,v_k) | \label{cumbound} \\
& \leq ||f||^k_{\infty} \sum_{(S,T) \in \Xi(k) }  \sum_{(S_1,T_1),(S_2,T_2) }  |  \alpha((S_1,T_1),(S_2,T_2) )| \int_{W(S,T)} d( U_n^{S_1,T_1}
	M_n^{|S_2|} M_n^{|T_2|}) (v_1,...,v_k), \no
\end{align}
where the last inequality follows by \eqref{cluster}.
	As noted at the outset, this bound is valid for any $f \in \B(\R^d)$ and any measure  $\mu_n^\xi$ of the form \eqref{eqn:rand_meas}.
	
	We now specify the collection of sets $W(S,T), \ (S,T) \in \Xi(k)$, to be  used in \eqref{cumbound} as well as in all that follows.
	Given $v:= (v_1,...,v_k) \in W_n^k$, let $$D_k(v):= D_k(v_1,...,v_k):=
	\max_{i \leq k} (|v_1 - v_i| + ... + |v_k - v_i|)$$
	be the $l^1$ diameter for $v$.  For all such
	partitions consider the subset $W(S,T)$ of  $W_n^{S} \times
	W_n^{T}$ having the property that $v \in  W(S,T)$ implies $d(v^S, v^T) \geq D_k(v)/k^2,$ where $v^S$ and $v^T$ are the projections
	of $v$ onto $W_n^S$ and $W_n^T$, respectively, and where $d(v^S, v^T)$
	is the minimal Euclidean distance between pairs of points from
	$v^S$ and $v^T$.
	
It is easy to see that for every $v := (v_1,...,v_k) \in W_n^k$, there is a partition $(S,T)$ of $\{1,...,k\}$ such that $d(v^S, v^T) \geq D_k(v)/k^2.$ If this were not the case then given $v:= (v_1,...,v_k)$, the distance between any two components of $v$ must be strictly less than $D_k(v)/k^2$ and we would get  $\max_{i \leq k} \sum_{j = 1}^k |v_i - v_j| \leq (k -1)k D_k/k^2 < D_k$, a contradiction. Thus $W_n^k$ is the union of sets $W(S,T), \ (S,T) \in \Xi(k)$, as asserted.	We next describe the behavior of the differential $d( U_n^{S_1,T_1}M_n^{|S_2|} M_n^{|T_2|})$ on $W(S,T)$.
	
\paragraph{Semi-cluster measures on $W(S,T)$}
	
	Next, given $S_1 \subset S$ and $T_1 \subset T$, notice that $d(v^{S_1}, v^{T_1}) \geq d(v^S, v^T)$ where $v^{S_1}$ denotes the projection of $v^S$ onto $W_n^{S_1}$ and
	$v^{T_1}$ denotes the projection of $v^T$ onto $W_n^{T_1}$.
	Let $\Pi(S_1, T_1)$ be the partitions of $S_1$ into $j_1$ sets ${\cal V}_1,...,{\cal V}_{j_1}$, with $1 \leq j_1 \leq |S_1|$,
	and the partitions of $T_1$ into $j_2$ sets ${\cal V}_{j_1 + 1},...,{\cal V}_{j_1 + j_2}$, with $1 \leq j_2 \leq |T_1|$.
	Thus an element of $\Pi(S_1, T_1)$ is a partition of $S_1 \cup T_1$.
	
	If a partition $\cal V$ of $S_1 \cup T_1$ does not belong to $\Pi(S_1, T_1)$, then there is a partition element of
	$\cal V$ containing points in $S_1$ and $T_1$ and thus, recalling \eqref{Mk}, we have $\delta({\cal V}) = 0 $ on the set $W(S,T)$.  Thus we make the	crucial observation that, {\em on the set $W(S,T)$  the differential $d(M_n^{S_1 \cup T_1})$ collapses into a sum over partitions in
		$\Pi(S_1, T_1)$.}
	Thus $d( M_n^{S_1 \cup T_1})$ and $d (M_n^{S_1}M_n^{T_1})$ both involve sums of measures on common diagonal subspaces, as does their difference,
	made precise as follows.
	
	\begin{lemma} \label{wst}  On the set $W(S,T)$ we have
		\be \label{wsta}
		d(U_n^{S_1,T_1}) = \sum_{j_1 = 1}^{|S_1|} \sum_{j_2 = 1}^{|T_1|} \sum_{ {\cal V} \in \Pi(S_1, T_1)}  [ .... ] \Pi_{i = 1}^{j_1 + j_2} \md y_i({\cal V})
		\delta ({\cal V})
		\ee
		where
\begin{align*}
		[....] &:= m^{( |{\cal V}_1|,...,|{\cal V}_{j_1}|, |{\cal V}_{j_1 + 1}|,...,|{\cal V}_{j_1 + j_2}| )}(\tilde{y}_1,....,\tilde{y}_{j_1}\tilde{y}_{j_1 + 1},
		...\tilde{y}_{j_1 + j_2};n) \\
	& \, \,	 \,- m^{( |{\cal V}_1|,...,|{\cal V}_{j_1}| )} (\tilde{y}_1,....,\tilde{y}_{j_1};n ) m^{( |{\cal V}_{j_1 + 1}|,...,|{\cal V}_{j_1 + j_2}| )}(\tilde{y}_{j_1 + 1},
		...\tilde{y}_{j_1 + j_2};n).
\end{align*}
\end{lemma}
	The representations of  $dM_n^{|S_2|}$   and  $dM_n^{|T_2|}$ follow from \eqref{Mk}, that is to say
	\be \label{wstb}
	d M_n^{|S_2|} =  \sum_{j_3 = 1}^{|S_2|} \sum_{{\cal V} \in \Pi(S_2) }  m^{( |{\cal V}_1|,..., |{\cal V}_{j_3}| )}   (\tilde{y}_1,...,\tilde{y}_{j_3};n) \Pi_{i = 1}^{j_3}
	\md y_i( {\cal V}) \delta( {\cal V}),
	\ee
	where  $\Pi(S_2)$ runs over partitions of $S_2$ into $j_3$ sets, $1 \leq j_3 \leq |S_2|$. Similarly
	\be \label{wstc}
	d M_n^{|T_2|} =  \sum_{j_4 = 1}^{|T_2|} \sum_{{\cal V} \in \Pi(T_2) }  m^{( |{\cal V}_1|,..., |{\cal V}_{j_4}| )}   (\tilde{y}_1,...,\tilde{y}_{j_4};n) \Pi_{i = 1}^{j_4}
	\md y_i( {\cal V}) \delta( {\cal V}),
	\ee
	where  $\Pi(T_2)$ runs over partitions of $T_2$ into $j_4$ sets,  $1 \leq j_4 \leq |T_2|$.
	
	\subsubsection{Fast decay of correlations and semi-cluster measures}
	\label{sec:semi_cluster}
	
	The previous section established properties of semi-cluster and cumulant measures valid for any $\mu_n^\xi$ of the form \eqref{eqn:rand_meas}.
	If  $\xi$ with $\P$ exhibit fast decay of correlations~\eqref{e.clustering-generalized-mixed-monents} of the $\xi$-weighted measures
  and satisfies moment bounds, we now assert that each integral in \eqref{cumbound}  is  $O(n)$.
	
\begin{lemma} \label{UMN}  Assume $\xi$ satisfies  moment bounds \eqref{eqn:mom} for all $p \geq 1$ and
exhibits fast decay of correlations~\eqref{e.clustering-generalized-mixed-monents} in its
$\xi$-weighted measure.
		 For each partition element $(S,T)$ of $\Xi(k)$ we have
		\be \label{umnintegral}
		\int_{W(S,T) \subset W_n^S \times W_n^T}   | d(U_n^{S_1,T_1} M_n^{|S_2|}  M_n^{|T_2|} )| = O(n).
		\ee
	\end{lemma}
	
	\begin{proof} The differential  $d(U_n^{S_1,T_1}M_n^{|S_2|}  M_n^{|T_2|}   )$ is a  sum
		$$\sum_{j_1 = 1}^{|S_1|} \sum_{j_2 = 1}^{|T_1|}  \sum_{j_3 = 1}^{|S_2|}\sum_{j_4 = 1}^{|T_2|} [...][...][...]$$
		of products of three factors,
		one factor coming from each of the summands in \eqref{wsta}-  \eqref{wstc}.
		By Theorem \ref{prop:clustering_gen}, on the set $W(S,T)$ the factor arising
		from \eqref{wsta} is bounded  in absolute value by
		$$
		   \tilde{C}_k \tilde{\phi}(  \frac{ \tilde{c}_k D_k(y) } {k^2 }).
		$$
		By the moment bound \eqref{eqn:mom} the two remaining factors arising from  summands in \eqref{wstb}-  \eqref{wstc}
		are bounded by a constant $M'(k)$ depending only on $k$.

Thus we have
\begin{align*}
 \int_{W(S,T)}    | d(U_n^{S_1,T_1}M_n^{|S_2|}  M_n^{|T_2|}  )|
		& \leq \tilde{C}_k  (M'(k))^2 \sum_{j = 1}^k \sum_{{\cal V}} \int_{W(S,T)}   \tilde{\phi}(  \frac{ \tilde{c}_k D_k(y) } {k^2 })
		\Pi_{i = 1}^{j} \md y_i( {\cal V}) \delta( {\cal V}) \\
		& \leq \tilde{C}_k  (M'(k))^2 \sum_{j = 1}^k \sum_{{\cal V}} \int_{(W_n)^j } \tilde{\phi}(  \frac{ \tilde{c}_k D_k(y) } {k^2 })
		\Pi_{i = 1}^{j} \md y_i( {\cal V}) \delta( {\cal V}).
		\end{align*}
		Here ${\cal V}$ runs over all partitions of the $k$ coordinates into $j$ non-empty disjoint subsets.
		We assert that all summands are $O(n)$.  We show this when $j = k$, as the proof for the remaining indices $j \in \{1,...,k-1\}$ is similar.
		Write
		\begin{align*}
		& \int_{y_1 \in W_n} ... \int_{y_k \in W_n} \tilde{\phi}(  \frac{ \tilde{c}_k D_k(y)} {k^2}  ) \md y_1 ... \md y_k \\
		& = \int_{y_1 \in W_n} \int_{w_2 \in W_n - y_1}... \int_{w_k \in W_n - y_1} \tilde{\phi}( \frac{ \tilde{c}_k D_k(\0, w_2,...,w_k)} {k^2} ) \md y_1 \md w_2... \md w_k.
		\end{align*}
		Now $D_k(\0, w_2,...,w_k) \geq \sum_{i = 2}^k |w_i|$.  Letting $e_k := \tilde{c}_k (k-1)/k^2$ gives
		\begin{align*}
		& \int_{y_1 \in W_n} ... \int_{y_k \in W_n} \tilde{\phi}(  \frac{ \tilde{c}_k D_k(y)} {k^2} ) \md y_1 ... \md y_k \leq n  \int_{w_2 \in \R^d }... \int_{w_k \in \R^d} \tilde{\phi}( \frac{ e_k} {k -1}   \sum_{i = 2}^k |w_i| ) \md w_2... \md w_k \\
       & \leq n  \int_{w_2 \in \R^d }... \int_{w_k \in \R^d} \tilde{\phi}(    \Pi_{i = 2}^k |w_i|^{1/(k-1)} ) \md w_2... \md w_k = O(n),
				\end{align*}
		where the first inequality follows from the decreasing behavior of $\tilde{\phi}$, the second inequality follows from the arithmetic geometric mean inequality, and the last equality follows since $\tilde{\phi}$ is decreasing faster than any polynomial. We similarly bound the other summands for $j \in \{1,...,k-1\}$, completing the proof of Lemma \ref{UMN}.
	\end{proof}
	
	\subsubsection{Proof of Theorem \ref{thm:main0}}  By the bound \eqref{cumbound} and Lemma \ref{UMN} we obtain
	\eqref{cumbounds}.
	Letting $C_n^k$  be  the $kth$ cumulant for
	$( \Var \langle f, \mu_n^\xi \rangle)^{-1/2} \langle f, \mu_n^\xi \rangle $, we obtain $C_n^1 = 0,  C_n^2 = 1$, and for all $k = 3,4,....$
	$$
	C_n^k = O(n (\Var \langle f, \mu_n^\xi \rangle)^{-k/2} ).
	$$
	
	Since $\Var \langle f, \mu_n^\xi \rangle = \Omega(n^{\nu} )$ by assumption, it follows that
	if $k \in (2/\nu, \infty)$, then  the $k$th cumulant tends  $C_n^k$ to zero as $n \to \infty$. By a classical result of Marcinkiewicz 
	(see e.g.   \cite[Lemma 3]{So}), we get that all cumulants
	$C_n^k, \ k \geq 3$, converge to zero as $n \to \infty$.  This gives \eqref{1CLT} as desired and completes the proof of Theorem \ref{thm:main0}.  \qed
	
	\vskip.5cm

	\subsection{\bf Second proof of the central limit theorem}
	\label{sec:clt_proof2}
	
We  now give a second proof of the central limit theorem which we believe is of independent interest. Even though this proof is also based on the cumulant method used in Section \ref{sec:method_cumulants}, we shall bound the cumulants using a different approach, using Ursell functions of the $\xi$-weighted measure and establishing a property equivalent to Brillinger mixing; see Remarks at the end of Section \ref{sec:ursell_bounds}.
 Though much of this proof can be read independently of the proof in Section \ref{sec:proof_clt1}, we repeatedly  use the definition of moments and cumulants from Section \ref{sec:prop_cumulants}.

\paragraph{Our approach} We shall adapt the approach in~\cite[Sec. 4]{Nazarov12}
	replacing  $\P_{GEF}$  by our  $\xi$-weighted measures, which are purely atomic measures.
As noted  in Section~\ref{sec:strong_clustering_mix_mom},
the correlation functions of the $\xi$-weighted measure are generalizations of the
 correlations functions of the simple point process,
but the extension  of the approach  used in~\cite[Sec. 4]{Nazarov12}
requires some care regarding the  repeated arguments captured by general exponents $k_i$ in~\eqref{eqn:mixedmomentn}.

	\subsubsection{Ursell functions of the $\xi$-weighted measures}
	\label{sec:ursell}
	Recall the definition of the correlation functions~\eqref{eqn:mixedmomentn}
of the $\xi$-weighted measures
	$$ m^{(k_1,\ldots,k_p)}(x_1,\ldots,x_p;n):= \E_{x_1,\ldots,x_p}\bigl((\xi(x_1,\P_n))^{k_1}\ldots
	(\xi(x_p,\P_n))^{k_p}\bigr) \rho^{(p)}(x_1,\ldots,x_p). $$
	We will drop dependence on $n$, i.e., $m^{(k_1,\ldots,k_p)}(x_1,\ldots,x_p;n)=m^{(k_1,\ldots,k_p)}(x_1,\ldots,x_p)$
	unless asymptotics in $n$ is considered.
	
Inspired by the approach in~\cite[Section~2]{Baumann85} we now introduce {\em Ursell functions}
   $m_\T^{(k_1,\ldots,k_p)}$ (sometimes called {\em truncated correlation function})
of the $\xi$-weighted measures.
Define  $m_\T^{(k_1,\ldots,k_p)}$
	by taking $m_\T^{(k)}(x):=m^{(k)}(x)$ for all $k \in  \N$ and  inductively
\begin{equation}\label{e.Ursel-induction}
	m_\T^{(k_1,\ldots,k_p)}(x_1,\ldots,x_p):=
	m^{(k_1,\ldots,k_p)}(x_1,\ldots,x_p)-
	\sum_{\gamma\in\Pi[p]\atop |\gamma|>1}\prod_{i=1}^{|\gamma|}
	m_\T^{(k_j:j\in \gamma (i))}(x_j:j\in\gamma(i))\,.
	\end{equation}
	for distinct $x_1,\ldots,x_p\in W_n$
	and all integers $k_1,\ldots,k_p$,  $p\ge 1$, and (implicitly)
	$n\le\infty$.
	It is straightforward to prove that these functions satisfy
	the following relations. They extend the  known relations for point processes, where $m^{(k_1,\ldots,k_p)}(x_1,\ldots,x_p)=\rho^{(p)}(x_1,\ldots,x_p)$ depend only on $p$, but we were unable to
	find them in the literature for (signed) purely atomic random measures, as our $\xi$-weighted measures.
	Assuming $1 \in \gamma(1)$ in ~\eqref{e.Ursel-induction} and summing over partitions of $\{1,\ldots,p\} \setminus \gamma(1)$, we get the following relation :
	\begin{align}
	m^{(k_1,\ldots,k_p)}(x_1,\ldots,x_p) & =  m_\T^{(k_1,\ldots,k_p)}(x_1,\ldots,x_p) + \sum_{I\subsetneqq\{1,\ldots,p\}\atop 1 \in I} m_\T^{(k_j:j\in I)}(x_j:j\in I)\,
m^{(k_j:j\in I^c)}(x_j:j\in I^c)\,, \label{e.Ursel-j0}
	\end{align}
	where $I^c:=\{1,\ldots,p\}\setminus I$.
	Using~\eqref{e.Ursel-j0}, by induction with respect to $p$,
	one obtains the direct relation to the correlation functions
	\begin{equation}\label{e.Ursel-direct}
	m_\T^{(k_1,\ldots,k_p)}(x_1,\ldots,x_p)
	=\sum_{\gamma\in\Pi[p]}(-1)^{|\gamma|-1}(|\gamma|-1)!
	\prod_{i=1}^{|\gamma|} m^{(k_j:j\in \gamma(i))}(x_j:j\in\gamma(i))\,.
	\end{equation}
This extends the relation~\cite[(27)]{Nazarov12}, valid for point processes. We say that a partition $\gamma = \{\gamma(1),\ldots,\gamma(l)\} \in \Pi(p)$ {\em refines} partition $\sigma = \{\sigma(1),\ldots,\sigma(l_1)\} \in \Pi(p)$ if for all $i \in \{1,\ldots,l\}$, $\gamma(i) \subset \sigma(j)$ for some $j \in \{1,\ldots,l_1\}$. Otherwise, the partition $\gamma$ is said to {\em mix} partition $\sigma$. Now using~\eqref{e.Ursel-induction}, we get for any $I\subsetneqq\{1,\ldots,p\}$
	\begin{equation}\label{e.Ursel-product}
	m^{(k_j:j\in I)}(x_j:j\in I)\,
	m^{(k_j:j\in I^c)}(x_j:j\in I^c)=
	\sum_{\gamma\in\Pi[p]\atop \gamma\,\text{refines}\,\{I,I^c\}}
	\prod_{i=1}^{|\gamma|} m_\T^{(k_j:j\in \gamma(i))}(x_j:j\in\gamma(i))\,,
	\end{equation}
	and therefore, again in view of~\eqref{e.Ursel-induction}
	\begin{align}
	m_\T^{(k_1,\ldots,k_p)}(x_1,\ldots,x_p)&=
	m^{(k_1,\ldots,k_p)}(x_1,\ldots,x_p)-
	m^{(k_j:j\in I)}(x_j:j\in I)\,
	m^{(k_j:j\in I^c)}(x_j:j\in I^c) \no \\
	&+\sum_{\gamma\in\Pi[p], |\gamma|>1 \atop
		\gamma\,\text{mixes}\,\{I,I^c\}}
	\prod_{i=1}^{|\gamma|} m_\T^{(k_j:j\in
		\gamma(i))}(x_j:j\in\gamma(i))\,\label{e.Ursel-cluster}.
	\end{align}
This extends the relation \cite[last displayed formula in the proof of Claim 4.1]{Nazarov12} valid for point processes.
	
	\subsubsection{Fast decay of correlations and bounds for Ursell functions}
	\label{sec:ursell_bounds}
	We show now that fast decay of correlations~\eqref{e.clustering-generalized-mixed-monents}  of the $\xi$-weighted measures  implies
	some bounds on the Ursell functions of these measures. Since $m^{(k_1,\ldots,k_p)}(x_1,\ldots,x_p;n)$ is invariant with respect to any joint permutation of its arguments $(k_1,\ldots,k_p)$ and $(x_1,\ldots,x_p)$, fast decay of correlations~\eqref{e.clustering-generalized-mixed-monents} of the $\xi$-weighted measures may be rephrased as follows : 	There exists a fast-decreasing function $\tilde\phi$ and constants
	$\tilde C_k$, $\tilde c_k$, such that
	for any collection of  positive integers $k_1,\ldots,k_p$, $p\ge2$,
	satisfying $k_1+\ldots+k_p=k$, for any nonempty, proper subset
	$I\subsetneqq \{1,\ldots,p\}$,
	for all  $n\le\infty$ and all configurations  $x_1,\ldots,x_p\in W_n$ of distinct points
	we have
	\begin{equation}\label{e.clustering-generalized-mixed-monents1}
	\Bigl|m^{(k_1,\ldots,k_p)}(x_1,\ldots,x_p;n)-
	m^{(k_j: j\in I)}(x_j:j\in I; n)\, m^{(k_j: j\in I^c)}(x_j:j\in
	I^c;n)\Bigr|\le
	\tilde C_{k}\tilde\phi(\tilde c_{k} s)\,,
	\end{equation}
	where $s:=d\bigl(\{x_j:j\in I\}, \{x_j:j\in I^c\}\bigr)$.

Now we consider the bounds of Ursell functions  of the $\xi$-weighted measures.
Following the idea of~\cite[Claim~4.1]{Nazarov12} one
proves that fast decay of correlations~\eqref{e.clustering-generalized-mixed-monents} of the $\xi$-weighted measures	and the $p$-moment condition \eqref{eqn:mom} imply
	that there exists a fast-decreasing function $\tilde\phi_\T$ and constants
	$\tilde C^\T_k$, $\tilde c^\T_k$, such that
	for any collection of  positive integers $k_1,\ldots,k_p$, $p\ge2$,
	satisfying $k_1+\ldots+k_p=k$,
	for all  $n\le\infty$ and all configurations  $x_1,\ldots,x_p\in W_n$ of distinct points
	we have
	\begin{equation}\label{e.Usell-clustering-bound}
	|m_{\T}^{(k_1,\ldots,k_p)}(x_1,\ldots,x_p;n)|\le
	\tilde
	C^\T_k\tilde\phi_\T\bigl(c^\T_{k}\mathrm{diam}(x_1,\ldots,x_p)\bigr)
	\,,
	\end{equation}
where $\mathrm{diam}(x_1,\ldots,x_p):=\max_{i,j=1\ldots p}(|x_i-x_j|)$.
	The proof uses the representation~\eqref{e.Ursel-cluster}, fast decay of correlations~\eqref{e.clustering-generalized-mixed-monents} of the $\xi$-weighted measures, together with the fact that there exist constants  $c^\T_p$ (depending on the dimension $d$) such that for each configuration $x_1,\ldots,x_p\in W_n$ , there exists a partition
	$\{I,I^c\}$ of $\{1,\ldots,p\}$ such that $d(\{x_j:j\in I\}, \{x_j:j\in I^c\})\ge \tilde c^\T_p\mathrm{diam}(x_1,\ldots,x_p)$.
	
	Next, inequality ~\eqref{e.Usell-clustering-bound} allows one to bound integrals
	\begin{equation}\label{e.Ursel-integral-bound}
	\sup_{n\le\infty}\sup_{x_1\in W_n}\sup_{k_1+\ldots+k_p=k\atop k_i>0}
	\int_{(W_n)^{p-1}}|m_\T^{(k_1,\ldots,k_p)}(x_1,\ldots,x_p;n)|\,\md x_2\cdots \md x_p< \infty.
	\end{equation}
Indeed, for a fixed point $x_1\in W_n$, we split $(W_n)^{p-1}$ into disjoint sets:
\begin{align*}G_0 &:= \{(x_2,\ldots,x_p)\in(W_n)^{p-1} : \text{diam}(x_1,\ldots,x_p)\le 1\}\\
G_l &:= \{(x_2,\ldots,x_p)\in(W_n)^{p-1} : 2^{l-1}<\text{diam}(x_1,\ldots,x_p)\le 2^l\},\quad l\ge 1
\end{align*}
and use estimate \eqref{e.Usell-clustering-bound} to bound the integral on the left-hand side of (~\ref{e.Ursel-integral-bound}) by
	$$\tilde C^\T_k+\tilde C^\T_k\sum_{l=1}^\infty 2^{dl(k-1)}\tilde \phi_\T(\tilde c_k^\T 2^{l-1})<\infty$$
	since $\tilde \phi_\T$ is fast-decreasing; cf. \cite[Claim~4.2]{Nazarov12}.

\paragraph{\bf Remarks}\ \\
(i)  A careful inspection of the relation~\eqref{e.Ursel-cluster} 
shows that in fact the  fast decay of correlations~\eqref{e.clustering-generalized-mixed-monents} of the $\xi$-weighted measures
{\em is equivalent} to the bound~\eqref{e.Usell-clustering-bound} on Ursell functions of these measures.\\
(ii) Condition~\eqref{e.Ursel-integral-bound}, implied by~\eqref{e.Usell-clustering-bound}, can be interpreted as the Brillinger mixing condition of the $\xi$-weighted measures.
In fact it is slightly stronger in the sense that the bound on the  Ursell functions
integrated over $\md x_2\cdots \md x_p$ in the entire space (corresponding to the total reduced cumulant measures)
is uniform for the whole family of the  $\xi$-weighted measuress considered on $W_n$, parametrized by $n\le\infty$ and,
for $n<\infty$ the bound is also uniform over $x_1\in W_n$ (which is immediate for reduced cumulant measures in the stationary case $n=\infty$).

\subsubsection{Proof of Theorem \ref{thm:main0}}
	
	The cumulant of order one is equal to the expectation and hence
	disappears for the considered (centered) random variable $\overline\mu_n^\xi(f)$.
	The cumulant
	of order $2$ is equal to the variance and hence equal to $1$ in our case.
	For $k\ge 2$, note the following relation between the normalized and the unnormalized cumulants :
	\begin{equation}
	\label{e.cum_reln}
	S_k((\Var\ \overline\mu_n^\xi(f))^{-1/2} \mu_n^\xi(f)) = (\Var\  \mu_n^\xi(f))^{-k/2} \times S_k(\mu_n^\xi(f)).
	\end{equation}
	
We establish the vanishing of~\eqref{e.cum_reln} for $k$ large by showing that the $k$th order cumulant $S_k(\mu_n^\xi(f)$ is of order $O(n)$, $k \geq 2$, and then use assumption \eqref{vlb0}, i.e., $\Var \langle f, \mu_n^\xi \rangle = \Omega(n^{\nu})$. We have
	\begin{align*}
	M_n^k&:=\E(\langle f, \mu_n^\xi \rangle)^k
	=\E\Bigl(\sum_{x_i\in\P_n} f_n(x_i)\xi(x_i,\P_n) \Bigr)^k,
	\end{align*}
        where $f_n(\cdot) = f(\cdot/n^{1/d})$. Considering appropriately the repetitions of points $x_i$ in the $k\,$th  product of the sum and using  the Campbell theorem at \eqref{disint}, one obtains
      \be \label{momentformula}
	M_n^k=\sum_{\sigma\in\Pi[k]}\langle \bigotimes_{i=1}^{|\sigma|}f_n^{|\sigma(i)|}m^{(\sigma)},\lambda_n^{|\sigma|}\rangle\,,\ee
where $\lambda_n^l$ denotes the Lebesgue measure on $(W_n)^l$ and $\bigotimes$ denotes the tensor product of functions
$$(\bigotimes_{i=1}^pf_n^{k_j})(x_1,\ldots,x_{p}) = \prod_{i=1}^p(f_n)^{k_j}(x_j)\,, \, m^{(\sigma)}(x_1,\ldots,x_{|\sigma|};n):=m^{(|\sigma(1)|,\ldots,|\sigma(|\sigma|)|)}(x_1,\ldots,x_{|\sigma|};n)\,$$
Using the above representation and~\eqref{e.cumulants-moments} the $k$th cumulant $S_k(\mu^\xi_n(f))$ can be expressed as follows
        \begin{align}
	S_k(\mu^\xi_n(f))& = \sum_{\gamma\in\Pi[k]} (-1)^{|\gamma|-1}
	(|\gamma|-1)!\sum_{\sigma\in\Pi[k]\atop \sigma\, \text{refines}\,\gamma}\prod_{i=1}^{|\gamma|}
	\langle \bigotimes_{j=1}^{|\gamma(i)/\sigma|} f_n^{(\gamma(i)/\sigma)(j)}m^{(\gamma(i)/\sigma)},\lambda_n^{|\gamma(i)/\sigma|}\rangle  \no \\
        &=\sum_{\sigma\in\Pi[k]} \sum_{\gamma \in\Pi[k]\atop \sigma\, \text{refines}\,\gamma}
	(-1)^{|\gamma|-1}(|\gamma|-1)!
	\prod_{i=1}^{|\gamma|}
        \langle  \bigotimes_{j=1}^{|\gamma(i)/\sigma|} f_n^{(\gamma(i)/\sigma)(j)}m^{(\gamma(i)/\sigma)},\lambda_n^{|\gamma(i)/\sigma|}\rangle\,,\label{e.S_k-Mk}
	\end{align}
where $\gamma(i)/\sigma$ is the partition of $\gamma(i)$ induced by $\sigma$. Note that for any partition $\sigma\in\Pi[k]$, with $|\sigma(j)|=k_j$, $j=1,\ldots, |\sigma|=p$, the inner sum in~\eqref{e.S_k-Mk} can be rewritten as follows:
\be
	\label{e.tensor}
	\sum_{\gamma \in\Pi[p]}
	(-1)^{|\gamma|-1}(|\gamma|-1)!
	\prod_{i=1}^{|\gamma|}
	\langle \bigotimes_{j\in\gamma(i)}f_n^{k_j}m^{(k_j:j\in\gamma(i))},\lambda_n^{|\gamma(i)|}\rangle=
	\langle \bigotimes_{j=1}^pf_n^{k_j}m_\T^{(k_1,\ldots,k_p)},\lambda_n^p\rangle\,,
	\ee
where the equality is due to~\eqref{e.Ursel-direct}. Consequently
\begin{equation}\label{e.cumulant-Ursell}
S_k(\mu^\xi_n(f))=\sum_{\sigma\in\Pi[k]} \langle \bigotimes_{j=1}^{|\sigma|}f_n^{|\sigma(j)|}m_\T^{(|\sigma(1)|,\ldots,|\sigma(|\sigma|)|)},\lambda_n^{|\sigma|}\rangle\,,
	\end{equation}
which extends the relation \cite[Claim 4.3]{Nazarov12} valid for point processes. The formula~\eqref{e.cumulant-Ursell}, which expresses the $k$th cumulant in terms of the Ursell functions, is the counterpart to the standard formula \eqref{momentformula} expressing $k$th moments in terms of correlation functions. Now, using \eqref{e.Ursel-integral-bound} and denoting the supremum therein by $\hat{C_k}$, we have that
	\begin{eqnarray*}
&  & |\langle \bigotimes_{j=1}^pf_n^{k_j}m_\T^{(k_1,\ldots,k_p)},\lambda_n^p\rangle\ | \leq \int_{W_n^p} |\bigotimes_{j=1}^pf_n^{k_j}| |m_\T^{(k_1,\ldots,k_p)}(x_1,\ldots,x_p)| \md x_1 \ldots \md x_p \\
		& \leq & \|f\|^k_{\infty} \int_{W_n} \md x_1 \int_{W_n^{p-1}} |m_\T^{(k_1,\ldots,k_p)}(x_1,\ldots,x_p)| \md x_2 \ldots \md x_p
		\, \, \leq  \|f\|^k_{\infty} \hat{C}_k\Vol_{ d} (W_n).
	\end{eqnarray*}

So, the above bound along with \eqref{e.S_k-Mk} and \eqref{e.tensor} gives us that $S_k(\mu^\xi_n(f)) = O(n)$ for all $k \geq 2$. Thus, using the variance lower bound condition \eqref{vlb0} and the relation \eqref{e.cum_reln}, we get for large enough $k$, that $S_k((\Var \mu_n^\xi(f))^{-1/2} \mu_n^\xi(f)) \to 0$ as $n \to \infty$. Now, as discussed in \eqref{1CLT}, this suffices to guarantee normal convergence.  \qed
		
\section{\bf Appendix}
\label{sec:appendix}

	
\subsection{\bf Facts needed in the proof of fast decay of correlations of the $\xi$-weighted measures}	
	
The following facts about U-statistics and the radius of stabilization are used in the proof		
of Lemma~\ref{l.FMEp}.

\begin{lemma}\label{l.U-stats}
Let $f,g$ be two real valued, symmetric functions defined on $(\R^d)^k$
and $(\R^d)^l$ respectively. Let $F:= \frac{1}{k!}\sum_{\x \in \X^{(k)}}f(\x)$ and $G:=\frac{1}{l!}\sum_{\x' \in
\X^{(l)}}g(\x')$ be the corresponding U-statistics of order $k$ and $l$
respectively, on the input $\X\subset\R^d$. Then we have: \\
\noindent(i) The product $F\,G$ is a sum of U-statistics of order not greater than $k+l$.\\
\noindent(ii) Let $\A$ be a fixed, finite subset of $\R^d$. The statistic $F_\A:=\frac{1}{k!}\sum_{\x \in (\X\cup\A)^{(k)}}f(\x)$  is a
sum of U-statistics of $\X$ of order not greater than $k$.
\end{lemma}

\noindent {\em Proof.} The two statements follow from symmetrizing the inner summands in the below representations
$$F\,G=\sum_{m=\max(k,l)}^{k+l} \sum_{\z\in\X^{(m)}}\frac{ f(z_1,\ldots,z_k)g(z_{m-l+1},\ldots,z_m)}{(k+l-m)!(m-k)!(m-l)!}\,,$$
%
%
%
%
$$F_\A=\sum_{m=0}^{\min(|\A|,k)}\sum_{\a\in\A^{(m)}} \sum_{\z\in\X^{(k-m)}}\frac{f(a_1,\ldots,a_m,z_1,\ldots,z_{k-m})}{m!(k-m)!}\,,$$
%
%
%
For a proof of the first representation note that by symmetry of $f$, $F(\Y) = f(y_1,\ldots,y_k)$ if $Y = \{y_1,\ldots,y_k\}$ and similarly for $G$. Thus, we derive that
\begin{eqnarray*}
F\,G & = & \sum_{\Y_i \subset \X, i=1,2}F(\Y_1)G(\Y_2) \1[|\Y_1| = k, |\Y_2| = l] \\
& = & \sum_{m = \max(k,l)}^{k+l} \sum_{\Y_i \subset \X, i=1,2}F(\Y_1)G(\Y_2) \1[|\Y_1 \cup \Y_2| = m,|\Y_1| = k, |\Y_2| = l] \\
&  = &  \sum_{m = \max(k,l)}^{k+l} \sum_{\Y \subset \X} \1[|\Y| = m] \sum_{\Y_i  \subset \Y, i=1,2}F(\Y_1)G(\Y_2) \1[\Y_1 \cup \Y_2 = \Y,|\Y_1| = k, |\Y_2| = l] \\
& = & \sum_{m = \max(k,l)}^{k+l} \sum_{\z \in \X^{(m)}} \frac{1}{m!}  \sum_{\Y_1,\Y_2}F(\Y_1)G(\Y_2) \1[\Y_1 \cup \Y_2 = \{z_1,\ldots,z_m\},|\Y_1| = k, |\Y_2| = l] \\
& = &  \sum_{m = \max(k,l)}^{k+l} \sum_{\z \in \X^{(m)}} \frac{f(z_1,\ldots,z_k)g(z_{m-l+1},\ldots,z_m)}{m!}  \sum_{\Y_1,\Y_2} \1[\Y_1 \cup \Y_2 = \{z_1,\ldots,z_m\},|\Y_1| = k, |\Y_2| = l] \\
& = &\sum_{m=\max(k,l)}^{k+l} \sum_{\z\in\X^{(m)}}\frac{ f(z_1,\ldots,z_k)g(z_{m-l+1},\ldots,z_m)}{(k+l-m)!(m-k)!(m-l)!}\,,
\end{eqnarray*}
thus proving the first representation above. The second representation follows similarly. \, \, \, $\square$

%
\begin{lemma}\label{l.Radius}
Let $\xi$ be a score function on  locally finite input $\X$ and $R^{\xi}:= R^\xi(x, \X)$ its radius of stabilization. Given $t>0$ consider the score function
		$\tilde\xi(x,\X):=\xi(x,\X)\1[R^{\xi}(x,\X)\le t]$. Then the radius of stabilization $R^{\tilde\xi}:= R^{\tilde\xi}(x, \X)$ of $\tilde\xi$ is bounded by $t$, i.e.,
		$R^{\tilde\xi}(x,\X)\le t$ for any  $x\in\X$.
	\end{lemma}
\begin{proof}
		Let  $\X,\A$ be  locally finite subsets of $\R^d$ with $x\in\X$. We have
		\begin{align*}
		&\tilde\xi(x,(\X\cap B_t(x))\cup (\cA\cap B_t^c(x)))\\
		&=\xi(x,(\X\cap B_t(x))\cup (\cA\cap B_t^c(x)))
		\,\1\bigl[R^\xi(x,(\X\cap B_t(x))\cup (\cA\cap B_t^c(x)))\le t\bigr]\\
		&=\xi(x,\X\cap B_t(x))
		\,\1\bigl[R^\xi(x,(\X\cap B_t(x))\cup (\cA\cap B_t^c(x)))\le t\bigr]\,,
		\end{align*}
where the last equality follows from the definition of $R^\xi$.
Notice
$$\1\bigl[R^\xi(x,(\X\cap B_t(x))\cup (\cA\cap B_t^c(x)))\le t\bigr]
		=\1\bigl[R^\xi(x,\X\cap B_t(x))\le t\bigr]$$
and so $\tilde\xi(x,(\X \cap B_t(x))\cup (\cA \cap B_t^c(x))) =
		\tilde\xi(x, \X \cap B_t(x)) $, which was to be shown.
	\end{proof}
\subsection{\bf Determinantal and permanental point process lemmas}	

The following facts illustrate the tractability of determinantal and permanental point processes and are of independent interest. If determinantal and  permanental point processes have a kernel $K$ decreasing fast enough, then they generate admissible point processes having fast decay of correlations as well as satisfying conditions \eqref{eqn:sum} and \eqref{eqn:clustering_condition} respectively. We are indebted to Manjunath Krishnapur for sketching  the proof of this result.	
	
	\begin{lemma}
		\label{lem:clustering_DPP}
		Let $\P$ be a stationary 
		determinantal point process on $\mR^d$ with a kernel satisfying
		$K(x,y) \leq \omega(|x-y|)$, where  $\omega$  is at \eqref{fastdk}. Then
		\be \label{dppestimate}
		| \rho^{(n)}(x_1,\ldots,x_{p+q}) - \rho^{(p)}(x_1,\ldots,x_{p})\rho^{(q)}(x_{p+1},\ldots,x_{p+q})|  \leq n^{1+\frac{n}{2}}\omega(s) \|K\|^{n-1},
		\ee
		where $\|K\| := \sup_{x,y \in \mR^d}|K(x,y)|$, $s$ is at \eqref{defs},  and $n = p + q$.
	\end{lemma}
		
	\noindent{\em Proof.}  Define the matrices $K_0 := (K(x_i,x_j))_{1 \leq i,j \leq n}, K_1 := (K(x_i,x_j))_{1 \leq i,j \leq p}$,
	and $K_2 := (K(x_i,x_j))_{p+1 \leq i,j \leq n}$.
	Let $L$ be the block diagonal matrix with blocks $K_1,K_2$. We define $\|K_0\| := \sup_{1 \leq i,j \leq n} |K_0(x_i,x_j)|$ and similarly for the other matrices. Then
	\begin{eqnarray}
| \rho^{(n)}(x_1,\ldots,x_{p+q}) & - & \rho^{(p)}(x_1,\ldots,x_{p})\rho^{(q)}(x_{p+1},\ldots,x_{p+q})| \no \\
	& = & | \det(K_0) - \det(K_1)\det(K_2)| \, = \,  |\det(K_0) - \det(L) | \no \\
		& \leq & n^{1+\frac{n}{2}}\|K_0-L\| \|K_0\|^{n-1} \, \leq \, n^{1+\frac{n}{2}}\omega(s) \|K\|^{n-1} \label{eqn:clustering_DPP}
	\end{eqnarray}
	where the inequality follows by \cite[(3.4.5)]{Anderson10}.  This gives \eqref{dppestimate}.  \qed
	
As a first step to prove the analogue of Lemma~\ref{lem:clustering_DPP} for permanental point processes, we prove an analogue of \eqref{eqn:clustering_DPP}.  We follow verbatim the proof of \eqref{eqn:clustering_DPP} as given in
	\cite[(3.4.5)]{Anderson10}.  Instead of using Hadamard's inequality for determinants as in \cite{Anderson10}, we use  the following version of Hadamard's inequality for permanents (\cite[Theorem 1.1]{CLL}): For any column vectors $v_1,\ldots,v_n$ of length $n$ with complex entries, it holds that

	$$ | {\rm{per}}([v_1,\ldots,v_n])| \leq \frac{n!}{n^{\frac{n}{2}}}\prod_{i=1}^n\sqrt{\bar{v_i}^Tv_i} \leq n! \prod_{i=1}^n\|v_i\|,  $$	
where $\|v_i\|$ is the $l_{\infty}$-norm of $v_i$ viewed as an $n$-dimensional complex vector.
\begin{lemma}
\label{lem:estimate_perm}
Let $n \in \N$. For any two matrices $K$ and $L$, we have
$$ |{\rm{per}}(K) - {\rm{per}}(L)| \leq n n! \|K-L\|\max\{\|K\|,\|L\|\}^{n-1}  .$$
\end{lemma}
 Now, in the proof of Lemma \ref{lem:clustering_DPP}, using the above estimate instead of \eqref{eqn:clustering_DPP}, we
establish fast decay of correlations \eqref{eqn:clustering_condition} of permanental point processes with fast-decreasing kernels $K$.
	\begin{lemma}
		\label{lem:clustering_PermPP}
		Let $\P$ be a stationary  permanental  point process on $\mR^d$ with a fast-decreasing kernel satisfying
		$K(x,y) \leq \omega(|x-y|)$  where $\omega$ is  at \eqref{fastdk}. With	$s$ as at \eqref{defs} and $n = p + q$, we have
\begin{equation*}
| \rho^{(n)}(x_1,\ldots,x_{p+q}) - \rho^{(p)}(x_1, \ldots,x_{p})\rho^{(q)}(x_{p+1},\ldots,x_{p+q})|  \leq n n! \omega(s) \|K\|^{n-1}.
\end{equation*}
\end{lemma}

To bound the radius of stabilization of geometric functionals on determinantal point processes, we rely on  the following exponential decay of Palm void probabilities. Though the proof is inspired by that of a similar estimate in \cite[Lemma 2]{Miyoshi15}, we derive a more general and explicit bound.

\begin{lemma}
\label{lem:palm-void-DPP}
Let $\P$ be a stationary determinantal point process on $\mR^d$. Then for $p,k \in \N$,
$\x \in (\mR^{d})^p$, and any bounded Borel subset $B \subset \R^d,$ we have
\be
\label{eqn:palm-void-DPP}
\sP^!_{\x}(\P(B) \leq k) \leq e^{(2k+p)/8}e^{-K(\0,\0)\Vol_d(B)/8}.
\ee
\end{lemma}
\noindent {\em Proof.}  
 For any determinantal point process $\P$ (even a non-stationary one), let $\P_{x}$ be the reduced Palm point process with respect to $x \in \mR^d$. From \eqref{eqn:palm_correlation} (see also \cite[Theorem 6.5]{Shirai03}), we
have that $\P_{x}$ is also a determinantal point process and its kernel $L$ is given by
\be
\label{eqn:palm_kernel}
		L(y_1,y_2) = K(y_1,y_2) - \frac{K(y_1,x)K(x,y_2)}{K(x,x)}.
		\ee
Next we assert that

\be \label{simpleineq}
\int_{\mR^d} |K(x,y)|^2 \md y \leq K(x,x),  \ x \in \R^d.
\ee
\remove{To see this, from Campbell's theorem \eqref{disint} and the determinantal property, we derive that for all bounded sets $B$,
\begin{eqnarray*}
\Var(\P(B)) &=&  \int_B \rho^{(1)}(x) \md x - \int_{B \times B} (\rho^{(2)}(x,y) - \rho^{(1)}(x)\rho^{(1)}(y)) \md x \md y \\
& = & \int_{B} K(x,x) \md x - \int_{B \times B} |K(x,y)|^2 \md x \md y.
\end{eqnarray*}
We obtain the desired inequality by the non-negativity of the variance and stationarity of $\P$.}
To see this, write $K(x,y) = \sum_j \lambda_j \phi_j(x) \bar{\phi}_j(y)$, $\lambda_j \in [0,1]$, where $\phi_j, j \geq 1$, is an orthonormal basis for $L^2(\R^d, dx)$ (cf. Lemma 4.2.2 of \cite{HKPV}). In view of $\overline{ {K(x,y)}} = K(y,x)$ we get $\int_{\mR^d} |K(x,y)|^2 \md y = \int_{\mR^d} K(x,y) K(y,x) \md y  = \sum_j \lambda_j^2   \phi_j(x) \bar{\phi}_j(x)  \leq  \sum_j \lambda_j \phi_j(x) \bar{\phi}_j(x) = K(x,x)$, whence the assertion \eqref{simpleineq}.
The bound \eqref{simpleineq}  shows for any bounded Borel subset $B$ and $x \in \R^d$ that
		\begin{eqnarray}
		\sE^!_{x}(\P(B)) & = & \int_B L(y,y) \md y \no = \int_B K(y,y) \md y - \frac{1}{K(x,x)}\int_B |K(x,y)|^2 \md y
		\geq \sE(\P(B)) - 1.\no
		\end{eqnarray}
		Re-iterating the above inequality, we get that for all $\x \in \mR^{dp}$ and any bounded Borel subset $B$

		\be
		\label{eqn:palm-dpp-p}
		\sE^!_{\x}(\P(B)) \geq \sE(\P(B)) - p.
		\ee
Since the point count of a determinantal point process in a given set is a sum of independent Bernoulli random variables  \cite[Theorem 4.5.3]{HKPV},  the Chernoff-Hoeffding bound \cite[Theorem 4.5]{Mitzenmacher05} yields
		\be\label{eqn:C-H}
		\sP^!_{\x}(\P(B) \leq \sE^!_{\x}(\P(B))/2) \leq e^{-\sE^!_{\x}(\P(B))/8}.
		\ee
Now we return to our stationary determinantal point process $\P$ and note that $\E(\P(B)) = K(\0,\0)\Vol_{ d} (B)$.
Suppose first that $B$ is large enough so that $K(\0,\0)\Vol_{d}(B) \geq 2k + p$.
Thus combining \eqref{eqn:palm-dpp-p} and \eqref{eqn:C-H}, we have
		\[ \sP^!_{\x}(\P(B) \leq k) \leq \sP^!_{\x}(\P(B) \leq \sE^!_{\x}(\P(B))/2 ) \leq e^{-(K(\0,\0)\Vol_{ d}(B) - p)/8}. \]
On the other hand, if $B$ is small and satisfies  $K(\0,\0)\Vol_{ d} (B_{r_0}) < 2k + p$, then
		 the right-hand side of \eqref{eqn:palm-void-DPP} is larger than $1$ and hence it is a trivial bound. \qed

\vskip.3cm			
	Inequality \eqref{eqn:palm-dpp-p}  can also be deduced from the stronger coupling result of \cite[Prop. 5.10(iv)]{Pemantle14} for determinantal point processes with a continuous kernel but we have given an elementary proof. Given Ginibre input, we may improve the exponent in the void probability bound \eqref{eqn:palm-void-DPP}.  We believe this result to be of independent interest, as it generalizes \cite[Lemma 6.1]{AY1}, which treats the case $k = 0$.
		
	%
	\begin{lemma}
		\label{lem:palm-void_Ginibre}
		Let $B_r:= B_r(\0) \subset \mR^2$ and $\P$ be the Ginibre point process. Then for $p,k \in \N$ and $\x \in \mR^{2p}$,
		\be
		\label{eqn:void_ineq_gin}
		\sP^!_{\x}(\P(B_r) \leq k) \leq \exp\{p(k+1)r^2\}\sP(\P(B_r) \leq k) \leq kr^{2k}\exp\{(p(k+1)+k)r^2-\frac{1}{4}r^4(1+o(1))\}.
		\ee
\end{lemma}
We remark that stationarity shows the above bound holds for any radius $r$ ball.
	\vskip.3cm
	\noindent{\em Proof.} We shall prove the result for $p = 1$ and use induction to deduce the general case.
	
	Let $\cK_{B_r}$ be the restriction to $B_r$ of the integral operator $\cK$ (generated by kernel $K$) corresponding to Ginibre point process and $\cL_{B_r}$ be the restriction to $B_r$ of the integral operator $\cL$ (generated by kernel $L$) corresponding to the reduced Palm point process (also a determinantal point process). Let $\lam_i, i =1,2,\ldots$ and $\mu_i, i=1,2,\ldots$ be the eigenvalues of $\cK_{B_r}$ and $\cL_{B_r}$ in decreasing order respectively.
	
	Then from \eqref{eqn:palm_kernel} we have that the rank of the operator $\cK_{B_r} - \cL_{B_r}$ is one. Secondly, note that
	$$\sum_i \mu_i = \sE_x(\P(B_r)) = \int_{B_r} L(y,y) \md y \leq \int_{B_r} K(y,y) \md y = \sE(\P(B_r)) = \sum_i \lam_i.$$
	Hence, by a generalisation of Cauchy's interlacing theorem \cite[Theorem 4]{Dancis87} combined with the above inequality, we get the interlacing inequality $\lam_i \geq \mu_i \geq \lam_{i+1}$ for $i =1,2,\ldots$.
	
	Now, fix $\x = x \in \mR^2$. Again by \cite[Theorem 4.5.3]{HKPV}, we have that $P(B_r) \stackrel{d}{=} \sum_i {\rm{Bernoulli}}(\lam_i)$ and under Palm measure, $P(B_r) \stackrel{d}{=} \sum_i {\rm{Bernoulli}}(\mu_i)$ where both the sums involve independent Bernoulli random variables. Independence of the Bernoulli random variables gives
	\begin{eqnarray*}
		\sP_x(\P(B_r) \leq k) & = &  \sum_{J \subset \mathbb{N}, |J| \leq k}\prod_{j \in J}\mu_j\prod_{j \notin J}(1-\mu_j) \, \leq \, \sum_{J \subset \mathbb{N}, |J| \leq k}\prod_{j \in J}\lam_j\prod_{j \notin J}(1-\lam_{j+1}) \\
		& \leq & \sum_{J \subset \mathbb{N}, |J| \leq k}\prod_{j \in J}\lam_j \prod_{j \notin J}(1-\lam_j) \prod_{j-1 \in J \cup \{0\}, j \notin J}(1-\lam_j)^{-1} \\
		& \leq & (1 - \lam_1)^{-k-1}\sum_{J \subset \mathbb{N}, |J| \leq k}\prod_{j \in J}\lam_j\prod_{j \notin J}(1-\lam_j) = (1-\lam_1)^{-k-1}\sP[\P(B_r) \leq k].\\
	\end{eqnarray*}
	The proof of the first inequality in \eqref{eqn:void_ineq_gin} for the case $p =1$ is complete by noting that $\lam_1 = \sP(EXP(1) \leq r^2)$ (see \cite[Theorems 4.7.1 and 4.7.3]{HKPV}), where $EXP(1)$ stands for an exponential random variable with mean $1$. As said before, iteratively the first inequality in \eqref{eqn:void_ineq_gin} can be proven for an arbitrary $p$.  To complete the proof of the second inequality, we  bound $\sP(\P(B_r) \leq k)$ in a manner similar to the proof of \cite[Proposition 7.2.1]{HKPV}.
	
	Let $\P^* := \{R_1^2,R_2^2,\ldots,\} = \{ |X|^2 : X \in \P\}$ be the point process of squared modulii of the Ginibre point process. Then, from \cite[Theorem 4.7.3]{HKPV}, it is known that $R_i^2 \stackrel{d}{=} \Gamma(i,1)$ ($\Gamma(i,1)$ denotes a gamma random variable with parameters $i,1$) and are independently distributed.
There is a constant $\beta \in (0,1)$ such that
	\[ \sP(R_i^2 \geq r^2) \leq e^{-\beta r^2}\sE(e^{\beta R_i^2}) \leq e^{-\beta r^2}(1- \beta)^{-i}, \ i \geq 1. \]
	%
	 For $i < r^2$, the bound is optimal for $\beta = 1 - \frac{i}{r^2}$. For $r$, set $r_* := \ulcorner r^2 \urcorner$, the ceiling of $r^2$. Then,
	\begin{eqnarray*}
		\sP(\P(B_r) \leq k) & = & \sP(\sharp \{i : R_i^2 \leq r^2\} \leq k) \,  \leq  \, \sP(\sharp \{i \leq  r_* : R_i^2 \leq r^2\} \leq k) \no \\
		& \leq  & \sum_{J \subset [r_*], |J| \leq k}\prod_{i \in J}\sP(R_j^2 \leq r^2) \prod_{i \notin J}\sP(R_j^2 > r^2) \no \\
		& \leq &  \sum_{J \subset [r_*], |J| \leq k}\prod_{i \in J}e^{r^2}e^{-\beta r^2}(1-\beta )^{-i} \prod_{i \notin J} e^{-\beta r^2}(1-\beta)^{-i} \, \, \\
		& \leq & kr^{2k}e^{kr^2} \prod_{i=1}^{r_*}e^{-ar^2}(1-\beta)^{-i}  =  kr^{2k}e^{kr^2}e^{-\frac{1}{4}r^4(1+o(1))}, \no
	\end{eqnarray*}
where equality follows by substituting the optimal $\beta$ for each $i$, as in \cite[Section 7.2]{HKPV}. \qed


\subsection{\bf Facts about superposition of independent point processes}	

The following facts on superposition of independent point processes were useful in the applications involving $\alpha$-determinantal point processes, $|\alpha| = 1/m, m \in \N$.
\vskip .3cm

\noindent {\em Proof of Proposition \ref{prop:sum_clustering_pp}.} We shall prove the proposition in the case $m = 2$; the general case follows in the same fashion albeit with
considerably  more notation.  	  Let $x_1,\ldots,x_{p+q}$ be distinct points in $\mR^d$ with $s$ at \eqref{defs} as usual. For a subset $S \subset [p+q]$, we abbreviate $\rho^{|S|}(x_j : j \in S)$ by $\rho(S)$. Using \eqref{eqn:corr_sum_pp} we have that
\begin{eqnarray*}
	& & \rho_0^{(p+q)}([p+q]) = \sum_{S_1 \sqcup S_2 = [p+q]}\rho(S_1)\rho(S_2) = 2\rho([p+q]) + 2\rho([p])\rho([q])
\end{eqnarray*}
\begin{eqnarray*}
	& & +  \sum_{S_1 \sqcup S_2 = [p+q], S_2 \cap [p] = \emptyset,  S_i \neq \emptyset}\rho(S_1)\rho(S_2) + \sum_{S_1 \sqcup S_2 = [p+q], S_1 \cap [p] = \emptyset,  S_i \neq \emptyset}\rho(S_1)\rho(S_2) \\
	& & +  \sum_{S_1 \sqcup S_2 = [p+q], S_2 \cap [q] = \emptyset,  S_i \neq \emptyset}\rho(S_1)\rho(S_2) + \sum_{S_1 \sqcup S_2 = [p+q], S_1 \cap [q] = \emptyset, S_i \neq \emptyset}\rho(S_1)\rho(S_2) \\
	& & + \sum_{S_1 \sqcup S_2 = [p+q], S_i \cap [p] \neq \emptyset, S_i \cap [q] \neq \emptyset}\rho(S_1)\rho(S_2) \\
	&=& 2\rho([p+q]) + 2\rho([p])\rho([q]) + \sum_{S_{21} \sqcup S_{22} = [q], S_{ij} \neq \emptyset}(\rho(S_{21} \cup [p])\rho(S_{22})+\rho(S_{22} \cup [p])\rho(S_{21})) \\
	&  & + \sum_{S_{11} \sqcup S_{12} = [p], S_{ij} \neq \emptyset}(\rho(S_{11} \cup [q])\rho(S_{12}) + \rho(S_{12} \cup [q])\rho(S_{11})) \\
	& & + \sum_{S_{21} \sqcup S_{22} = [q],S_{11} \sqcup S_{12} = [p],S_{ij} \neq \emptyset} \rho(S_{11} \cup S_{21})\rho(S_{12} \cup S_{22}).
\end{eqnarray*}
On the other hand the product of correlation functions is
\begin{eqnarray*}
	& & \rho_0([p])\rho_0([q]) = (\sum_{S_{11} \sqcup S_{12} = [p]}\rho(S_{11})\rho(S_{12}))(\sum_{S_{21} \sqcup S_{22} = [q]}\rho(S_{21})\rho(S_{22})) \\
	&=& (2\rho([p]) +   \sum_{S_{11} \sqcup S_{12} = [p],S_{ij} \neq \emptyset} \rho(S_{11})\rho(S_{12}))  \times (2\rho([q]) +   \sum_{S_{21} \sqcup S_{22} = [q],S_{ij} \neq \emptyset} \rho(S_{21})\rho(S_{22})) \\
	& = & 4\rho([p])\rho([q])  + 2 \sum_{S_{21} \sqcup S_{22} = [q], S_{ij} \neq \emptyset} \rho(S_{21})\rho([p])\rho(S_{22}) \\
	&  & + 2\sum_{S_{11} \sqcup S_{12} = [p], S_{ij} \neq \emptyset}\rho(S_{11})\rho([q])\rho(S_{12}) + \sum_{S_{21} \sqcup S_{22} = [q],S_{11} \sqcup S_{12} = [p],S_{ij} \neq \emptyset} \rho(S_{11})\rho( S_{21})\rho(S_{12})\rho(S_{22}).
\end{eqnarray*}	
Now, we shall match the two summations term-wise and bound the differences using correlation bound \eqref{eqn:corr_bounded} and fast decay of correlations condition \eqref{eqn:clustering_condition}:
\begin{eqnarray*}
	& & |\rho_0([p+q]) - \rho_0([p])\rho_0([q])| \leq 2|\rho([p+q]) - \rho([p])\rho([q])| \\
	& & + \sum_{S_{21} \sqcup S_{22} = [q], S_{ij} \neq \emptyset}|\rho(S_{21} \cup [p])\rho(S_{22}) - \rho(S_{21})\rho([p])\rho(S_{22})| \\
	& & + \sum_{S_{21} \sqcup S_{22} = [q], S_{ij} \neq \emptyset}|\rho(S_{22} \cup [p])\rho(S_{21}) - \rho(S_{21})\rho([p])\rho(S_{22})| \\
	& & + \sum_{S_{11} \sqcup S_{12} = [p], S_{ij} \neq \emptyset}|\rho(S_{11} \cup [q])\rho(S_{12}) - \rho(S_{11})\rho([q])\rho(S_{12})|  \\
	& & + \sum_{S_{11} \sqcup S_{12} = [p], S_{ij} \neq \emptyset}|\rho(S_{12} \cup [q])\rho(S_{11}) - \rho(S_{11})\rho([q])\rho(S_{12})|  \\
	& & + \sum_{S_{21} \sqcup S_{22} = [q],S_{11} \sqcup S_{12} = [p],S_{ij} \neq \emptyset} |\rho(S_{11} \cup S_{21})\rho(S_{12} \cup S_{22}) - \rho(S_{11})\rho( S_{21})\rho(S_{12})\rho(S_{22})| \\
		& \leq & 2\kappa_{p+q}C_{p+q}\phi(c_{p+q}s)\sum_{S_1 \sqcup S_2 = [p+q]}1 \, \, \, \, = \, \, \, \, 2\kappa_{p+q}C_{p+q}\phi(c_{p+q}s)2^{p+q}. \, \, \, \, \, \, \, \square
\end{eqnarray*}
%
%
We now provide void probability bounds for superposition of independent point processes.
\begin{proposition}
\label{prop:void_prob_sum}
Let $\P_1,\ldots,\P_m,  m \in \N$, be independent admissible point processes. For $p,k \in \N$ and a bounded Borel set $B$, set
$$\nu_{p,k}(B) := \sup_{i = 1,\ldots,m}\sup_{0 \leq p' \leq p} \sup_{x_1,...,x_{p'}}\sP_{x_1,...,x_{p'}}(\P_i(B) \leq k).$$
Let $\P := \cup_{i=1}^m \P_i$ be the independent superposition. Then, 
$\alpha^{(k)}$ a.e. $x_1,\ldots,x_p$, we have
$$ \sP_{x_1,\ldots,x_p}(\P(B) \leq k) \leq \nu_{p,k}(B)^m.$$
\end{proposition}
\begin{proof}
We shall show the proposition for $m = 2$ and the general case follows similarly. Further, we use $\rho,\rho_1,\rho_2$ to denote the correlation functions of $\P,\P_1,\P_2$ respectively. Let $A = A_1 \times \ldots \times A_p$ where $A_1,\ldots,A_p$ are disjoint bounded Borel subsets. By setting $f(x_1,\ldots,x_p;\P) = \1[\P(B) \leq k]$ in the refined Campbell theorem in \eqref{disint} and using the independence of $\P_1,\P_2$, we derive that
\begin{align*}
& \, \, \int_{A} \sP_{x_1,\ldots,x_p}(\P(B) \leq k)\rho^{(p)}(x_1,\ldots,x_p) \md x_1 \ldots \md x_p =  \sE(\1[\P(B) \leq k]\P(A_1)\ldots\P(A_p)) \\
& \leq  \sum_{S \subset [p]} \sE\left(\1[\P_1(B) \leq k]\prod_{i \in S}\P_1(A_i)\right) \, \sE\left(\1[\P_2(B) \leq k]\prod_{i \notin S}\P_2(A_i)\right) \\
& = \int_{A} \sum_{S \subset [p]}\sP_{x_i; i \in S}(\P_1(B) \leq k)\rho_1^{(|S|)}(x_i;i \in S) \sP_{x_i; i \notin S}(\P_2(B) \leq k)\rho_2^{(p-|S|)}(x_i;i \notin S)  \md x_1 \ldots \md x_p
\end{align*}
Thus, by definition of $\nu_{p,k}(B)$, we get that for a.e. $x_1,\ldots,x_p$
\begin{align*}
&	\sP_{x_1,\ldots,x_p}(\P(B) \leq k)\rho^{(p)}(x_1,\ldots,x_p) \\
& \leq \sum_{S \subset [p]}\sP_{x_i; i \in S}(\P_1(B) \leq k)\rho_1^{(|S|)}(x_i;i \in S) \sP_{x_i; i \notin S}(\P_2(B) \leq k)\rho_2^{(p-|S|)}(x_i;i \notin S)  \\
	& \leq \nu_{p,k}(B)^2 \sum_{S \subset [p]}\rho_1^{(|S|)}(x_i;i \in S)\rho_2^{(p-|S|)}(x_i;i \notin S) = \nu_{p,k}(B)^2 \rho^{(p)}(x_1,\ldots,x_p),
\end{align*}
where the last equality follows from \eqref{eqn:corr_sum_pp}. The proposition now follows from the above inequality.
\end{proof}
Now, as a trivial corollary of Lemma \ref{lem:palm-void-DPP} and Proposition \ref{prop:void_prob_sum}, we obtain the following useful result.
\begin{corollary}
\label{cor:palm-void-aDPP}
Let $\P$ be a stationary $\alpha$-determinantal point process on $\mR^d$ with $\alpha = -1/m, m \in \N$. Then for $p,k \in \N$,
$\x \in (\mR^{d})^p$, and any bounded Borel subset $B \subset \R^d,$ we have
\be
\label{eqn:palm-void-DPP}
\sP^!_{\x}(\P(B) \leq k) \leq e^{m(2k+p)/8}e^{-K(\0,\0)\Vol_d(B)/8}.
\ee
\end{corollary}
Consider the same assumptions as in Proposition \ref{prop:void_prob_sum}. For a bounded Borel subset $B$, set
$$M_p(B) :=   \sup_{i = 1,\ldots,m}\sup_{0 \leq p' \leq p} \sup_{x_1,...,x_{p'}}\sE_{x_1,...,x_{p'}}(t^{\P_i(B)}), \, \, \, t \geq 0, p \geq 1.$$
Now setting $f(x_1,\ldots,x_p ; \P) = t^{\P(B)}$ in the proof of the proposition, we may deduce that
\begin{equation}
\label{eqn:mom_bd_superp}
\sup_{0 \leq p' \leq p} \sup_{x_1,...,x_{p'}}\sE_{x_1,...,x_{p'}}(t^{\P(B)}) \leq M_p(B)^m.
\end{equation}
\section*{Acknowledgements}
The work benefitted from DY's visits to Lehigh University and IMA, Minneapolis supported in part by the respective institutions. Part of this work was done when DY was a post-doc at Technion, Israel. He is thankful to the institute for its support and to his host Robert Adler for many discussions. 
 The authors thank Manjunath Krishnapur for numerous inputs, especially those related to determinantal point processes and Gaussian analytic functions. The authors thank Jesper M\o ller and G\"{u}nter Last for useful comments on the first draft of this article and Christophe Biscio for discussions related to Brillinger mixing. Finally, this paper benefitted from a thorough reading by an anonymous referee, whose numerous comments improved the exposition.
	

\end{document}